\documentclass[a4paper,12pt]{article}
\usepackage{amsmath}
\usepackage{amssymb,amsthm,graphicx}
\usepackage{enumitem}
\usepackage{color}
\usepackage{epsfig}
\usepackage{graphics}
\usepackage{pdfpages}
\usepackage{subcaption}
\usepackage[font=small]{caption}
\usepackage[hang,flushmargin]{footmisc} 
\usepackage{float}
\usepackage[section]{placeins}
\usepackage{rotating,tabularx}
\usepackage{booktabs}
\usepackage[mathscr]{euscript}
\usepackage{natbib}
\usepackage{setspace}
\usepackage{mathrsfs}
\usepackage[left=2.7cm,right=2.7cm,bottom=2.7cm,top=2.7cm]{geometry}
\newcommand{\reals}{\mathbb{R}}
\newcommand{\integers}{\mathbb{Z}}
\newcommand{\naturals}{\mathbb{N}}

\newcommand{\pr}{\mathbb{P}}        
\newcommand{\ex}{\mathbb{E}}        
\newcommand{\var}{\textnormal{Var}} 
\newcommand{\cov}{\textnormal{Cov}} 

\newcommand{\normal}{N}        

\newcommand{\ind}{1} 
\newcommand{\kernel}{K} 

\theoremstyle{plain}

\newtheorem{theorem}{Theorem}[section]
\newtheorem{prop}{Proposition}[section]

\newtheorem{remark}{Remark}[section]

\newtheorem{propA}{Proposition}
\newtheorem{lemmaA}{Lemma}

\makeatletter
\newcommand{\lefteqno}{\let\veqno\@@leqno}
\makeatother

\newcommand{\heading}[4]
{  \setcounter{page}{1}
   \begin{center}

   \phantom{Distance to upper boundary}
   \vspace{0.5cm}

   {\LARGE \textbf{#1}}\\[0.45cm]
   {\LARGE \textbf{#2}}\\[0.3cm]
   {\LARGE \textbf{#3}}\\[0.3cm]
   {\LARGE \textbf{#4}}
   \end{center}
}

\newcommand{\headingSupplement}[5]
{  \setcounter{page}{1}
   \begin{center}

   \phantom{Distance to upper boundary}

   {\LARGE \textbf{#1}}\\[0.3cm]
   {\LARGE \textbf{#2}}\\[0.45cm]
   {\LARGE \textbf{#3}}\\[0.3cm]
   {\LARGE \textbf{#4}}\\[0.3cm]
   {\LARGE \textbf{#5}}
   \end{center}
}

\newcommand{\authors}[4]
{  \parindent0pt
   \begin{center}
      \begin{minipage}[c][2cm][c]{5.25cm}
      \begin{center} 
      {\large #1} 
      \vspace{0.125cm}
      
      #2 
      \end{center}
      \end{minipage}
      \begin{minipage}[c][2cm][c]{5.25cm}
      \begin{center} 
      {\large #3}
      \vspace{0.125cm}

      #4 
      \end{center}
      \end{minipage}
   \end{center}
}

\begin{document}

\heading{Multiscale Inference}{and Long-Run Variance Estimation}{in Nonparametric Regression}{with Time Series Errors}


\authors{Marina Khismatullina\renewcommand{\thefootnote}{1}\footnotemark[1]}{University of Bonn}{Michael Vogt\renewcommand{\thefootnote}{2}\footnotemark[2]}{University of Bonn} 
\footnotetext[1]{Address: Bonn Graduate School of Economics, University of Bonn, 53113 Bonn, Germany. Email: \texttt{marina.k@uni-bonn.de}.}
\renewcommand{\thefootnote}{2}
\footnotetext[2]{Address: Department of Economics and Hausdorff Center for Mathematics, University of Bonn, 53113 Bonn, Germany. Email: \texttt{michael.vogt@uni-bonn.de}.}
\renewcommand{\thefootnote}{\arabic{footnote}}
\setcounter{footnote}{2}

\vspace{-0.9cm}

\renewcommand{\abstractname}{}
\begin{abstract}
\noindent In this paper, we develop new multiscale methods to test qualitative hypotheses about the function $m$ in the nonparametric regression model $Y_{t,T} = m(t/T) + \varepsilon_t$ with time series errors $\varepsilon_t$. In time series applications, $m$ represents a nonparametric time trend. Practitioners are often interested in whether the trend $m$ has certain shape properties. For example, they would like to know whether $m$ is constant or whether it is increasing/decreasing in certain time regions. Our multiscale methods allow to test for such shape properties of the trend $m$. In order to perform the methods, we require an estimator of the long-run error variance $\sigma^2 = \sum\nolimits_{\ell=-\infty}^{\infty} \cov(\varepsilon_0,\varepsilon_{\ell})$. We propose a new difference-based estimator of $\sigma^2$ for the case that $\{ \varepsilon_t \}$ is an AR($p$) process. In the technical part of the paper, we derive asymptotic theory for the proposed multiscale test and the estimator of the long-run error variance. The theory is complemented by a simulation study and an empirical application to climate data. 
\end{abstract}

\renewcommand{\baselinestretch}{1.2}\normalsize



\renewcommand{\theequation}{\thesection.\arabic{equation}}
\allowdisplaybreaks[1]

\setcounter{equation}{0}
\section{Introduction}\label{sec-intro}

The analysis of time trends is an important aspect of many time series applications. In a wide range of situations, practitioners are particularly interested in certain shape properties of the trend. They raise questions such as the following: Does the observed time series have a trend at all? If so, is the trend increasing/decreasing in certain time regions? Can one identify the regions of increase/decrease? As an example, consider the time series plotted in Figure \ref{temp_data} which shows the yearly mean temperature in Central England from 1659 to 2017. Climatologists are very much interested in learning about the trending behaviour of temperature time series like this; see e.g.\ \cite{Benner1999} and \cite{Rahmstorf2017}. Among other things, they would like to know whether there is an upward trend in the Central England mean temperature towards the end of the sample as visual inspection might suggest.

\begin{figure}
\centering
\includegraphics[width=0.9\textwidth]{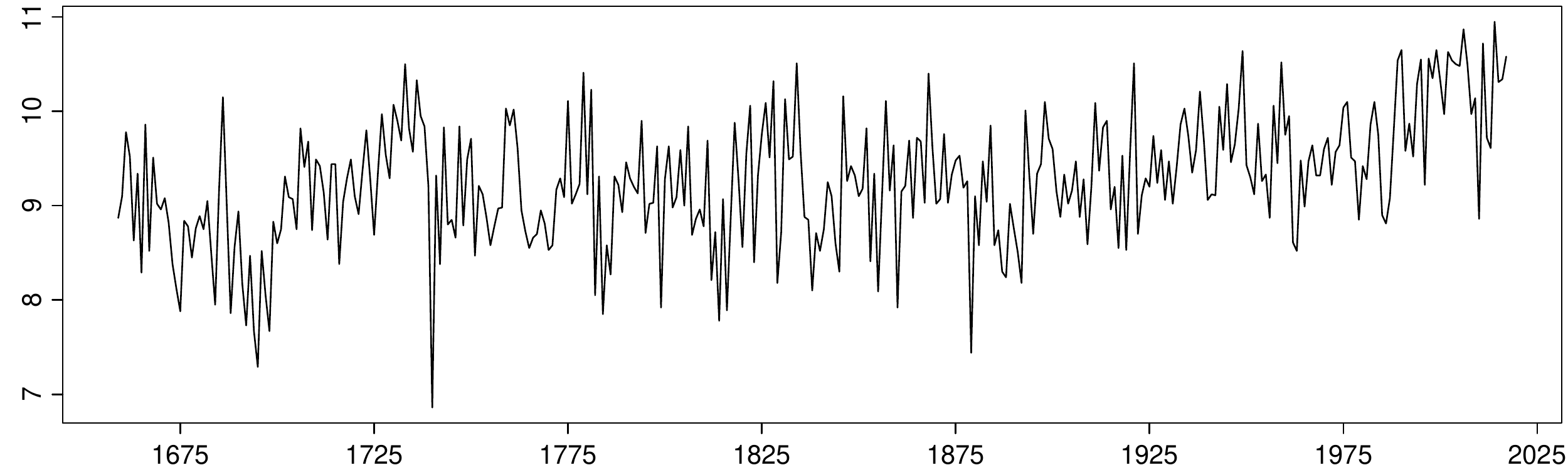}
\vspace{0.15cm}

\caption{Yearly mean temperature in Central England from 1659 to 2017 measured in $^\circ$C.}\label{temp_data}
\end{figure}

In this paper, we develop new methods to test for certain shape properties of a nonparametric time trend. We in particular construct a multiscale test which allows to identify local increases/decreases of the trend function. We develop our test in the context of the following model setting: We observe a time series $\{ Y_{t,T}: 1 \le t \le T \}$ of the form 
\begin{equation}\label{model-intro}
Y_{t,T} = m \Big( \frac{t}{T} \Big) + \varepsilon_t
\end{equation}
for $1 \le t \le T$, where $m: [0,1] \rightarrow \mathbb{R}$ is an unknown nonparametric regression function and the error terms $\varepsilon_t$ form a stationary time series process with $\ex[\varepsilon_t] = 0$. In a time series context, the design points $t/T$ represent the time points of observation and $m$ is a nonparametric time trend. As usual in nonparametric regression, we let the function $m$ depend on rescaled time $t/T$ rather than on real time $t$. A detailed description of model \eqref{model-intro} is provided in Section \ref{sec-model}.

Our multiscale test is developed step by step in Section \ref{sec-method}. Roughly speaking, the procedure can be outlined as follows: Let $H_0(u,h)$ be the hypothesis that $m$ is constant in the time window $[u-h,u+h] \subseteq [0,1]$, where $u$ is the midpoint and $2h$ the size of the window. In a first step, we set up a test statistic $\widehat{s}_T(u,h)$ for the hypothesis $H_0(u,h)$. In a second step, we aggregate the statistics $\widehat{s}_T(u,h)$ for a large number of different time windows $[u-h,u+h]$. We thereby construct a multiscale statistic which allows to test the hypothesis $H_0(u,h)$ simultaneously for many time windows $[u-h,u+h]$. In the technical part of the paper, we derive the theoretical properties of the resulting multiscale test. To do so, we come up with a proof strategy which combines strong approximation results for dependent processes with anti-concentration bounds for Gaussian random vectors. This strategy is of interest in itself and may be applied to other multiscale test problems for dependent data. As shown by our theoretical analysis, our multiscale test is a rigorous level-$\alpha$-test of the overall null hypothesis $H_0$ that $H_0(u,h)$ is simultaneously fulfilled for all time windows $[u-h,u+h]$ under consideration. Moreover, for a given significance level $\alpha \in (0,1)$, the test allows to make simultaneous confidence statements of the following form: We can claim, with statistical confidence $1-\alpha$, that there is an increase/decrease in the trend $m$ on all time windows $[u-h,u+h]$ for which the hypothesis $H_0(u,h)$ is rejected. Hence, the test allows to identify, with a pre-specified statistical confidence, time regions where the trend $m$ is increasing/decreasing.

For independent data, multiscale tests have been developed in a variety of different contexts in recent years. In the regression context, \cite{ChaudhuriMarron1999,ChaudhuriMarron2000} introduced the so-called SiZer method which has been extended in various directions; see e.g.\ \cite{HannigMarron2006} where a refined distribution theory for SiZer is derived. \cite{HallHeckman2000} constructed a multiscale test on monotonicity of a regression function. \cite{DuembgenSpokoiny2001} developed a multiscale approach which works with additively corrected supremum statistics and derived theoretical results in the context of a continuous Gaussian white noise model. Rank-based multiscale tests for nonparametric regression were proposed in \cite{Duembgen2002} and \cite{Rohde2008}. More recently, \cite{ProkschWernerMunk2018} have constructed multiscale tests for inverse regression models. In the context of density estimation, multiscale tests have been investigated in \cite{DuembgenWalther2008}, \cite{RufibachWalther2010}, \cite{SchmidtHieber2013} and \cite{EckleBissantzDette2017} among others.

Whereas a large number of multiscale tests for independent data have been developed in recent years, multiscale tests for dependent data are much rarer. Most notably, there are some extensions of the SiZer approach to a time series context. \cite{Rondonotti2004} and \cite{Rondonotti2007} have introduced SiZer methods for dependent data which can be used to find local increases/decreases of a trend and which may thus be regarded as an alternative to our multiscale test. However, these SiZer methods are mainly designed for data exploration rather than for rigorous statistical inference. Our multiscale method, in contrast, is a rigorous level-$\alpha$-test of the hypo\-thesis $H_0$ which allows to make simultaneous confidence statements about the time regions where the trend $m$ is increasing/decreasing. Some theoretical results for dependent SiZer methods have been derived in \cite{ParkHannigKang2009}, but only under a quite severe restriction: Only time windows $[u-h,u+h]$ with window sizes or scales $h$ are taken into account that remain bounded away from zero as the sample size $T$ grows. Scales $h$ that converge to zero as $T$ increases are excluded. This effectively means that only large time windows $[u-h,u+h]$ are taken into consideration. Our theory, in contrast, allows to simultaneously consider scales $h$ of fixed size and scales $h$ that converge to zero at various different rates. We are thus able to take into account time windows of many different sizes.

Our multiscale approach is also related to Wavelet-based methods: Similar to the latter, it takes into account different locations $u$ and resolution levels or scales $h$ simultaneously. However, while our multiscale approach is designed to test for local increases/decreases of a nonparametric trend, Wavelet methods are commonly used for other purposes. Among other things, they are employed for estimating/reconstructing nonparametric regression curves [see e.g.\ \cite{Donoho1995} or \cite{vonSachsMacGibbon2000}] and for change point detection [see e.g.\ \citet{ChoFryzlewicz2012}].

The test statistic of our multiscale method depends on the long-run error variance $\sigma^2 = \sum\nolimits_{\ell=-\infty}^{\infty} \cov(\varepsilon_0,\varepsilon_{\ell})$, which is usually unknown in practice. To carry out our multiscale test, we thus require an estimator of $\sigma^2$. Indeed, such an estimator is required for virtually all inferential procedures in the context of model \eqref{model-intro}. Hence, the problem of estimating $\sigma^2$ in model \eqref{model-intro} is of broader interest and has received a lot of attention in the literature; see \cite{MuellerStadtmueller1988}, \cite{Herrmann1992} and \cite{Hall2003} among many others. In Section \ref{sec-error-var}, we discuss several estimators of $\sigma^2$ which are valid under different conditions on the error process $\{\varepsilon_t\}$. Most notably, we introduce a new difference-based estimator of $\sigma^2$ for the case that $\{ \varepsilon_t \}$ is an AR($p$) process. This estimator improves on existing methods in several respects.

The methodological and theoretical analysis of the paper is complemented by a simulation study in Section \ref{sec-sim} and an empirical application in Section \ref{sec-data}. In the simulation study, we examine the finite sample properties of our multiscale test and compare it to the dependent SiZer methods introduced in \cite{Rondonotti2004} and \cite{Rondonotti2007}. Moreover, we investigate the small sample performance of our estimator of $\sigma^2$ in the AR($p$) case and compare it to the estimator of \cite{Hall2003}. In Section \ref{sec-data}, we use our methods to analyse the temperature data from Figure \ref{temp_data}.

\setcounter{equation}{0}
\section{The model}\label{sec-model}

We now describe the model setting in detail which was briefly outlined in the Introduction. We observe a time series $\{Y_{t,T}: 1 \le t \le T \}$ of length $T$ which satisfies the nonparametric regression equation 
\begin{equation}\label{model}
Y_{t,T} = m \Big( \frac{t}{T} \Big) + \varepsilon_t 
\end{equation}
for $1 \le t \le T$. Here, $m$ is an unknown nonparametric function defined on $[0,1]$ and $\{ \varepsilon_t: 1 \le t \le T \}$ is a zero-mean stationary error process. For simplicity, we restrict attention to equidistant design points $x_t = t/T$. However, our methods and theory can also be carried over to non-equidistant designs. The stationary error process $\{\varepsilon_t\}$ is assumed to have the following properties: 
\begin{enumerate}[label=(C\arabic*),leftmargin=1.05cm]

\item \label{C-err1} The variables $\varepsilon_t$ allow for the representation $\varepsilon_t = G(\ldots,\eta_{t-1},\eta_t,\eta_{t+1},\ldots)$, where $\eta_t$ are i.i.d.\ random variables and $G: \reals^\integers \rightarrow \reals$ is a measurable function. 

\item \label{C-err2} It holds that $\| \varepsilon_t \|_q < \infty$ for some $q > 4$, where $\| \varepsilon_t \|_q = (\ex|\varepsilon_t|^q)^{1/q}$. 

\end{enumerate}
Following \cite{Wu2005}, we impose conditions on the dependence structure of the error process $\{\varepsilon_t\}$ in terms of the physical dependence measure $d_{t,q} = \| \varepsilon_t - \varepsilon_t^\prime \|_q$, where $\varepsilon_t^\prime = G(\ldots,\eta_{-1},\eta_0^\prime,\eta_1,\ldots,\eta_{t-1},\eta_t,\eta_{t+1},\ldots)$ with $\{\eta_t^\prime\}$ being an i.i.d.\ copy of $\{\eta_t\}$. In particular, we assume the following: 
\begin{enumerate}[label=(C\arabic*),leftmargin=1.05cm]
\setcounter{enumi}{2}

\item \label{C-err3} Define $\Theta_{t,q} = \sum\nolimits_{|s| \ge t} d_{s,q}$ for $t \ge 0$. It holds that 
$\Theta_{t,q} = O ( t^{-\tau_q} (\log t)^{-A} )$,  
where $A > \frac{2}{3} (1/q + 1 + \tau_q)$ and $\tau_q = \{q^2 - 4 + (q-2) \sqrt{q^2 + 20q + 4}\} / 8q$. 

\end{enumerate}
The conditions \ref{C-err1}--\ref{C-err3} are fulfilled by a wide range of stationary processes $\{\varepsilon_t\}$. As a first example, consider linear processes of the form $\varepsilon_t = \sum\nolimits_{i=0}^{\infty} c_i \eta_{t-i}$ with $\| \varepsilon_t \|_q < \infty$, where $c_i$ are absolutely summable coefficients and $\eta_t$ are i.i.d.\ innovations with $\ex[\eta_t] = 0$ and $\| \eta_t\|_q < \infty$. Trivially, \ref{C-err1} and \ref{C-err2} are fulfilled in this case. Moreover, if $|c_i| = O(\rho^i)$ for some $\rho \in (0,1)$, then \ref{C-err3} is easily seen to be satisfied as well. As a special case, consider an ARMA process $\{\varepsilon_t\}$ of the form $\varepsilon_t - \sum\nolimits_{i=1}^p a_i \varepsilon_{t-i} = \eta_t + \sum\nolimits_{j=1}^r b_j \eta_{t-j}$  with $\| \varepsilon_t \|_q < \infty$, where $a_1,\ldots,a_p$ and $b_1,\ldots,b_r$ are real-valued parameters. As before, we let $\eta_t$ be i.i.d.\ innovations with $\ex[\eta_t] = 0$ and $\| \eta_t\|_q < \infty$. Moreover, as usual, we suppose that the complex polynomials $A(z) = 1 - \sum\nolimits_{j=1}^p a_jz^j$ and $B(z) = 1 + \sum\nolimits_{j=1}^r b_jz^j$ do not have any roots in common. If $A(z)$ does not have any roots inside the unit disc, then the ARMA process $\{ \varepsilon_t \}$ is stationary and causal. Specifically, it has the representation $\varepsilon_t = \sum\nolimits_{i=0}^{\infty} c_i \eta_{t-i}$ with $|c_i| = O(\rho^i)$ for some $\rho \in (0,1)$, implying that \ref{C-err1}--\ref{C-err3} are fulfilled. The results in \cite{WuShao2004} show that condition \ref{C-err3} (as well as the other two conditions) is not only fulfilled for linear time series processes but also for a variety of non-linear processes.

\setcounter{equation}{0}
\section{The multiscale test}\label{sec-method}

In this section, we introduce our multiscale method to test for local increases/decreases of the trend function $m$ and analyse its theoretical properties. We assume throughout that $m$ is continuously differentiable on $[0,1]$. The test problem under consideration can be formulated as follows: Let $H_0(u,h)$ be the hypothesis that $m$ is constant on the interval $[u-h,u+h]$. Since $m$ is continuously differentiable, $H_0(u,h)$ can be reformulated as
\[ H_0(u,h): m^\prime(w) = 0 \text { for all } w \in [u-h,u+h], \]
where $m^\prime$ is the first derivative of $m$. We want to test the hypothesis $H_0(u,h)$ not only for a single interval $[u-h,u+h]$ but simultaneously for many different intervals. The overall null hypothesis is thus given by
\[ H_0: \text{ The hypothesis } H_0(u,h) \text{ holds true for all } (u,h) \in \mathcal{G}_T, \]
where $\mathcal{G}_T$ is some large set of points $(u,h)$. The details on the set $\mathcal{G}_T$ are discussed at the end of Section \ref{subsec-method-stat} below. Note that $\mathcal{G}_T$ in general depends on the sample size $T$, implying that the null hypothesis $H_0 = H_{0,T}$ depends on $T$ as well. We thus consider a sequence of null hypotheses $\{H_{0,T}: T = 1,2,\ldots \}$ as $T$ increases. For simplicity of notation, we however suppress the dependence of $H_0$ on $T$. In Sections \ref{subsec-method-stat} and \ref{subsec-method-test}, we step by step construct the multiscale test of the hypothesis $H_0$. The theoretical properties of the test are analysed in Section \ref{subsec-method-theo}.

\subsection{Construction of the multiscale statistic}\label{subsec-method-stat}

We first construct a test statistic for the hypothesis $H_0(u,h)$, where $[u-h,u+h]$ is a given interval. To do so, we consider the kernel average
\begin{equation*}
\widehat{\psi}_T(u,h) = \sum\limits_{t=1}^T w_{t,T}(u,h) Y_{t,T}, 
\end{equation*}
where $w_{t,T}(u,h)$ is a kernel weight and $h$ is the bandwidth. In order to avoid boundary issues, we work with a local linear weighting scheme. We in particular set 
\begin{equation}\label{weights}
w_{t,T}(u,h) = \frac{\Lambda_{t,T}(u,h)}{ \{\sum\nolimits_{t=1}^T \Lambda_{t,T}(u,h)^2 \}^{1/2} }, 
\end{equation}
where
\[ \Lambda_{t,T}(u,h) = K\Big(\frac{\frac{t}{T}-u}{h}\Big) \Big[ S_{T,0}(u,h) \Big(\frac{\frac{t}{T}-u}{h}\Big) - S_{T,1}(u,h) \Big], \]
$S_{T,\ell}(u,h) = (Th)^{-1} \sum\nolimits_{t=1}^T K(\frac{\frac{t}{T}-u}{h}) (\frac{\frac{t}{T}-u}{h})^\ell$ for $\ell = 0,1,2$ and $K$ is a kernel function with the following properties: 
\begin{enumerate}[label=(C\arabic*),leftmargin=1.05cm]
\setcounter{enumi}{3}
\item \label{C-ker} The kernel $K$ is non-negative, symmetric about zero and integrates to one. Moreover, it has compact support $[-1,1]$ and is Lipschitz continuous, that is, $|K(v) - K(w)| \le C |v-w|$ for any $v,w \in \reals$ and some constant $C > 0$. 
\end{enumerate} 
The kernel average $\widehat{\psi}_T(u,h)$ is nothing else than a rescaled local linear estimator of the derivative $m^\prime(u)$ with bandwidth $h$.\footnote{Alternatively to the local linear weights defined in \eqref{weights}, we could also work with the weights $w_{t,T}(u,h) = K^\prime( h^{-1} [u - t/T] )/ \{ \sum\nolimits_{t=1}^T  K^\prime( h^{-1}[u - t/T] )^2 \}^{1/2}$, where the kernel function $K$ is assumed to be differentiable and $K^\prime$ is its derivative. We however prefer to use local linear weights as these have superior theoretical properties at the boundary.}

A test statistic for the hypothesis $H_0(u,h)$ is given by the normalized kernel average $\widehat{\psi}_T(u,h)/\widehat{\sigma}$, where $\widehat{\sigma}^2$ is an estimator of the long-run variance $\sigma^2 = \sum\nolimits_{\ell=-\infty}^{\infty} \cov(\varepsilon_0,\varepsilon_\ell)$ of the error process $\{\varepsilon_t\}$. The problem of estimating $\sigma^2$ is discussed in detail in Section \ref{sec-error-var}. For the time being, we suppose that $\widehat{\sigma}^2$ is an estimator with reasonable theoretical properties. Specifically, we assume that $\widehat{\sigma}^2 = \sigma^2 + o_p(\rho_T)$ with $\rho_T = o(1/\log T)$. This is a fairly weak condition which is in particular satisfied by the estimators of $\sigma^2$ analysed in Section \ref{sec-error-var}. The kernel weights $w_{t,T}(u,h)$ are chosen such that in the case of independent errors $\varepsilon_t$, $\var(\widehat{\psi}_T(u,h)) = \sigma^2$ for any location $u$ and bandwidth $h$, where the long-run error variance $\sigma^2$ simplifies to $\sigma^2 = \var(\varepsilon_t)$. In the more general case that the error terms satisfy the weak dependence conditions from Section \ref{sec-model}, $\var(\widehat{\psi}_T(u,h)) = \sigma^2 + o(1)$ for any $u$ and $h$ under consideration. Hence, for sufficiently large sample sizes $T$, the test statistic $\widehat{\psi}_T(u,h)/\widehat{\sigma}$ has approximately unit variance.

We now combine the test statistics $\widehat{\psi}_T(u,h)/\widehat{\sigma}$ for a wide range of different locations $u$ and bandwidths or scales $h$. There are different ways to do so, leading to different types of multiscale statistics. Our multiscale statistic is defined as
\begin{equation}\label{multiscale-stat}
\widehat{\Psi}_T = \max_{(u,h) \in \mathcal{G}_T} \Big\{ \Big|\frac{\widehat{\psi}_T(u,h)}{\widehat{\sigma}}\Big| - \lambda(h) \Big\}, 
\end{equation} 
where $\lambda(h) = \sqrt{2 \log \{ 1/(2h) \}}$ and $\mathcal{G}_T$ is the set of points $(u,h)$ that are taken into consideration. The details on the set $\mathcal{G}_T$ are given below. As can be seen, the statistic $\widehat{\Psi}_T$ does not simply aggregate the individual statistics $\widehat{\psi}_T(u,h)/\widehat{\sigma}$ by taking the supremum over all points $(u,h) \in \mathcal{G}_T$ as in more traditional multiscale approaches. We rather calibrate the statistics $\widehat{\psi}_T(u,h)/\widehat{\sigma}$ that correspond to the bandwidth $h$ by subtracting the additive correction term $\lambda(h)$. This approach was pioneered by \cite{DuembgenSpokoiny2001} and has been used in numerous other studies since then; see e.g.\ \cite{Duembgen2002}, \cite{Rohde2008}, \cite{DuembgenWalther2008}, \cite{RufibachWalther2010}, \cite{SchmidtHieber2013} and \cite{EckleBissantzDette2017}.

To see the heuristic idea behind the additive correction $\lambda(h)$, consider for a moment the uncorrected statistic
\[ \widehat{\Psi}_{T,\text{uncorrected}} = \max_{(u,h) \in \mathcal{G}_T} \Big|\frac{\widehat{\psi}_T(u,h)}{\widehat{\sigma}}\Big| \]
and suppose that the hypothesis $H_0(u,h)$ is true for all $(u,h) \in \mathcal{G}_T$. For simplicity, assume that the errors $\varepsilon_t$ are i.i.d.\ normally distributed and neglect the estimation error in $\widehat{\sigma}$, that is, set $\widehat{\sigma} = \sigma$. Moreover, suppose that the set $\mathcal{G}_T$ only consists of the points $(u_k,h_\ell) = ((2k - 1)h_\ell,h_\ell)$ with $k = 1,\ldots,\lfloor 1/2h_\ell \rfloor$ and $\ell = 1,\ldots,L$. In this case, we can write
\[ \widehat{\Psi}_{T,\text{uncorrected}} = \max_{1 \le \ell \le L} \max_{1 \le k \le \lfloor 1/2h_\ell \rfloor} \Big|\frac{\widehat{\psi}_T(u_k,h_\ell)}{\sigma}\Big|. \]
Under our simplifying assumptions, the statistics $\widehat{\psi}_T(u_k,h_\ell)/\sigma$ with $k = 1,\ldots,\lfloor 1/2h_\ell \rfloor$ are independent and standard normal for any given bandwidth $h_\ell$. Since the maximum over $\lfloor 1/2h \rfloor$ independent standard normal random variables is $\lambda(h) + o_p(1)$ as $h \rightarrow 0$, we obtain that $\max_{k} \widehat{\psi}_T(u_k,h_\ell)/\sigma$ is approximately of size $\lambda(h_\ell)$ for small bandwidths $h_\ell$. As $\lambda(h) \rightarrow \infty$ for $h \rightarrow 0$, this implies that $\max_{k} \widehat{\psi}_T(u_k,h_\ell)/\sigma$ tends to be much larger in size for small than for large bandwidths $h_\ell$. As a result, the stochastic behaviour of the uncorrected statistic $\widehat{\Psi}_{T,\text{uncorrected}}$ tends to be dominated by the statistics $\widehat{\psi}_T(u_k,h_\ell)$ corresponding to small bandwidths $h_\ell$. The additively corrected statistic $\widehat{\Psi}_T$, in contrast, puts the statistics $\widehat{\psi}_T(u_k,h_\ell)$ corresponding to different bandwidths $h_\ell$ on a more equal footing, thus counteracting the dominance of small bandwidth values.

The multiscale statistic $\widehat{\Psi}_T$ simultaneously takes into account all locations $u$ and bandwidths $h$ with $(u,h) \in \mathcal{G}_T$. Throughout the paper, we suppose that $\mathcal{G}_T$ is some subset of $\mathcal{G}_T^{\text{full}} = \{ (u,h): u = t/T \text{ for some } 1 \le t \le T \text{ and } h \in [h_{\min},h_{\max}] \}$, where $h_{\min}$ and $h_{\max}$ denote some minimal and maximal bandwidth value, respectively. For our theory to work, we require the following conditions to hold:
\begin{enumerate}[label=(C\arabic*),leftmargin=1.05cm]
\setcounter{enumi}{4}

\item \label{C-grid} $|\mathcal{G}_T| = O(T^\theta)$ for some arbitrarily large but fixed constant $\theta > 0$, where $|\mathcal{G}_T|$ denotes the cardinality of $\mathcal{G}_T$. 

\item \label{C-h} $h_{\min} \gg T^{-(1-\frac{2}{q})} \log T$, that is, $h_{\min} / \{ T^{-(1-\frac{2}{q})} \log T \} \rightarrow \infty$ with $q > 4$ defined in \ref{C-err2} and $h_{\max} < 1/2$.

\end{enumerate}
According to \ref{C-grid}, the number of points $(u,h)$ in $\mathcal{G}_T$ should not grow faster than $T^\theta$ for some arbitrarily large but fixed $\theta > 0$. This is a fairly weak restriction as it allows the set $\mathcal{G}_T$ to be extremely large compared to the sample size $T$. For example, we may work with the set 
\begin{align*}
\mathcal{G}_T = \big\{ & (u,h): u = t/T \text{ for some } 1 \le t \le T \text{ and } h \in [h_{\min},h_{\max}] \\ & \text{ with } h = t/T \text{ for some } 1 \le t \le T  \big\},
\end{align*}
which contains more than enough points $(u,h)$ for most practical applications. Condition \ref{C-h} imposes some restrictions on the minimal and maximal bandwidths $h_{\min}$ and $h_{\max}$. These conditions are fairly weak, allowing us to choose the bandwidth window $[h_{\min},h_{\max}]$ extremely large. The lower bound on $h_{\min}$ depends on the parameter $q$ defined in \ref{C-err2} which specifies the number of existing moments for the error terms $\varepsilon_t$. As one can see, we can choose $h_{\min}$ to be of the order $T^{-1/2}$ for any $q > 4$. Hence, we can let $h_{\min}$ converge to $0$ very quickly even if only the first few moments of the error terms $\varepsilon_t$ exist. If all moments exist (i.e.\ $q = \infty$), $h_{\min}$ may converge to $0$ almost as quickly as $T^{-1} \log T$. Furthermore, the maximal bandwidth $h_{\max}$ is not even required to converge to $0$, which implies that we can pick it very large.

\begin{remark}
The above construction of the multiscale statistic can be easily adapted to hypotheses other than $H_0$. To do so, one simply needs to replace the kernel weights $w_{t,T}(u,h)$ defined in \eqref{weights} by appropriate versions which are suited to test the hypothesis of interest. For example, if one wants to test for local convexity/concavity of $m$, one may define the kernel weights $w_{t,T}(u,h)$ such that the kernel average $\widehat{\psi}_T(u,h)$ is a (rescaled) estimator of the second derivative of $m$ at the location $u$ with bandwidth $h$. 
\end{remark}

\subsection{The test procedure}\label{subsec-method-test}

In order to formulate a test for the null hypothesis $H_0$, we still need to specify a critical value. To do so, we define the statistic
\begin{equation}\label{Phi-statistic}
\Phi_T = \max_{(u,h) \in \mathcal{G}_T} \Big\{ \Big|\frac{\phi_T(u,h)}{\sigma}\Big| - \lambda(h) \Big\},
\end{equation} 
where $\phi_T(u,h) = \sum\nolimits_{t=1}^T w_{t,T}(u,h) \, \sigma Z_t$ and $Z_t$ are independent standard normal random variables. The statistic $\Phi_T$ can be regarded as a Gaussian version of the test statistic $\widehat{\Psi}_T$ under the null hypothesis $H_0$. Let $q_T(\alpha)$ be the $(1-\alpha)$-quantile of $\Phi_T$. Importantly, the quantile $q_T(\alpha)$ can be computed by Monte Carlo simulations and can thus be regarded as known. Our multiscale test of the hypothesis $H_0$ is now defined as follows: For a given significance level $\alpha \in (0,1)$, we reject $H_0$ if $\widehat{\Psi}_T > q_T(\alpha)$.

\subsection{Theoretical properties of the test}\label{subsec-method-theo}

In order to examine the theoretical properties of our multiscale test, we introduce the auxiliary multiscale statistic 
\begin{align}
\widehat{\Phi}_T 
 & = \max_{(u,h) \in \mathcal{G}_T} \Big\{ \Big| \frac{\widehat{\phi}_T(u,h)}{\widehat{\sigma}} \Big| - \lambda(h) \Big\} \label{Phi-hat-statistic}
\end{align}
with $\widehat{\phi}_T(u,h) = \widehat{\psi}_T(u,h) - \ex [\widehat{\psi}_T(u,h)] = \sum\nolimits_{t=1}^T w_{t,T}(u,h) \varepsilon_t$. The following result is central to the theoretical analysis of our multiscale test. According to it, the (known) quantile $q_T(\alpha)$ of the Gaussian statistic $\Phi_T$ defined in Section \ref{subsec-method-test} can be used as a proxy for the $(1-\alpha)$-quantile of the multiscale statistic $\widehat{\Phi}_T$.
\begin{theorem}\label{theo-stat}
Let \ref{C-err1}--\ref{C-h} be fulfilled and assume that $\widehat{\sigma}^2 = \sigma^2 + o_p(\rho_T)$ with $\rho_T = o(1/\log T)$. Then 
\[ \pr \big( \widehat{\Phi}_T \le q_T(\alpha) \big) = (1 - \alpha) + o(1). \]
\end{theorem}
A full proof of Theorem \ref{theo-stat} is given in the Supplementary Material. 
We here shortly outline the proof strategy, which splits up into two main steps. 
In the first, we replace the statistic $\widehat{\Phi}_T$ for each $T \ge 1$ by a statistic $\widetilde{\Phi}_T$ with the same distribution as $\widehat{\Phi}_T$ and the property that 
\begin{equation}\label{eq-theo-stat-strategy-step1}
\big| \widetilde{\Phi}_T - \Phi_T \big| = o_p(\delta_T),
\end{equation}
where $\delta_T = o(1)$ and the Gaussian statistic $\Phi_T$ is defined in Section \ref{subsec-method-test}. We thus replace the statistic $\widehat{\Phi}_T$ by an identically distributed version which is close to a Gaussian statistic whose distribution is known. To do so, we make use of strong approximation theory for dependent processes as derived in \cite{BerkesLiuWu2014}. In the second step, we show that 
\begin{equation}\label{eq-theo-stat-strategy-claim}
\sup_{x \in \reals} \big| \pr(\widetilde{\Phi}_T \le x) - \pr(\Phi_T \le x) \big| = o(1), 
\end{equation}
which immediately implies the statement of Theorem \ref{theo-stat}. Importantly, the convergence result \eqref{eq-theo-stat-strategy-step1} is not sufficient for establishing \eqref{eq-theo-stat-strategy-claim}. Put differently, the fact that $\widetilde{\Phi}_T$ can be approximated by $\Phi_T$ in the sense that $\widetilde{\Phi}_T - \Phi_T = o_p(\delta_T)$ does not imply that the distribution of $\widetilde{\Phi}_T$ is close to that of $\Phi_T$ in the sense of \eqref{eq-theo-stat-strategy-claim}. For \eqref{eq-theo-stat-strategy-claim} to hold, we additionally require the distribution of $\Phi_T$ to have some sort of continuity property. Specifically, we prove that 
\begin{equation}\label{eq-theo-stat-strategy-step2}
\sup_{x \in \reals} \pr \big( |\Phi_T - x| \le \delta_T \big) = o(1),
\end{equation}
which says that $\Phi_T$ does not concentrate too strongly in small regions of the form $[x-\delta_T,x+\delta_T]$. The main tool for verifying \eqref{eq-theo-stat-strategy-step2} are anti-concentration results for Gaussian random vectors as derived in \cite{Chernozhukov2015}. The claim \eqref{eq-theo-stat-strategy-claim} can be proven by using \eqref{eq-theo-stat-strategy-step1} together with \eqref{eq-theo-stat-strategy-step2}, which in turn yields Theorem \ref{theo-stat}.

The main idea of our proof strategy is to combine strong approximation theory with anti-concentration bounds for Gaussian random vectors to show that the quantiles of the multiscale statistic $\widehat{\Phi}_T$ can be proxied by those of a Gaussian analogue. This strategy is quite general in nature and may be applied to other multiscale problems for dependent data. Strong approximation theory has also been used to investigate multiscale tests for independent data; see e.g.\ \cite{SchmidtHieber2013}. However, it has not been combined with anti-concentration results to approximate the quantiles of the multiscale statistic. As an alternative to strong approximation theory, \cite{EckleBissantzDette2017} and \cite{ProkschWernerMunk2018} have recently used Gaussian approximation results derived in \cite{Chernozhukov2014, Chernozhukov2017} to analyse multiscale tests for independent data. Even though it might be possible to adapt these techniques to the case of dependent data, this is not trivial at all as part of the technical arguments and the Gaussian approximation tools strongly rely on the assumption of independence.

We now investigate the theoretical properties of our multiscale test with the help of Theorem \ref{theo-stat}. The first result is an immediate consequence of Theorem \ref{theo-stat}. It says that the test has the correct (asymptotic) size. 
\begin{prop}\label{prop-test-1}
Let the conditions of Theorem \ref{theo-stat} be satisfied. Under the null hypothesis $H_0$, it holds that 
\[ \pr \big( \widehat{\Psi}_T \le q_T(\alpha) \big) = (1 - \alpha) + o(1). \]
\end{prop}
The second result characterizes the power of the multiscale test against local alternatives. To formulate it, we consider any sequence of functions $m = m_T$ with the following property: There exists $(u,h) \in \mathcal{G}_T$ with $[u-h,u+h] \subseteq [0,1]$ such that 
\begin{equation}\label{loc-alt}
m_T^\prime(w) \ge c_T \sqrt{\frac{\log T}{Th^3}} \quad \text{for all } w \in [u-h,u+h], 
\end{equation}
where $\{c_T\}$ is any sequence of positive numbers with $c_T \rightarrow \infty$. Alternatively to \eqref{loc-alt}, we may also assume that $-m_T^\prime(w) \ge c_T \sqrt{\log T/(Th^3)}$ for all $w \in [u-h,u+h]$. According to the following result, our test has asymptotic power $1$ against local alternatives of the form \eqref{loc-alt}. 
\begin{prop}\label{prop-test-2}
Let the conditions of Theorem \ref{theo-stat} be satisfied and consider any sequence of functions $m_T$ with the property \eqref{loc-alt}. Then 
\[ \pr \big( \widehat{\Psi}_T \le q_T(\alpha) \big) = o(1). \]
\end{prop}
The proof of Proposition \ref{prop-test-2} can be found in the Supplementary Material. To formulate the next result, we define 
\begin{align*}
\Pi_T^\pm   & = \big\{ I_{u,h} = [u-h,u+h]: (u,h) \in \mathcal{A}_T^\pm \big\} \\
\Pi_T^+ & = \big\{ I_{u,h} = [u-h,u+h]: (u,h) \in \mathcal{A}_T^+ \text{ and } I_{u,h} \subseteq [0,1] \big\} \\
\Pi_T^- & = \big\{ I_{u,h} = [u-h,u+h]: (u,h) \in \mathcal{A}_T^- \text{ and } I_{u,h} \subseteq [0,1] \big\} 
\end{align*}
together with 
\begin{align*}
\mathcal{A}_T^\pm & = \Big\{ (u,h) \in \mathcal{G}_T: \Big|\frac{\widehat{\psi}_T(u,h)}{\widehat{\sigma}}\Big| > q_T(\alpha) + \lambda(h) \Big\} \\ 
\mathcal{A}_T^+  & = \Big\{ (u,h) \in \mathcal{G}_T: \frac{\widehat{\psi}_T(u,h)}{\widehat{\sigma}} > q_T(\alpha) + \lambda(h) \Big\} \\ 
\mathcal{A}_T^-  & = \Big\{ (u,h) \in \mathcal{G}_T: -\frac{\widehat{\psi}_T(u,h)}{\widehat{\sigma}} > q_T(\alpha) + \lambda(h) \Big\}. 
\end{align*}
$\Pi_T^\pm$ is the collection of intervals $I_{u,h} = [u-h,u+h]$ for which the (corrected) test statistic $|\widehat{\psi}_T(u,h)/\widehat{\sigma}| - \lambda(h)$ lies above the critical value $q_T(\alpha)$, that is, for which our multiscale test rejects the hypothesis $H_0(u,h)$. $\Pi_T^+$ and $\Pi_T^-$ can be interpreted analogously but take into account the sign of the statistic $\widehat{\psi}_T(u,h)/\widehat{\sigma}$. With this notation at hand, we consider the events 
\begin{align*}
E_T^\pm & = \Big\{ \forall I_{u,h} \in \Pi_T^\pm: m^\prime(v) \ne 0 \text{ for some } v \in I_{u,h} = [u-h,u+h] \Big\} \\
E_T^+  & = \Big\{ \forall I_{u,h} \in \Pi_T^+: m^\prime(v) > 0 \text{ for some } v \in I_{u,h} = [u-h,u+h] \Big\} \\
E_T^-  & = \Big\{ \forall I_{u,h} \in \Pi_T^-: m^\prime(v) < 0 \text{ for some } v \in I_{u,h} = [u-h,u+h] \Big\}.
\end{align*}
$E_T^\pm$ ($E_T^+$, $E_T^-$) is the event that the function $m$ is non-constant (increasing, decreasing) on all intervals $I_{u,h} \in \Pi_T^\pm$ ($\Pi_T^+$, $\Pi_T^-$). More precisely, $E_T^\pm$ ($E_T^+$, $E_T^-$) is the event that for each interval $I_{u,h} \in \Pi_T^\pm$ ($\Pi_T^+$, $\Pi_T^-$), there is a subset $J_{u,h} \subseteq I_{u,h}$ with $m$ being a non-constant (increasing, decreasing) function on $J_{u,h}$. We can make the following formal statement about the events $E_T^\pm$, $E_T^+$ and $E_T^-$, whose proof is given in the Supplementary Material. 
\begin{prop}\label{prop-test-3}
Let the conditions of Theorem \ref{theo-stat} be fulfilled. Then for $\ell \in \{ \pm,+,-\}$, it holds that
\[ \pr \big( E_T^\ell \big) \ge (1-\alpha) + o(1). \]
\end{prop}
According to Proposition \ref{prop-test-3}, we can make simultaneous confidence statements of the following form: With (asymptotic) probability $\ge (1-\alpha)$, the trend function $m$ is non-constant (increasing, decreasing) on some part of the interval $I_{u,h}$ for all $I_{u,h} \in \Pi_T^\pm$ ($\Pi_T^+$, $\Pi_T^-$). Hence, our multiscale procedure allows to identify, with a pre-specified confidence, time regions where there is an increase/decrease in the time trend $m$.

\begin{remark}
Unlike $\Pi_T^\pm$, the sets $\Pi_T^+$ and $\Pi_T^-$ only contain intervals $I_{u,h} = [u-h,u+h]$ which are subsets of $[0,1]$. We thus exclude points $(u,h) \in \mathcal{A}_T^+$ and $(u,h) \in \mathcal{A}_T^-$ which lie at the boundary, that is, for which $I_{u,h} \nsubseteq [0,1]$. The reason is as follows: Let $(u,h) \in \mathcal{A}_T^+$ with $I_{u,h} \nsubseteq [0,1]$. Our technical arguments allow us to say, with asymptotic confidence $\ge 1 - \alpha$, that $m^\prime(v) \ne 0$ for some $v \in I_{u,h}$. However, we cannot say whether $m^\prime(v) > 0$ or $m^\prime(v) < 0$, that is, we cannot make confidence statements about the sign. Crudely speaking, the problem is that the local linear weights $w_{t,T}(u,h)$ behave quite differently at boundary points $(u,h)$ with $I_{u,h} \nsubseteq [0,1]$. As a consequence, we can include boundary points $(u,h)$ in $\Pi_T^\pm$ but not in $\Pi_T^+$ and $\Pi_T^-$.
\end{remark}

The statement of Proposition \ref{prop-test-3} suggests to graphically present the results of our multiscale test by plotting the intervals $I_{u,h} \in \Pi_T^\ell$ for $\ell \in \{\pm, +,-\}$, that is, by plotting the intervals where (with asymptotic confidence $\ge 1-\alpha$) our test detects a violation of the null hypothesis. The drawback of this graphical presentation is that the number of intervals in $\Pi_T^\ell$ is often quite large. To obtain a better graphical summary of the results, we replace $\Pi_T^\ell$ by a subset $\Pi_T^{\ell,\min}$ which is constructed as follows: As in \cite{Duembgen2002}, we call an interval $I_{u,h} \in \Pi_T^\ell$ minimal if there is no other interval $I_{u^\prime,h^\prime} \in \Pi_T^\ell$ with $I_{u^\prime,h^\prime} \subset I_{u,h}$. Let $\Pi_T^{\ell,\min}$ be the set of all minimal intervals in $\Pi_T^\ell$ for $\ell \in \{\pm, +,-\}$ and define the events
\begin{align*}
E_T^{\pm,\min} & = \Big\{ \forall I_{u,h} \in \Pi_T^{\pm,\min}: m^\prime(v) \ne 0 \text{ for some } v \in I_{u,h} = [u-h,u+h] \Big\} \\
E_T^{+,\min} & = \Big\{ \forall I_{u,h} \in \Pi_T^{+,\min}: m^\prime(v) > 0 \text{ for some } v \in I_{u,h} = [u-h,u+h] \Big\} \\ 
E_T^{-,\min} & = \Big\{ \forall I_{u,h} \in \Pi_T^{-,\min}: m^\prime(v) < 0 \text{ for some } v \in I_{u,h} = [u-h,u+h] \Big\}.  
\end{align*}
It is easily seen that $E_T^\ell = E_T^{\ell,\min}$ for $\ell \in \{\pm, +,-\}$. Hence, by Proposition \ref{prop-test-3}, it holds that 
\[ \pr \big(E_T^{\ell,\min}\big) \ge (1-\alpha) + o(1) \] 
for $\ell \in \{\pm, +,-\}$. This suggests to plot the minimal intervals in $\Pi_T^{\ell,\min}$ rather than the whole collection of intervals $\Pi_T^\ell$ as a graphical summary of the test results. We in particular use this way of presenting the test results in our application in Section \ref{sec-data}.

\setcounter{equation}{0}
\section{Estimation of the long-run error variance}\label{sec-error-var}

In this section, we discuss how to estimate the long-run variance $\sigma^2 = \sum\nolimits_{\ell=-\infty}^{\infty} \cov(\varepsilon_0,\varepsilon_{\ell})$ of the error terms in model \eqref{model}. There are two broad classes of estimators: residual- and difference-based estimators. In residual-based approaches, $\sigma^2$ is estimated from the residuals $\widehat{\varepsilon}_t = Y_{t,T} - \widehat{m}_h(t/T)$, where $\widehat{m}_h$ is a nonparametric estimator of $m$ with the bandwidth or smoothing parameter $h$. Difference-based methods proceed by estimating $\sigma^2$ from the $\ell$-th differences $Y_{t,T} - Y_{t-\ell,T}$ of the observed time series $\{Y_{t,T}\}$ for certain orders $\ell$. In what follows, we focus attention on difference-based methods as these do not involve a nonparametric estimator of the function $m$ and thus do not require to specify a bandwidth $h$ for the estimation of $m$. To simplify notation, we let $\Delta_\ell Z_t = Z_t - Z_{t-\ell}$ denote the $\ell$-th differences of a general time series $\{Z_t\}$ throughout the section.

\subsection{Weakly dependent error processes}

We first consider the case that $\{ \varepsilon_t \}$ is a general stationary error process. We do not impose any time series model such as a moving average (MA) or an autoregressive (AR) model on $\{\varepsilon_t\}$ but only require that $\{\varepsilon_t\}$ satisfies certain weak dependence conditions such as those from Section \ref{sec-model}. These conditions imply that the autocovariances $\gamma_\varepsilon(\ell) = \cov(\varepsilon_0,\varepsilon_{\ell})$ decay to zero at a certain rate as $|\ell| \rightarrow \infty$. For simplicity of exposition, we assume that the decay is exponential, that is, $|\gamma_\varepsilon(\ell)| \le C \rho^{|\ell|}$ for some $C > 0$ and $0 < \rho < 1$. In addition to these weak dependence conditions, we suppose that the trend $m$ is smooth. Specifically, we assume $m$ to be Lipschitz continuous on $[0,1]$, that is, $|m(u) - m(v)| \le C|u-v|$ for all $u,v \in [0,1]$ and some constant $C < \infty$.

Under these conditions, a difference-based estimator of $\sigma^2$ can be obtained as follows: To start with, we construct an estimator of the short-run error variance $\gamma_\varepsilon(0) = \var(\varepsilon_0)$. As $m$ is Lipschitz continuous, it holds that $\Delta_q Y_{t,T} = \Delta_q \varepsilon_t + O(q/T)$. Hence, the differences $\Delta_q Y_{t,T}$ of the observed time series are close to the differences $\Delta_q \varepsilon_t$ of the unobserved error process as long as $q$ is not too large in comparison to $T$. Moreover, since $|\gamma_\varepsilon(q)| \le C \rho^{q}$, we have that $\ex[(\Delta_q \varepsilon_t)^2]/2 = \gamma_\varepsilon(0) - \gamma_\varepsilon(q) = \gamma_\varepsilon(0) + O(\rho^q)$. Taken together, these considerations yield that $\gamma_\varepsilon(0) = \ex[(\Delta_q Y_{t,T})^2]/2 + O(\{q/T\}^2 + \rho^q)$, which motivates to estimate $\gamma_\varepsilon(0)$ by  
\begin{equation}\label{est-g0-general}
\widehat{\gamma}_\varepsilon(0) = \frac{1}{2(T-q)} \sum\limits_{t=q+1}^T (\Delta_q Y_{t,T})^2,
\end{equation}
where we assume that $q = q_T \rightarrow \infty$ with $q_T/\log T \rightarrow \infty$ and $q_T/\sqrt{T} \rightarrow 0$. Estimators of the autocovariances $\gamma_\varepsilon(\ell)$ for $\ell \ne 0$ can be derived by similar considerations. Since $\gamma_\varepsilon(\ell) = \gamma_\varepsilon(0) - \ex[(\Delta_\ell \varepsilon_t)^2]/2 = \gamma_\varepsilon(0) - \ex[(\Delta_\ell Y_{t,T})^2]/2 + O(\{\ell/T\}^2)$, we may in particular define  
\begin{equation}\label{est-gl-general}
\widehat{\gamma}_\varepsilon(\ell) = \widehat{\gamma}_\varepsilon(0) - \frac{1}{2(T-|\ell|)} \sum\limits_{t=|\ell|+1}^T (\Delta_{|\ell|} Y_{t,T} )^2
\end{equation}
for any $\ell \ne 0$. Difference-based estimators of the type \eqref{est-g0-general} and \eqref{est-gl-general} have been used in different contexts in the literature before. Estimators similar to \eqref{est-g0-general} and \eqref{est-gl-general} were analysed, for example, in \cite{MuellerStadtmueller1988} and \cite{Hall2003} in the context of $m$-dependent and autoregressive error terms, respectively. In order to estimate the long-run error variance $\sigma^2$, we may employ HAC-type estimation procedures as discussed in \cite{Andrews1991} or \cite{DeJong2000}. In particular, an estimator of $\sigma^2$ may be defined as 
\begin{equation}\label{est-lrv-general}
\widehat{\sigma}^2 = \sum_{|\ell| \le b_T} W \Big( \frac{\ell}{b_T} \Big) \, \widehat{\gamma}_\varepsilon(\ell), 
\end{equation}
where $W: [-1,1] \rightarrow \reals$ is a kernel (e.g.\ of Bartlett or Parzen type) and $b_T$ is a bandwidth parameter with $b_T \rightarrow \infty$ and $b_T/q_T \rightarrow 0$. The additional bandwidth $b_T$ comes into play because estimating $\sigma^2$ under general weak dependence conditions is a nonparametric problem. In particular, it is equivalent to estimating the (nonparametric) spectral density $f_\varepsilon$ of the process $\{\varepsilon_t\}$ at frequency $0$ (assuming that $f_\varepsilon$ exists).

Estimating the long-run error variance $\sigma^2$ under general weak dependence conditions is a notoriously difficult problem. Estimators of $\sigma^2$ such as $\widehat{\sigma}^2$ from \eqref{est-lrv-general} tend to be quite imprecise and are usually very sensitive to the choice of the smoothing parameter, that is, to $b_T$ in the case of $\widehat{\sigma}^2$ from \eqref{est-lrv-general}. To circumvent this issue in practice, it may be beneficial to impose a time series model on the error process $\{\varepsilon_t\}$. Estimating $\sigma^2$ under the restrictions of such a model may of course create some misspecification bias. However, as long as the model gives a reasonable approximation to the true error process, the produced estimates of $\sigma^2$ can be expected to be fairly reliable even though they are a bit biased. Which time series model is appropriate of course depends on the application at hand. In the sequel, we follow authors such as \cite{Hart1994} and \cite{Hall2003} and impose an autoregressive structure on the error terms $\{\varepsilon_t\}$, which is a very popular error model in many application contexts. We thus do not dwell on the nonparametric estimator $\widehat{\sigma}^2$ from \eqref{est-lrv-general} any further but rather give an in-depth analysis of the case of autoregressive error terms.

\subsection{Autoregressive error processes}\label{subsec-error-var-AR}

Estimators of the long-run error variance $\sigma^2$ in model \eqref{model} have been developed for different kinds of error processes $\{\varepsilon_t\}$. A number of authors have analysed the case of MA($m$) or, more generally, $m$-dependent error terms. Difference-based estimators of $\sigma^2$ for this case were proposed in \cite{MuellerStadtmueller1988}, \cite{Herrmann1992} and \cite{Munk2017} among others. Under the assumption of $m$-dependence, $\gamma_\varepsilon(\ell) = 0$ for all $|\ell| > m$. Even though $m$-dependent time series are a reasonable error model in some applications, the condition that $\gamma_\varepsilon(\ell)$ is exactly equal to $0$ for sufficiently large lags $\ell$ is quite restrictive in many situations. Presumably the most widely used error model in practice is an AR($p$) process. Residual-based methods to estimate $\sigma^2$ in model \eqref{model} with AR($p$) errors can be found for example in \cite{Truong1991}, \cite{ShaoYang2011} and \cite{QiuShaoYang2013}. A difference-based method was proposed in \cite{Hall2003}.

In what follows, we introduce a difference-based estimator of $\sigma^2$ for the AR($p$) case which improves on existing methods in several respects. As in \cite{Hall2003}, we consider the following situation: $\{\varepsilon_t\}$ is a stationary and causal AR($p$) process of the form 
\begin{equation}\label{AR-errors} 
\varepsilon_t = \sum_{j=1}^p a_j \varepsilon_{t-j} + \eta_t, 
\end{equation} 
where $a_1,\ldots,a_p$ are unknown parameters and $\eta_t$ are i.i.d.\ innovations with $\ex[\eta_t] = 0$ and $\ex[\eta_t^2] = \nu^2$. The AR order $p$ is known and $m$ is Lipschitz continuous on $[0,1]$, that is, $|m(u) - m(v)| \le C|u-v|$ for all $u,v \in [0,1]$ and some constant $C < \infty$. Since $\{\varepsilon_t\}$ is causal, the variables $\varepsilon_t$ have an MA($\infty$) representation of the form $\varepsilon_t =  \sum_{k=0}^\infty c_k \eta_{t-k}$. The coefficients $c_k$ can be computed iteratively from the equations 
\begin{equation}\label{c-recursion}
c_k - \sum_{j=1}^p a_j c_{k-j} = b_k 
\end{equation}
for $k = 0,1,2,\ldots$, where $b_0 = 1$, $b_k = 0$ for $k > 0$ and $c_k = 0$ for $ k < 0$. Moreover, the coefficients $c_k$ can be shown to decay exponentially fast to zero as $k \rightarrow \infty$, in particular, $|c_k| \le C \rho^k$ with some $C > 0$ and $0 < \rho < 1$.

Our estimation method relies on the following simple observation: If $\{\varepsilon_t\}$ is an AR($p$) process of the form \eqref{AR-errors}, then the time series $\{ \Delta_q \varepsilon_t \}$ of the differences $\Delta_q \varepsilon_t = \varepsilon_t - \varepsilon_{t-q}$ is an ARMA($p,q$) process of the form 
\begin{equation}\label{AR-diff-errors} 
\Delta_q \varepsilon_t - \sum_{j=1}^p a_j \Delta_q \varepsilon_{t-j} = \eta_t - \eta_{t-q}. 
\end{equation}
As $m$ is Lipschitz, the differences $\Delta_q \varepsilon_t$ of the unobserved error process are close to the differences $\Delta_q Y_{t,T}$ of the observed time series in the sense that 
\begin{equation}\label{diff-Y-eps}
\Delta_q Y_{t,T} = \big[\varepsilon_t  - \varepsilon_{t-q} \big] + \Big[ m \Big(\frac{t}{T}\Big) - m \Big(\frac{t-q}{T}\Big) \Big] = \Delta_q \varepsilon_t + O \Big( \frac{q}{T} \Big).  
\end{equation} 
Taken together, \eqref{AR-diff-errors} and \eqref{diff-Y-eps} imply that the differenced time series $\{ \Delta_q Y_{t,T} \}$ is approximately an ARMA($p,q$) process of the form \eqref{AR-diff-errors}. It is precisely this point which is exploited by our estimation methods.

We first construct an estimator of the parameter vector $\boldsymbol{a} = (a_1,\ldots,a_p)^\top$. For any $q \ge 1$, the ARMA($p,q$) process $\{ \Delta_q \varepsilon_t \}$ satisfies the Yule-Walker equations
\begin{align}
\gamma_q(\ell) - \sum\limits_{j=1}^p a_j \gamma_q(\ell-j) & = -\nu^2 c_{q-\ell} \hspace{-1.5cm} & & \text{for } 1 \le \ell < q+1 \label{diff-eq-1} \\
\gamma_q(\ell) - \sum\limits_{j=1}^p a_j \gamma_q(\ell-j) & = 0 \hspace{-1.5cm} & & \text{for } \ell \ge q+1, \label{diff-eq-2}  
\end{align}
where $\gamma_q(\ell) = \cov(\Delta_q \varepsilon_t,$ $\Delta_q \varepsilon_{t-\ell})$ and $c_k$ are the coefficients from the MA($\infty$) expansion of $\{ \varepsilon_t \}$. From \eqref{diff-eq-1} and \eqref{diff-eq-2}, we get that 
\begin{equation}\label{YW-eq} 
\boldsymbol{\Gamma}_q \boldsymbol{a} = \boldsymbol{\gamma}_q + \nu^2 \boldsymbol{c}_q,  
\end{equation} 
where $\boldsymbol{c}_q = (c_{q-1},\dots,c_{q-p})^\top$, $\boldsymbol{\gamma}_q = (\gamma_q(1),\dots,\gamma_q(p))^\top$ and $\boldsymbol{\Gamma}_q$ denotes the $p \times p$ covariance matrix $\boldsymbol{\Gamma}_q = (\gamma_q(i-j): 1 \le i,j \le p)$. Since the coefficients $c_k$ decay exponentially fast to zero, $\boldsymbol{c}_q \approx \boldsymbol{0}$ and thus $\boldsymbol{\Gamma}_q \boldsymbol{a} \approx \boldsymbol{\gamma}_q$ for large values of $q$. This suggests to estimate $\boldsymbol{a}$ by 
\begin{equation}\label{est-AR-FS}
\widetilde{\boldsymbol{a}}_q = \widehat{\boldsymbol{\Gamma}}_q^{-1} \widehat{\boldsymbol{\gamma}}_q, 
\end{equation}
where $\widehat{\boldsymbol{\Gamma}}_q$ and $\widehat{\boldsymbol{\gamma}}_q$ are defined analogously as $\boldsymbol{\Gamma}_q$ and $\boldsymbol{\gamma}_q$ with $\gamma_q(\ell)$ replaced by the sample autocovariances $\widehat{\gamma}_q(\ell) = (T-q)^{-1} \sum_{t=q+\ell+1}^T \Delta_q Y_{t,T} \Delta_q Y_{t-\ell,T}$ and $q = q_T$ goes to infinity sufficiently fast as $T \rightarrow \infty$, specifically, $q = q_T \rightarrow \infty$ with $q_T / \log T \rightarrow \infty$ and $q_T/\sqrt{T} \rightarrow 0$.

The estimator $\widetilde{\boldsymbol{a}}_q$ depends on the tuning parameter $q$, which is very similar in nature to the two tuning parameters of the methods in \cite{Hall2003}. An appropriate choice of $q$ needs to take care of the following two points: 
(i) $q$ should be chosen large enough to ensure that the vector $\boldsymbol{c}_q = (c_{q-1},\dots,c_{q-p})^\top$ is close to zero. As we have already seen, the constants $c_k$ decay exponentially fast to zero and can be computed from the recursive equations \eqref{c-recursion} for given AR parameters $a_1,\ldots,a_p$. In the AR($1$) case, for example, one can readily calculate that $c_k \le 0.0035$ for any $k \ge 20$ and any $|a_1| \le 0.75$. Hence, if we have an AR($1$) model for the errors $\varepsilon_t$ and the error process is not too persistent, choosing $q$ such that $q \ge 20$ should make sure that $\boldsymbol{c}_q$ is close to zero. Generally speaking, the recursive equations \eqref{c-recursion} can be used to get some idea for which values of $q$ the vector $\boldsymbol{c}_q$ can be expected to be approximately zero. 
(ii) $q$ should not be chosen too large in order to ensure that the trend $m$ is appropriately eliminated by taking $q$-th differences. As long as the trend $m$ is not very strong, the two requirements (i) and (ii) can be fulfilled without much difficulty. For example, by choosing $q = 20$ in the AR($1$) case just discussed, we do not only take care of (i) but also make sure that moderate trends $m$ are differenced out appropriately.

When the trend $m$ is very pronounced, in contrast, even moderate values of $q$ may be too large to eliminate the trend appropriately. As a result, the estimator $\widetilde{\boldsymbol{a}}_q$ will have a strong bias. In order to reduce this bias, we refine our estimation procedure as follows: By solving the recursive equations \eqref{c-recursion} with $\boldsymbol{a}$ replaced by $\widetilde{\boldsymbol{a}}_q$, we can compute estimators $\widetilde{c}_k$ of the coefficients $c_k$ and thus estimators $\widetilde{\boldsymbol{c}}_r$ of the vectors $\boldsymbol{c}_r$ for any $r \ge 1$. Moreover, the innovation variance $\nu^2$ can be estimated by $\widetilde{\nu}^2 = (2T)^{-1} \sum_{t=p+1}^T \widetilde{r}_{t,T}^2$, where $\widetilde{r}_{t,T} = \Delta_1 Y_{t,T} - \sum_{j=1}^p \widetilde{a}_j \Delta_1 Y_{t-j,T}$ and $\widetilde{a}_j$ is the $j$-th entry of the vector $\widetilde{\boldsymbol{a}}_q$. Plugging the expressions $\widehat{\boldsymbol{\Gamma}}_r$, $\widehat{\boldsymbol{\gamma}}_r$, $\widetilde{\boldsymbol{c}}_r$ and $\widetilde{\nu}^2$ into \eqref{YW-eq}, we can estimate $\boldsymbol{a}$ by 
\begin{equation}\label{est-AR-SS} 
\widehat{\boldsymbol{a}}_r = \widehat{\boldsymbol{\Gamma}}_r^{-1} (\widehat{\boldsymbol{\gamma}}_r + \widetilde{\nu}^2 \widetilde{\boldsymbol{c}}_r),
\end{equation} 
where $r$ is any fixed number with $r \ge 1$. In particular, unlike $q$, the parameter $r$ does not diverge to infinity but remains fixed as the sample size $T$ increases. As one can see, the estimator $\widehat{\boldsymbol{a}}_r$ is based on differences of some small order $r$; only the pilot estimator $\widetilde{\boldsymbol{a}}_q$ relies on differences of a larger order $q$. As a consequence, $\widehat{\boldsymbol{a}}_r$ should eliminate the trend $m$ more appropriately and should thus be less biased than the pilot estimator $\widetilde{\boldsymbol{a}}_q$. In order to make the method more robust against estimation errors in $\widetilde{\boldsymbol{c}}_r$, we finally average the estimators $\widehat{\boldsymbol{a}}_r$ for a few small values of $r$. In particular, we define  
\begin{equation}\label{est-AR}
\widehat{\boldsymbol{a}} = \frac{1}{\overline{r}} \sum\limits_{r=1}^{\overline{r}} \widehat{\boldsymbol{a}}_r, 
\end{equation}
where $\overline{r}$ is a small natural number. For ease of notation, we suppress the dependence of $\widehat{\boldsymbol{a}}$ on the parameter $\overline{r}$. Once $\widehat{\boldsymbol{a}} =(\widehat{a}_1,\ldots,\widehat{a}_p)^\top$ is computed, the long-run variance $\sigma^2$ can be estimated by 
\begin{equation} \label{est-lrv}
\widehat{\sigma}^2 = \frac{\widehat{\nu}^2}{(1 - \sum_{j=1}^p \widehat{a}_j)^2}, 
\end{equation}
where $\widehat{\nu}^2 = (2T)^{-1} \sum_{t=p+1}^T \widehat{r}_{t,T}^2$ with $\widehat{r}_{t,T} = \Delta_1 Y_{t,T} - \sum_{j=1}^p \widehat{a}_j \Delta_1 Y_{t-j,T}$ is an estimator of the innovation variance $\nu^2$ and we make use of the fact that $\sigma^2 = \nu^2 / (1 - \sum_{j=1}^p a_j)^2$ for the AR($p$) process $\{\varepsilon_t\}$.

We briefly compare the estimator $\widehat{\boldsymbol{a}}$ to competing methods. Presumably closest to our approach is the procedure of \cite{Hall2003}. Nevertheless, the two approaches differ in several respects. The two main advantages of our method are as follows: 
\begin{enumerate}[label=(\alph*),leftmargin=0.7cm]
\item Our estimator produces accurate estimation results even when the AR process $\{\varepsilon_t\}$ is quite persistent, that is, even when the AR polynomial $A(z) = 1 - \sum_{j=1}^p a_j z^j$ has a root close to the unit circle. The estimator of \cite{Hall2003}, in contrast, may have very high variance and may thus produce unreliable results when the AR polynomial $A(z)$ is close to having a unit root. This difference in behaviour can be explained as follows: Our pilot estimator $\widetilde{\boldsymbol{a}}_q = (\widetilde{a}_1,\ldots,\widetilde{a}_p)^\top$ has the property that the estimated AR polynomial $\widetilde{A}(z) = 1 - \sum_{j=1}^p \widetilde{a}_j z^j$ has no root inside the unit disc, that is, $\widetilde{A}(z) \ne 0$ for all complex numbers $z$ with $|z| \le 1$.\footnote{More precisely, $\widetilde{A}(z) \ne 0$ for all $z$ with $|z| \le 1$, whenever the covariance matrix $(\widehat{\gamma}_q(i-j): 1 \le i,j \le p+1)$ is non-singular. Moreover, $(\widehat{\gamma}_q(i-j): 1 \le i,j \le p+1)$ is non-singular whenever $\widehat{\gamma}_q(0) > 0$, which is the generic case.} Hence, the fitted AR model with the coefficients $\widetilde{\boldsymbol{a}}_q$ is ensured to be stationary and causal. Even though this may seem to be a minor technical detail, it has a huge effect on the performance of the estimator: It keeps the estimator stable even when the AR process is very persistent and the AR polynomial $A(z)$ has almost a unit root. This in turn results in a reliable behaviour of the estimator $\widehat{\boldsymbol{a}}$ in the case of high persistence. The estimator of \cite{Hall2003}, in contrast, may produce non-causal results when the AR polynomial $A(z)$ is close to having a unit root. As a consequence, it may have unnecessarily high variance in the case of high persistence. We illustrate this difference between the estimators by the simulation exercises in Section \ref{subsec-sim-3}. A striking example is Figure \ref{fig:hist_scenario1}, which presents the simulation results for the case of an AR($1$) process $\varepsilon_t = a_1 \varepsilon_{t-1} + \eta_t$ with $a_1 = -0.95$ and clearly shows the much better performance of our method.  
\item Both our pilot estimator $\widetilde{\boldsymbol{a}}_q$ and the estimator of \cite{Hall2003} tend to have a substantial bias when the trend $m$ is pronounced. Our estimator $\widehat{\boldsymbol{a}}$ reduces this bias considerably as demonstrated in the simulations of Section \ref{subsec-sim-3}. Unlike the estimator of \cite{Hall2003}, it thus produces accurate results even in the presence of a very strong trend. 
\end{enumerate}

We now derive some basic asymptotic properties of the estimators $\widetilde{\boldsymbol{a}}_q$, $\widehat{\boldsymbol{a}}$ and $\widehat{\sigma}^2$. The following proposition shows that they are $\sqrt{T}$-consistent. 
\begin{prop}\label{prop-lrv}
Let $\{\varepsilon_t\}$ be a causal AR($p$) process of the form \eqref{AR-errors}. Suppose that the innovations $\eta_t$ have a finite fourth moment and let $m$ be Lipschitz continuous. If $q \rightarrow \infty$ with $q/\log T \rightarrow \infty$ and $q/\sqrt{T} \rightarrow 0$, then $\widetilde{\boldsymbol{a}}_q - \boldsymbol{a} = O_p(T^{-1/2})$ as well as $\widehat{\boldsymbol{a}} - \boldsymbol{a} = O_p(T^{-1/2})$ and $\widehat{\sigma}^2 - \sigma^2 = O_p(T^{-1/2})$.
\end{prop}
It can also be shown that $\widetilde{\boldsymbol{a}}_q$, $\widehat{\boldsymbol{a}}$ and $\widehat{\sigma}^2$ are asymptotically normal. In general, their asymptotic variance is somewhat larger than that of the estimators in \cite{Hall2003}. They are thus a bit less efficient in terms of asymptotic variance. However, this theoretical loss of efficiency is more than compensated by the advantages discussed in (a) and (b) above, which lead to a substantially better small sample performance as demonstrated in the simulations of Section \ref{subsec-sim-3}.

\setcounter{equation}{0}
\section{Simulations}\label{sec-sim}

To assess the finite sample performance of our methods, we conduct a number of simulations. In Sections \ref{subsec-sim-1} and \ref{subsec-sim-2}, we investigate the performance of our multiscale test and compare it to the SiZer methods for time series developed in \cite{Rondonotti2004}, \cite{Rondonotti2007} and \cite{ParkHannigKang2009}. In Section \ref{subsec-sim-3}, we analyse the finite sample properties of our long-run variance estimator from Section \ref{subsec-error-var-AR} and compare it to the estimator of \cite{Hall2003}.

\subsection{Size and power properties of the multiscale test}\label{subsec-sim-1}

Our simulation design mimics the situation in the application example of Section \ref{sec-data}. We generate data from the model $Y_{t,T} = m(t/T) + \varepsilon_t$ for different trend functions $m$, error processes $\{\varepsilon_t\}$ and time series lengths $T$. The error terms are supposed to have the AR($1$) structure $\varepsilon_t = a_1 \varepsilon_{t-1} + \eta_t$, where $a_1 \in \{-0.5,-0.25,0.25,0.5\}$ and $\eta_t$ are i.i.d.\ standard normal. In addition, we consider the AR($2$) specification $\varepsilon_t = a_1 \varepsilon_{t-1} + a_2 \varepsilon_{t-2} + \eta_t$, where $\eta_t$ are normally distributed with $\ex[\eta_t] = 0$ and $\ex[\eta_t^2] = \nu^2$. We set $a_1 = 0.167$, $a_2 = 0.178$ and $\nu^2 = 0.322$, thus matching the estimated values obtained in the application of Section \ref{sec-data}. To simulate data under the null hypothesis, we let $m$ be a constant function. In particular, we set $m = 0$ without loss of generality. To generate data under the alternative, we consider the trend functions $m(u) = \beta (u - 0.5) \cdot \ind(0.5 \le u \le 1)$ with $\beta = 1.5,2.0,2.5$. These functions are broken lines with a kink at $u = 0.5$ and different slopes $\beta$. Their shape roughly resembles the trend estimates in the application of Section \ref{sec-data}. The slope parameter $\beta$ corresponds to a trend with the value $m(1) = 0.5 \beta$ at the right endpoint $u = 1$. We thus consider broken lines with the values $m(1) = 0.75, 1.0, 1.25$. Inspecting the middle panel of Figure \ref{plot-results-app1}, the broken lines with the endpoints $m(1) = 1.0$ and $m(1) = 1.25$ (that is, with $\beta = 2.0$ and $\beta = 2.5$) can be seen to resemble the local linear trend estimates in the real-data example the most (where we neglect the nonlinearities of the local linear fits at the beginning of the observation period). The broken line with $\beta = 1.5$ is closer to the null, making it harder for our test to detect this alternative.\footnote{The broken lines $m$ are obviously non-differentiable at the kink point. We could replace them by slightly smoothed versions to satisfy the differentiability assumption that is imposed in the theoretical part of the paper. However, as this leaves the simulation results essentially unchanged but only creates additional notation, we stick to the broken lines.}

\begin{sidewaystable}
\centering
\footnotesize{
\caption{Size of our multiscale test for different AR parameters $a_1$ and $a_2$, sample sizes $T$ and nominal sizes $\alpha$.}\label{tab:size_test}
\newcolumntype{C}[1]{>{\hsize=#1\hsize\centering\arraybackslash}X}
\newcolumntype{Z}{>{\centering\arraybackslash}X}
\begin{tabularx}{\textwidth}{C{2} C{0.1} ZZZ C{0.1} ZZZ C{0.1} ZZZ C{0.1} ZZZ C{0.1} ZZZ} 
\toprule
 & &  \multicolumn{3}{c}{$a_1 = -0.5$} & &  \multicolumn{3}{c}{$a_1 = -0.25$} & &  \multicolumn{3}{c}{$a_1 = 0.25$} & &  \multicolumn{3}{c}{$a_1 = 0.5$} & &  \multicolumn{3}{c}{$(a_1,a_2) = (0.167,0.178)$} \\
\cmidrule[0.4pt]{3-5} \cmidrule[0.4pt]{7-9} \cmidrule[0.4pt]{11-13} \cmidrule[0.4pt]{15-17} \cmidrule[0.4pt]{19-21}
 & &  \multicolumn{3}{c}{nominal size $\alpha$} & &  \multicolumn{3}{c}{nominal size $\alpha$} & &  \multicolumn{3}{c}{nominal size $\alpha$} & &  \multicolumn{3}{c}{nominal size $\alpha$} & &  \multicolumn{3}{c}{nominal size $\alpha$} \\
 & &  0.01 & 0.05  & 0.1 & &  0.01 & 0.05  & 0.1  & &  0.01 & 0.05  & 0.1  & &  0.01 & 0.05  & 0.1  & &  0.01 & 0.05  & 0.1   \\
\cmidrule[0.4pt]{1-21}
$T = 250$ &  & 0.015 & 0.050 & 0.127 &  & 0.014 & 0.057 & 0.120 &  & 0.011 & 0.046 & 0.116 &  & 0.013 & 0.042 & 0.108 &  & 0.011 & 0.052 & 0.117 \\ 
 $T= 350$ &  & 0.009 & 0.067 & 0.120 &  & 0.010 & 0.055 & 0.095 &  & 0.009 & 0.055 & 0.096 &  & 0.010 & 0.049 & 0.090 &  & 0.010 & 0.059 & 0.114 \\ 
 $T= 500$ &  & 0.015 & 0.053 & 0.128 &  & 0.015 & 0.047 & 0.100 &  & 0.018 & 0.048 & 0.101 &  & 0.015 & 0.042 & 0.106 &  & 0.015 & 0.056 & 0.107 \\ 
\bottomrule
\end{tabularx}
\vspace{0.5cm}

\caption{Power of our multiscale test for different AR parameters $a_1$ and $a_2$, sample sizes $T$ and nominal sizes $\alpha$. The three panels (a)--(c) corresponds to different slope parameters $\beta$ of the broken line $m$.}\label{tab:power_test}

\begin{tabularx}{\textwidth}{C{2} C{0.1} ZZZ C{0.1} ZZZ C{0.1} ZZZ C{0.1} ZZZ C{0.1} ZZZ} 
\multicolumn{21}{c}{(a) $\beta = 1.5$} \\[0.2cm]
\toprule
 & &  \multicolumn{3}{c}{$a_1 = -0.5$} & &  \multicolumn{3}{c}{$a_1 = -0.25$} & &  \multicolumn{3}{c}{$a_1 = 0.25$} & &  \multicolumn{3}{c}{$a_1 = 0.5$} & &  \multicolumn{3}{c}{$(a_1,a_2) = (0.167,0.178)$} \\
\cmidrule[0.4pt]{3-5} \cmidrule[0.4pt]{7-9} \cmidrule[0.4pt]{11-13} \cmidrule[0.4pt]{15-17} \cmidrule[0.4pt]{19-21}
 & &  \multicolumn{3}{c}{nominal size $\alpha$} & &  \multicolumn{3}{c}{nominal size $\alpha$} & &  \multicolumn{3}{c}{nominal size $\alpha$} & &  \multicolumn{3}{c}{nominal size $\alpha$} & &  \multicolumn{3}{c}{nominal size $\alpha$} \\
 & &  0.01 & 0.05  & 0.1   & &  0.01 & 0.05  & 0.1   & &  0.01 & 0.05  & 0.1    & &  0.01 & 0.05  & 0.1    & &  0.01 & 0.05  & 0.1   \\
\cmidrule[0.4pt]{1-21}
$T=250$ &  & 0.484 & 0.726 & 0.853 &  & 0.319 & 0.548 & 0.702 &  & 0.077 & 0.177 & 0.324 &  & 0.036 & 0.097 & 0.181 &  & 0.269 & 0.460 & 0.612 \\ 
 $T= 350$ &  & 0.735 & 0.913 & 0.955 &  & 0.463 & 0.753 & 0.834 &  & 0.116 & 0.273 & 0.385 &  & 0.050 & 0.141 & 0.221 &  & 0.390 & 0.654 & 0.770 \\ 
  $T=500$ &  & 0.945 & 0.988 & 0.997 &  & 0.775 & 0.925 & 0.972 &  & 0.195 & 0.389 & 0.551 &  & 0.060 & 0.162 & 0.285 &  & 0.623 & 0.815 & 0.907 \\ 
\bottomrule
\end{tabularx}
\vspace{0.25cm}

\begin{tabularx}{\textwidth}{C{2} C{0.1} ZZZ C{0.1} ZZZ C{0.1} ZZZ C{0.1} ZZZ C{0.1} ZZZ} 
\multicolumn{21}{c}{(b) $\beta = 2.0$} \\[0.2cm]
\toprule
 & &  \multicolumn{3}{c}{$a_1 = -0.5$} & &  \multicolumn{3}{c}{$a_1 = -0.25$} & &  \multicolumn{3}{c}{$a_1 = 0.25$} & &  \multicolumn{3}{c}{$a_1 = 0.5$} & &  \multicolumn{3}{c}{$(a_1,a_2) = (0.167,0.178)$} \\
\cmidrule[0.4pt]{3-5} \cmidrule[0.4pt]{7-9} \cmidrule[0.4pt]{11-13} \cmidrule[0.4pt]{15-17} \cmidrule[0.4pt]{19-21}
 & &  \multicolumn{3}{c}{nominal size $\alpha$} & &  \multicolumn{3}{c}{nominal size $\alpha$} & &  \multicolumn{3}{c}{nominal size $\alpha$} & &  \multicolumn{3}{c}{nominal size $\alpha$} & &  \multicolumn{3}{c}{nominal size $\alpha$} \\
 & &  0.01 & 0.05  & 0.1   & &  0.01 & 0.05  & 0.1   & &  0.01 & 0.05  & 0.1    & &  0.01 & 0.05  & 0.1    & &  0.01 & 0.05  & 0.1   \\
\cmidrule[0.4pt]{1-21}
$T = 250$ &  & 0.869 & 0.961 & 0.985 &  & 0.663 & 0.846 & 0.916 &  & 0.164 & 0.340 & 0.520 &  & 0.062 & 0.143 & 0.259 &  & 0.549 & 0.724 & 0.851 \\ 
$T=  350$ &  & 0.979 & 0.997 & 1.000 &  & 0.863 & 0.969 & 0.986 &  & 0.262 & 0.483 & 0.615 &  & 0.092 & 0.231 & 0.334 &  & 0.759 & 0.922 & 0.958 \\ 
 $T= 500$ &  & 1.000 & 1.000 & 1.000 &  & 0.983 & 0.997 & 0.999 &  & 0.469 & 0.716 & 0.821 &  & 0.137 & 0.309 & 0.451 &  & 0.933 & 0.983 & 0.994 \\ 
\bottomrule
\end{tabularx}
\vspace{0.25cm}

\begin{tabularx}{\textwidth}{C{2} C{0.1} ZZZ C{0.1} ZZZ C{0.1} ZZZ C{0.1} ZZZ C{0.1} ZZZ} 
\multicolumn{21}{c}{(c) $\beta = 2.5$} \\[0.2cm]
\toprule
 & &  \multicolumn{3}{c}{$a_1 = -0.5$} & &  \multicolumn{3}{c}{$a_1 = -0.25$} & &  \multicolumn{3}{c}{$a_1 = 0.25$} & &  \multicolumn{3}{c}{$a_1 = 0.5$} & &  \multicolumn{3}{c}{$(a_1,a_2) = (0.167,0.178)$} \\
\cmidrule[0.4pt]{3-5} \cmidrule[0.4pt]{7-9} \cmidrule[0.4pt]{11-13} \cmidrule[0.4pt]{15-17} \cmidrule[0.4pt]{19-21}
 & &  \multicolumn{3}{c}{nominal size $\alpha$} & &  \multicolumn{3}{c}{nominal size $\alpha$} & &  \multicolumn{3}{c}{nominal size $\alpha$} & &  \multicolumn{3}{c}{nominal size $\alpha$} & &  \multicolumn{3}{c}{nominal size $\alpha$} \\
 & &  0.01 & 0.05  & 0.1   & &  0.01 & 0.05  & 0.1   & &  0.01 & 0.05  & 0.1    & &  0.01 & 0.05  & 0.1    & &  0.01 & 0.05  & 0.1   \\
\cmidrule[0.4pt]{1-21}
$T=250$ &  & 0.989 & 1.000 & 1.000 &  & 0.901 & 0.971 & 0.993 &  & 0.322 & 0.543 & 0.703 &  & 0.100 & 0.224 & 0.367 &  & 0.804 & 0.918 & 0.958 \\ 
 $T= 350$ &  & 1.000 & 1.000 & 1.000 &  & 0.990 & 1.000 & 1.000 &  & 0.470 & 0.737 & 0.833 &  & 0.162 & 0.361 & 0.481 &  & 0.950 & 0.988 & 0.997 \\ 
 $T= 500$ &  & 1.000 & 1.000 & 1.000 &  & 0.999 & 1.000 & 1.000 &  & 0.773 & 0.919 & 0.968 &  & 0.285 & 0.473 & 0.649 &  & 0.994 & 0.999 & 1.000 \\ 
\bottomrule
\end{tabularx}
}
\end{sidewaystable}

\newpage

To implement our test, we choose $K$ to be an Epanechnikov kernel and define the set $\mathcal{G}_T$ of location-scale points $(u,h)$ as
\begin{align}
\mathcal{G}_T = \big\{ (u, h): & \, \, u = 5k/T \text{ for some } 1 \le k \le T/5 \text{ and } \nonumber \\ & \, \, h = (3+5\ell)/T \text{ for some } 0 \le \ell \le T/20 \big\}. \label{grid-sim-app}
\end{align}
We thus take into account all rescaled time points $u \in [0,1]$ on an equidistant grid with step length $5/T$. For the bandwidth $h = (3 + 5\ell)/T$ and any $u \in [h,1-h]$, the kernel weights $\kernel(h^{-1} \{t/T-u\})$ are non-zero for exactly $5 + 10 \ell$ observations. Hence, the bandwidths $h$ in $\mathcal{G}_T$ correspond to effective sample sizes of $5, 15, 25, \ldots$ up to approximately $T/4$ data points. As a robustness check, we have re-run the simulations for a number of other grids. As the results are very similar, we do however not report them here. The long-run error variance $\sigma^2$ is estimated by the procedures from Section \ref{subsec-error-var-AR}: We first compute the estimator $\widehat{\boldsymbol{a}}$ of the AR parameter(s), where we use $\overline{r} = 10$ and the pilot estimator $\widetilde{\boldsymbol{a}}_q$ with $q = 25$. Based on $\widehat{\boldsymbol{a}}$, we then compute the estimator $\widehat{\sigma}^2$ of the long-run error variance $\sigma^2$. As a further robustness check, we have re-run the simulations for other choices of the parameters $q$ and $\overline{r}$, which yields very similar results. The dependence of the estimators $\widehat{\boldsymbol{a}}$ and $\widehat{\sigma}^2$ on $q$ and $\overline{r}$ is further explored in Section \ref{subsec-sim-3}. To compute the critical values of the multiscale test, we simulate $1000$ values of the statistic $\Phi_T$ defined in Section \ref{subsec-method-test} and compute their empirical $(1-\alpha)$ quantile $q_T(\alpha)$.

Tables \ref{tab:size_test} and \ref{tab:power_test} report the simulation results for the sample sizes $T=250,350,500$ and the significance levels $\alpha = 0.01, 0.05, 0.10$. The sample size $T = 350$ is approximately equal to the time series length $359$ in the real-data example of Section \ref{sec-data}. To produce our simulation results, we generate $S=1000$ samples for each model specification and carry out the multiscale test for each sample. The entries of Tables \ref{tab:size_test} and \ref{tab:power_test} are computed as the number of simulations in which the test rejects divided by the total number of simulations. As can be seen from Table \ref{tab:size_test}, the actual size of the test is fairly close to the nominal target $\alpha$ for all the considered AR specifications and sample sizes. Hence, the test has approximately the correct size. Inspecting Table \ref{tab:power_test}, one can further see that the test has reasonable power properties. For all the considered AR specifications, the power increases quickly (i) as the sample size gets larger and (ii) as we move away from the null by increasing the slope parameter $\beta$. The power is of course quite different across the various AR specifications. In particular, it is much lower for positive than for negative values of $a_1$ in the AR($1$) case, the lowest power numbers being obtained for the largest positive value $a_1 = 0.5$ under consideration. This reflects the fact that it is more difficult to detect a trend when there is strong positive autocorrelation in the data. For the AR($2$) specification of the errors, the sample size $T=350$ and the slopes $\beta = 2.0$ and $\beta = 2.5$, which yield the two model specifications that resemble the real-life data in Section \ref{sec-data} the most, the power of the test is above $92\%$ for the significance levels $\alpha = 0.05$ and $\alpha = 0.1$ and above $75\%$ for $\alpha = 0.01$. Hence, our method has substantial power in the two simulation scenarios which are closest to the situation in the application.

\subsection{Comparison with SiZer}\label{subsec-sim-2}

We now compare our multiscale test to SiZer for times series which was developed in \cite{Rondonotti2004}, \cite{Rondonotti2007} and \cite{ParkHannigKang2009}. Roughly speaking, the SiZer method proceeds as follows: For each location $u$ and bandwidth $h$ in a pre-specified set, SiZer computes an estimator $\widehat{m}_h^\prime(u)$ of the derivative $m^\prime(u)$ and a corresponding confidence interval. For each $(u,h)$, it then checks whether the confidence interval includes the value $0$. The set $\Pi_T^{\text{SiZer}}$ of points $(u,h)$ for which the confidence interval does not include $0$ corresponds to the set of intervals $\Pi_T^\pm$ for which our multiscale test finds an increase/decrease in the trend $m$. In order to explore how our test performs in comparison to SiZer, we compare the two sets $\Pi_T^\pm$ and $\Pi_T^{\text{SiZer}}$ in different ways to each other in what follows.

In order to implement SiZer for time series, we follow the exposition in \cite{ParkHannigKang2009}.\footnote{We have also examined the somewhat different implementation from \cite{Rondonotti2007}. As this yields worse simulation results than the procedure from \cite{ParkHannigKang2009}, we however do not report them here.} The details are given in Section S.3 in the Supplementary Material. To simplify the implementation of SiZer, we assume that the autocovariance function $\gamma_\varepsilon(\cdot)$ of the error process and thus the long-run error variance $\sigma^2$ is known. Our multiscale test is implemented in the same way as in Section \ref{subsec-sim-1}. To keep the comparison fair, we treat $\sigma^2$ as known also when implementing our method. Moreover, we use the same grid $\mathcal{G}_T$ of points $(u,h)$ for both methods. To achieve this, we start off with the grid $\mathcal{G}_T$ from \eqref{grid-sim-app}. We then follow \cite{Rondonotti2007} and \cite{ParkHannigKang2009} and restrict attention to those points $(u,h) \in \mathcal{G}_T$ for which the effective sample size $\text{ESS}^*(u,h)$ for correlated data is not smaller than $5$. This yields the grid $\mathcal{G}_T^* = \{ (u, h) \in \mathcal{G}_T : \text{ESS}^*(u, h) \geq 5 \}$. A detailed discussion of the effective sample size $\text{ESS}^*(u,h)$ for correlated data can be found in \cite{Rondonotti2007}.

\begin{table}[t!]
\centering
\footnotesize{
\caption{Size of our multiscale test (MT) and SiZer for different model specifications.}\label{tab:size_comparison}
\newcolumntype{C}[1]{>{\hsize=#1\hsize\centering\arraybackslash}X}
\newcolumntype{Z}{>{\centering\arraybackslash}X}
\begin{tabularx}{\textwidth}{C{2} C{0.0001} ZZZZZZ C{0.0001} ZZZZZZ} 
\toprule
        & & \multicolumn{6}{c}{$a_1 = -0.25$} & & \multicolumn{6}{c}{$a_1 = 0.25$} \\ 
\cmidrule[0.4pt]{3-8} \cmidrule[0.4pt]{10-15}
        & & \multicolumn{2}{c}{$\alpha=0.01$} & \multicolumn{2}{c}{$\alpha=0.05$}  & \multicolumn{2}{c}{$\alpha=0.1$} 
        & & \multicolumn{2}{c}{$\alpha=0.01$} & \multicolumn{2}{c}{$\alpha=0.05$}  & \multicolumn{2}{c}{$\alpha=0.1$} \\[0.1cm]
        & & MT & SiZer & MT & SiZer & MT & SiZer & & MT & SiZer & MT & SiZer & MT & SiZer \\
\cmidrule[0.4pt]{1-15}
$T = 250$ &  & 0.018 & 0.112 & 0.040 & 0.374 & 0.104 & 0.575 &  & 0.017 & 0.106 & 0.034 & 0.347 & 0.092 & 0.522 \\ 
  $T= 350$ &  & 0.012 & 0.140 & 0.058 & 0.426 & 0.080 & 0.621 &  & 0.012 & 0.130 & 0.046 & 0.399 & 0.074 & 0.578 \\ 
  $T=500$ &  & 0.005 & 0.140 & 0.041 & 0.489 & 0.097 & 0.680 &  & 0.006 & 0.136 & 0.039 & 0.452 & 0.097 & 0.639 \\ 
\bottomrule
\end{tabularx}
\vspace{0.5cm}

\caption{Power of our multiscale test (MT) and SiZer for different model specifications. The three panels (a)--(c) corresponds to different slope parameters $\beta$ of the linear tend $m$.}\label{tab:power_comparison}
\newcolumntype{C}[1]{>{\hsize=#1\hsize\centering\arraybackslash}X}
\newcolumntype{Z}{>{\centering\arraybackslash}X}
\begin{tabularx}{\textwidth}{C{2} C{0.0001} ZZZZZZ C{0.0001} ZZZZZZ} 
\multicolumn{15}{c}{(a) $\beta = 1.0$ for negative $a_1$ and $\beta = 2.0$ for positive $a_1$} \\[0.2cm]
\toprule
        & & \multicolumn{6}{c}{$a_1 = -0.25$} & & \multicolumn{6}{c}{$a_1 = 0.25$} \\ 
\cmidrule[0.4pt]{3-8} \cmidrule[0.4pt]{10-15}
        & & \multicolumn{2}{c}{$\alpha=0.01$} & \multicolumn{2}{c}{$\alpha=0.05$}  & \multicolumn{2}{c}{$\alpha=0.1$} 
        & & \multicolumn{2}{c}{$\alpha=0.01$} & \multicolumn{2}{c}{$\alpha=0.05$}  & \multicolumn{2}{c}{$\alpha=0.1$} \\[0.1cm]
        & & MT & SiZer & MT & SiZer & MT & SiZer & & MT & SiZer & MT & SiZer & MT & SiZer \\
\cmidrule[0.4pt]{1-15}
$T = 250$ &  & 0.218 & 0.544 & 0.454 & 0.869 & 0.664 & 0.949 &  & 0.359 & 0.717 & 0.653 & 0.947 & 0.829 & 0.989 \\ 
 $T= 350$ &  & 0.385 & 0.707 & 0.665 & 0.958 & 0.753 & 0.986 &  & 0.599 & 0.888 & 0.864 & 0.995 & 0.913 & 0.998 \\ 
  $T=500$ &  & 0.581 & 0.899 & 0.862 & 0.993 & 0.949 & 0.999 &  & 0.851 & 0.981 & 0.983 & 1.000 & 0.999 & 1.000 \\ 
 \bottomrule
\end{tabularx}
\vspace{0.25cm}

\begin{tabularx}{\textwidth}{C{2} C{0.0001} ZZZZZZ C{0.0001} ZZZZZZ} 
\multicolumn{15}{c}{(b) $\beta = 1.25$ for negative $a_1$ and $\beta = 2.25$ for positive $a_1$} \\[0.2cm]
\toprule
      & & \multicolumn{6}{c}{$a_1 = -0.25$} & & \multicolumn{6}{c}{$a_1 = 0.25$} \\ 
\cmidrule[0.4pt]{3-8} \cmidrule[0.4pt]{10-15}
      & & \multicolumn{2}{c}{$\alpha=0.01$} & \multicolumn{2}{c}{$\alpha=0.05$}  & \multicolumn{2}{c}{$\alpha=0.1$} 
      & & \multicolumn{2}{c}{$\alpha=0.01$} & \multicolumn{2}{c}{$\alpha=0.05$}  & \multicolumn{2}{c}{$\alpha=0.1$} \\[0.1cm]
      & & MT & SiZer & MT & SiZer & MT & SiZer & & MT & SiZer & MT & SiZer & MT & SiZer \\
\cmidrule[0.4pt]{1-15}
$T=250$ &  & 0.426 & 0.771 & 0.705 & 0.969 & 0.878 & 0.996 &  & 0.537 & 0.861 & 0.791 & 0.987 & 0.932 & 0.999 \\ 
  $T=350$ &  & 0.645 & 0.912 & 0.882 & 0.993 & 0.954 & 1.000 &  & 0.773 & 0.955 & 0.948 & 0.999 & 0.985 & 1.000 \\ 
 $T= 500$ &  & 0.915 & 0.994 & 0.993 & 1.000 & 0.998 & 1.000 &  & 0.962 & 0.999 & 1.000 & 1.000 & 0.999 & 1.000 \\ 
\bottomrule
\end{tabularx}
\vspace{0.25cm}

\begin{tabularx}{\textwidth}{C{2} C{0.0001} ZZZZZZ C{0.0001} ZZZZZZ} 
\multicolumn{15}{c}{(c) $\beta = 1.5$ for negative $a_1$ and $\beta = 2.5$ for positive $a_1$} \\[0.2cm]
\toprule
       & & \multicolumn{6}{c}{$a_1 = -0.25$} & & \multicolumn{6}{c}{$a_1 = 0.25$} \\ 
\cmidrule[0.4pt]{3-8} \cmidrule[0.4pt]{10-15}
       & & \multicolumn{2}{c}{$\alpha=0.01$} & \multicolumn{2}{c}{$\alpha=0.05$}  & \multicolumn{2}{c}{$\alpha=0.1$} 
       & & \multicolumn{2}{c}{$\alpha=0.01$} & \multicolumn{2}{c}{$\alpha=0.05$}  & \multicolumn{2}{c}{$\alpha=0.1$} \\[0.1cm]
       & & MT & SiZer & MT & SiZer & MT & SiZer & & MT & SiZer & MT & SiZer & MT & SiZer \\
\cmidrule[0.4pt]{1-15}
$T=250$ &  & 0.701 & 0.942 & 0.911 & 0.992 & 0.972 & 1.000 &  & 0.698 & 0.941 & 0.908 & 0.993 & 0.970 & 1.000 \\ 
 $T= 350$ &  & 0.895 & 0.994 & 0.981 & 1.000 & 0.996 & 1.000 &  & 0.893 & 0.993 & 0.980 & 1.000 & 0.996 & 1.000 \\ 
  $T=500$ &  & 0.995 & 1.000 & 1.000 & 1.000 & 1.000 & 1.000 &  & 0.995 & 1.000 & 1.000 & 1.000 & 1.000 & 1.000 \\ 
\bottomrule
\end{tabularx}
}
\end{table}

In the first part of the comparison study, we analyse the size and power of the two methods. To do so, we treat SiZer as a rigorous statistical test of the null hypothesis $H_0$ that $m$ is constant on all intervals $[u-h,u+h]$ with $(u,h) \in \mathcal{G}_T^*$. In particular, we let SiZer reject the null if the set $\Pi_T^{\text{SiZer}}$ is non-empty, that is, if the value $0$ is not included in the confidence interval for at least one point $(u,h) \in \mathcal{G}_T^*$. We simulate data from the model $Y_{t,T} = m(t/T) + \varepsilon_t$ with different AR($1$) error processes and different trends $m$. In particular, we let $\{\varepsilon_t\}$ be an AR($1$) process of the form $\varepsilon_t = a_1 \varepsilon_{t-1} + \eta_t$ with $a_1 \in \{ -0.25,0.25\}$ and i.i.d.\ standard normal innovations $\eta_t$. To simulate data under the null, we set $m = 0$ as in the previous section. To generate data under the alternative, we consider the linear trends $m(u) = \beta (u - 0.5)$ with different slopes $\beta$. As it is more difficult to detect a trend $m$ in the data when the error terms are positively autocorrelated, we choose the slopes $\beta$ larger in the AR($1$) case with $a_1 = 0.25$ than in the case with $a_1 = -0.25$. In particular, we let $\beta \in \{ 1.0,1.25,1.5 \}$ when $a_1 = -0.25$ and $\beta \in \{ 2.0,2.25,2.5 \}$ when $a_1 = 0.25$. Further model specifications with nonlinear trends are considered in the second part of the comparison study. To produce our simulation results, we generate $S=1000$ samples for each model specification and carry out the two methods for each sample.

The simulation results are reported in Tables \ref{tab:size_comparison} and \ref{tab:power_comparison}. Both for our multiscale test and SiZer, the entries in the tables are computed as the number of simulations in which the respective method rejects the null hypothesis $H_0$ divided by the total number of simulations. As can be seen from Table \ref{tab:size_comparison}, our test has approximately correct size in all of the considered settings, whereas SiZer is very liberal and rejects the null way too often. Examining Table \ref{tab:power_comparison}, one can further see that our procedure has reasonable power against the considered alternatives. The power numbers are of course higher for SiZer, which is a trivial consequence of the fact that SiZer is extremely liberal. These numbers should thus be treated with caution. All in all, the simulations suggest that SiZer can hardly be regarded as a rigorous statistical test of the null hypothesis $H_0$ that $m$ is constant on all intervals $[u-h,u+h]$ with $(u,h) \in \mathcal{G}_T^*$. This is not very surprising as SiZer is not designed to be such a test but to produce informative SiZer maps. In particular, the confidence intervals of SiZer are not constructed to control the level $\alpha$ under $H_0$. In what follows, we thus attempt to compare the two methods in a different way which goes beyond mere size and power comparisons.

Both our method and SiZer can be regarded as statistical tools to identify time regions where the curve $m$ is increasing/decreasing.\footnote{More precisely speaking, SiZer is usually interpreted as investigating the curve $m$, viewed at different levels of resolution, rather than the curve $m$ itself. Put differently, the underlying object of interest is a family of smoothed versions of $m$ rather than $m$ itself.} Suppose that $m$ is increasing/decreasing in the time region $\mathcal{R} \subset [0,1]$ but constant otherwise, that is, $m^\prime(u) \ne 0$ for all $u \in \mathcal{R}$ and $m^\prime(u) = 0$ for all $u \notin \mathcal{R}$. A natural question is the following: How well can the two methods identify the time region $\mathcal{R}$? In our framework, information on the region $\mathcal{R}$ is contained in the minimal intervals of the set $\Pi_T^\pm$. In particular, the union $\mathcal{R}_T^\pm$ of the minimal intervals in $\Pi_T^\pm$ can be regarded as an estimate of $\mathcal{R}$. This follows from the results in Propositions \ref{prop-test-2} and \ref{prop-test-3}. Let $\mathcal{R}_T^{\text{SiZer}}$ be the union of the minimal intervals in $\Pi_T^{\text{SiZer}}$. In what follows, we compare $\mathcal{R}_T^\pm$ and $\mathcal{R}_T^{\text{SiZer}}$ to the region $\mathcal{R}$. This gives us information on how well the two methods approximate the true region where $m$ is increasing/decreasing.\footnote{The same exercise could of course also be carried out separately for the time region where the trend $m$ increases and the region where it decreases.}

\begin{figure}[t]
\begin{subfigure}{.5\textwidth}
\centering
\includegraphics[width=.9\linewidth]{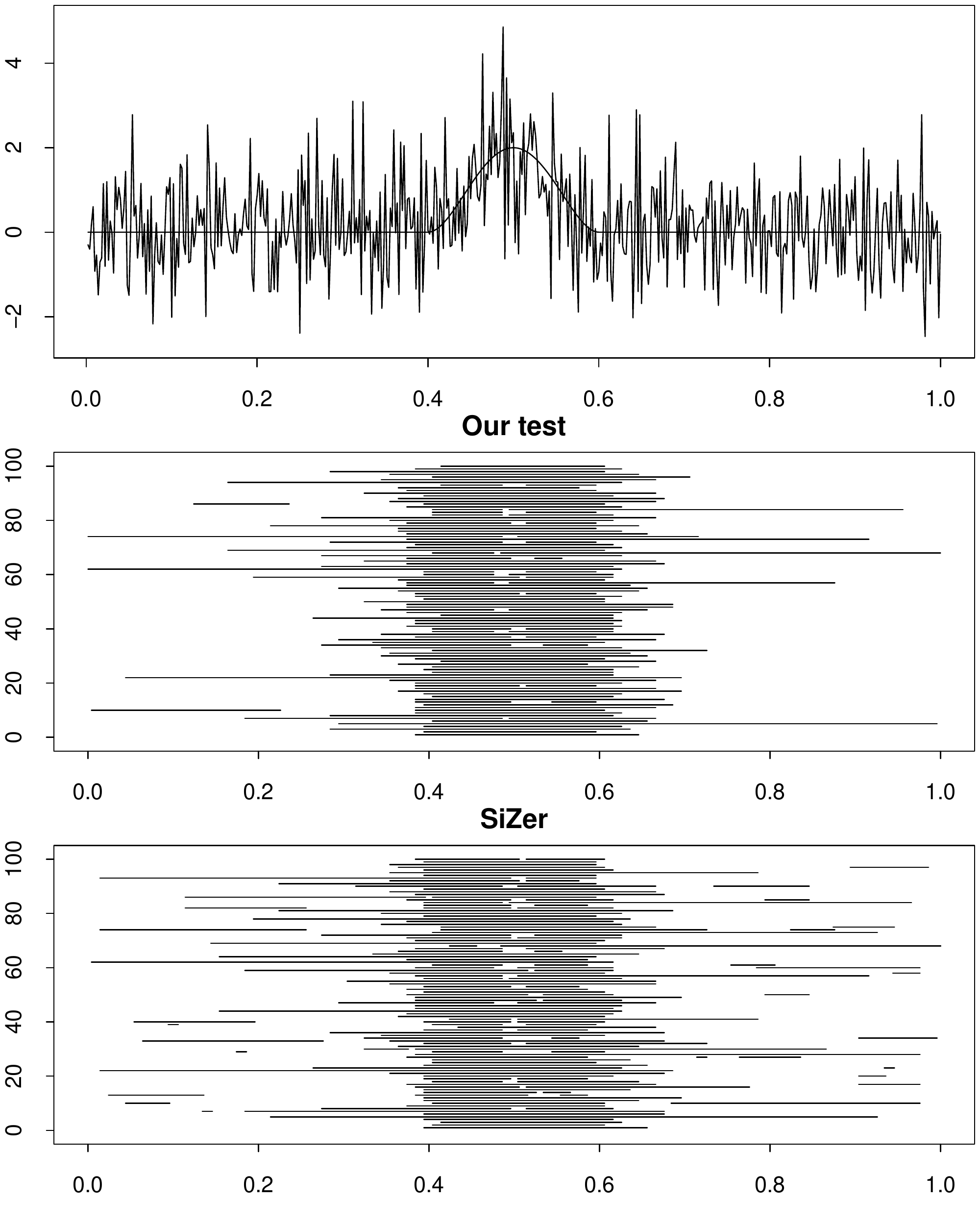}
\caption{$a_1 = -0.25$}
\end{subfigure}
\begin{subfigure}{.5\textwidth}
\centering
\includegraphics[width=.9\linewidth]{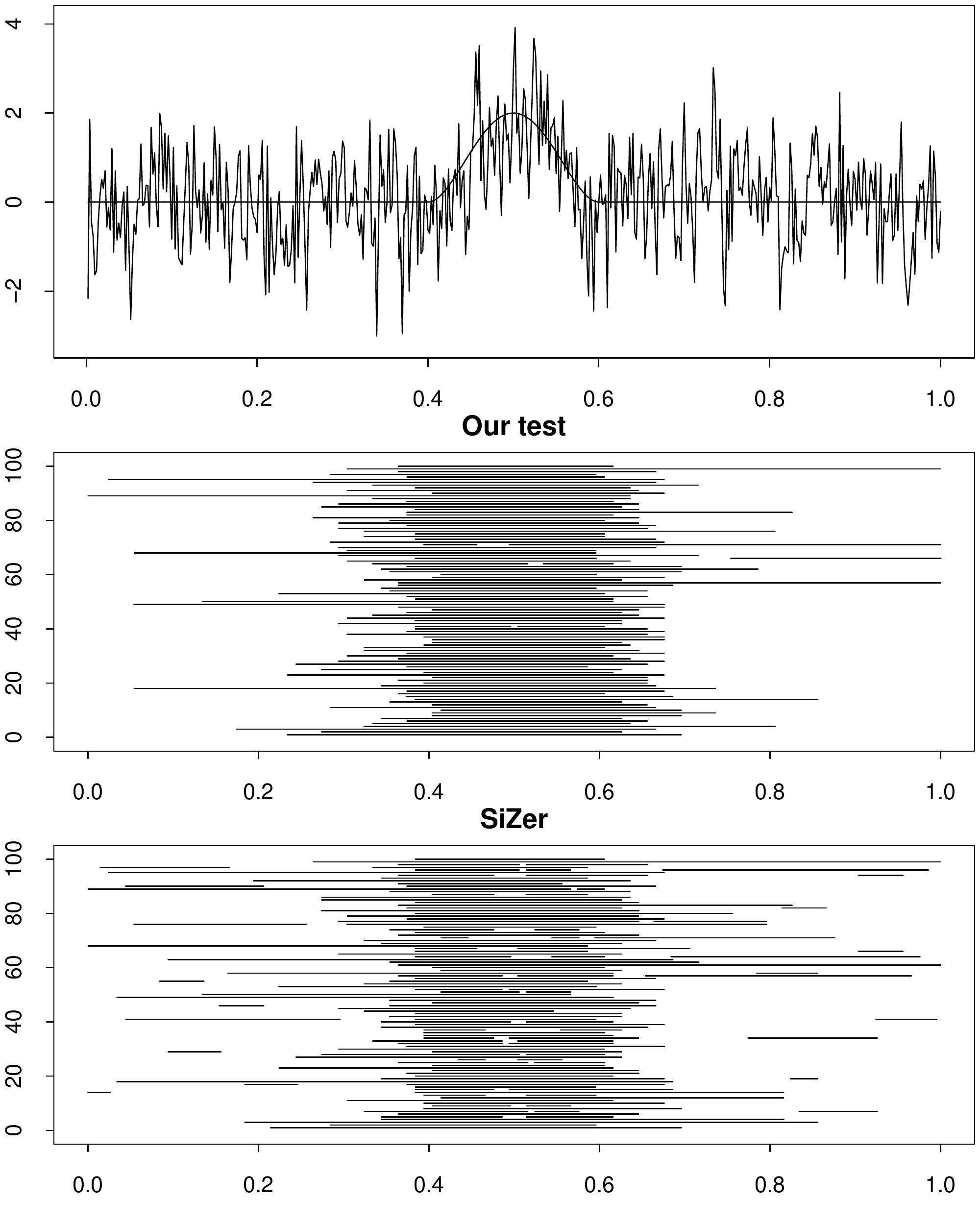}
\caption{$a_1 = 0.25$}
\end{subfigure}
\caption{Comparison of the regions $\mathcal{R}_T^\pm$ and $\mathcal{R}_T^{\text{SiZer}}$. Subfigure (a) corresponds to the model setting with the AR parameter $a_1 = -0.25$, subfigure (b) to the setting with $a_1 = 0.25$. The upper panel of each subfigure shows a simulated time series path together with the underlying trend function $m$. The middle panel depicts the regions $\mathcal{R}_T^\pm$ produced by our multiscale test for $100$ simulation runs. The lower panel presents the regions $\mathcal{R}_T^{\text{SiZer}}$ produced by SiZer.}  
\label{fig:comparison_SiZer}
\end{figure}

We consider the same simulation setup as in the first part of the comparison study, only the trend function $m$ is different. We let $m$ be defined as $m(u) = 2 \cdot \ind(u \in [0.4,0.6]) \cdot (1 - 100 \{u-0.5\}^2)^2$, which implies that $\mathcal{R} = (0.4,0.5) \cup (0.5,0.6)$. The function $m$ is plotted in the two upper panels of Figure \ref{fig:comparison_SiZer}. We set the significance level to $\alpha= 0.05$ and the sample size to $T=500$. For each AR parameter $a_1 \in \{ -0.25,0.25 \}$, we simulate $S=100$ samples and compute $\mathcal{R}_T^\pm$ and $\mathcal{R}_T^{\text{SiZer}}$ for each sample. The simulation results are depicted in Figure \ref{fig:comparison_SiZer}, the two subfigures (a) and (b) corresponding to different AR parameters. The upper panel of each subfigure displays the time series path of a representative simulation together with the trend function $m$. The middle panel shows the regions $\mathcal{R}_T^\pm$ produced by our multiscale approach for the $100$ simulation runs: On the $y$-axis, the simulation runs $i$ are enumerated for $1 \le i \le 100$, and the black line at $y$-level $i$ represents $\mathcal{R}_T^\pm$ for the $i$-th simulation. Finally, the lower panel of each subfigure depicts the regions $\mathcal{R}_T^{\text{SiZer}}$ in an analogous way.

Inspecting Figure \ref{fig:comparison_SiZer}, our multiscale method can be seen to approximate the region $\mathcal{R}$ fairly well in both simulation scenarios under consideration. In particular, $\mathcal{R}_T^\pm$ gives a good approximation to the region $\mathcal{R}$ for most simulations. Only in some simulation runs, $\mathcal{R}_T^\pm$ is too large compared to $\mathcal{R}$, which means that our method is not able to locate the region $\mathcal{R}$ sufficiently precisely. Overall, the SiZer method also produces quite satisfactory results. However, the SiZer estimates of $\mathcal{R}$ are not as precise as ours. In particular, SiZer spuriously finds regions of decrease/increase outside the interval $\mathcal{R}$ much more often than our method. It thus frequently mistakes fluctuations in the time series which are due to the dependence in the error terms for increases/decreases in the trend $m$.

To sum up, our multiscale test exhibits good size and power properties in the simulations, and the minimal intervals produced by it identify the time regions where $m$ increases/decreases in a quite reliable way. SiZer performs clearly worse in these respects. Nevertheless, it may still produce informative SiZer plots. All in all, we would like to regard the two methods as complementary rather than direct competitors. SiZer is an explorative tool which aims to give an overview of the increases/decreases in $m$ by means of a SiZer plot. Our method, in contrast, is tailored to be a rigorous statistical test of the hypothesis $H_0$. In particular, it allows to make rigorous confidence statements about the time regions where the trend $m$ increases/decreases.

\subsection{Small sample properties of the long-run variance estimator}\label{subsec-sim-3}

In the final part of the simulation study, we examine the estimators of the AR para\-meters and the long-run error variance from Section \ref{subsec-error-var-AR}. We simulate data from the model $Y_{t,T} = m(t/T) + \varepsilon_t$, where $\{ \varepsilon_t\}$ is an AR($1$) process of the form $\varepsilon_t = a_1 \varepsilon_{t-1} + \eta_t$. We consider the AR parameters $a_1 \in \{-0.95,-0.75,-0.5,-0.25,0.25,0.5,0.75,0.95\}$ and let $\eta_t$ be i.i.d.\ standard normal innovation terms. We report our findings for a specific sample size $T$, in particular for $T=500$, as the results for other sample sizes are very similar. For simplicity, $m$ is chosen to be a linear function of the form $m(u) = \beta u$ with the slope parameter $\beta$. For each value of $a_1$, we consider two different slopes $\beta$, one corresponding to a moderate and one to a pronounced trend $m$. In particular, we let $\beta = s_\beta \sqrt{\var(\varepsilon_t)}$ with $s_\beta \in \{1,10\}$. When $s_\beta = 1$, the slope $\beta$ is equal to the standard deviation $\sqrt{\var(\varepsilon_t)}$ of the error process, which yields a moderate trend $m$. When $s_\beta = 10$, in contrast, the slope $\beta$ is $10$ times as large as $\sqrt{\var(\varepsilon_t)}$, which results in a quite pronounced trend $m$.

For each model specification, we generate $S=1000$ data samples and compute the following quantities for each simulated sample: 
\begin{enumerate}[label=(\roman*),leftmargin=0.9cm]
\item the pilot estimator $\widetilde{a}_q$ from \eqref{est-AR-FS} with the tuning parameter $q$.
\item the estimator $\widehat{a}$ from \eqref{est-AR} with the tuning parameter $\overline{r}$ as well as the long-run variance estimator $\widehat{\sigma}^2$ from \eqref{est-lrv}. 
\item the estimators of $a_1$ and $\sigma^2$ from \cite{Hall2003}, which are denoted by $\widehat{a}_{\text{HvK}}$ and $\widehat{\sigma}^2_{\text{HvK}}$ for ease of reference. The estimator $\widehat{a}_{\text{HvK}}$ is computed as described in Section 2.2 of \cite{Hall2003} and $\widehat{\sigma}^2_{\text{HvK}}$ as defined at the bottom of p.447 in Section 2.3. The estimator $\widehat{a}_{\text{HvK}}$ (as well as $\widehat{\sigma}^2_{\text{HvK}}$) depends on two tuning parameters which we denote by $m_1$ and $m_2$ as in \cite{Hall2003}. 
\item oracle estimators $\widehat{a}_{\text{oracle}}$ and $\widehat{\sigma}^2_{\text{oracle}}$ of $a_1$ and $\sigma^2$, which are constructed under the assumption that the error process $\{\varepsilon_t\}$ is observed. For each simulation run, we compute $\widehat{a}_{\text{oracle}}$ as the maximum likelihood estimator of $a_1$ from the time series of simulated error terms $\varepsilon_1,\ldots,\varepsilon_T$. We then calculate the residuals $r_t = \varepsilon_t - \widehat{a}_{\text{oracle}} \, \varepsilon_{t-1}$ and estimate the innovation variance $\nu^2 = \ex[\eta_t^2]$ by $\widehat{\nu}_{\text{oracle}}^2 = (T-1)^{-1} \sum_{t=2}^T r_t^2$. Finally, we set $\widehat{\sigma}^2_{\text{oracle}} = \widehat{\nu}_{\text{oracle}}^2 / (1 - \widehat{a}_{\text{oracle}})^2$. 
\end{enumerate}
Throughout the section, we set $q = 25$, $\overline{r} = 10$ and $(m_1,m_2) = (20,30)$. We in particular choose $q$ to be in the middle of $m_1$ and $m_2$ to make the tuning parameters of the estimators $\widetilde{a}_q$ and $\widehat{a}_{\text{HvK}}$ more or less comparable. In order to assess how sensitive our estimators are to the choice of $q$ and $\overline{r}$, we carry out a number of robustness checks, considering a range of different values for $q$ and $\overline{r}$. In addition, we vary the tuning parameters $m_1$ and $m_2$ of the estimators from \cite{Hall2003} in order to make sure that the results of our comparison study are not driven by the particular choice of any of the involved tuning parameters. The results of our robustness checks are reported in Section S.3 of the Supplementary Material. They show that the results of our comparison study are robust to 
different choices of the parameters $q$, $\overline{r}$ and $(m_1,m_2)$. Moreover, they indicate that our estimators are rather insensitive to the choice of tuning parameters.

\begin{figure}[t!]
\begin{subfigure}[b]{0.475\textwidth}
\includegraphics[width=\textwidth]{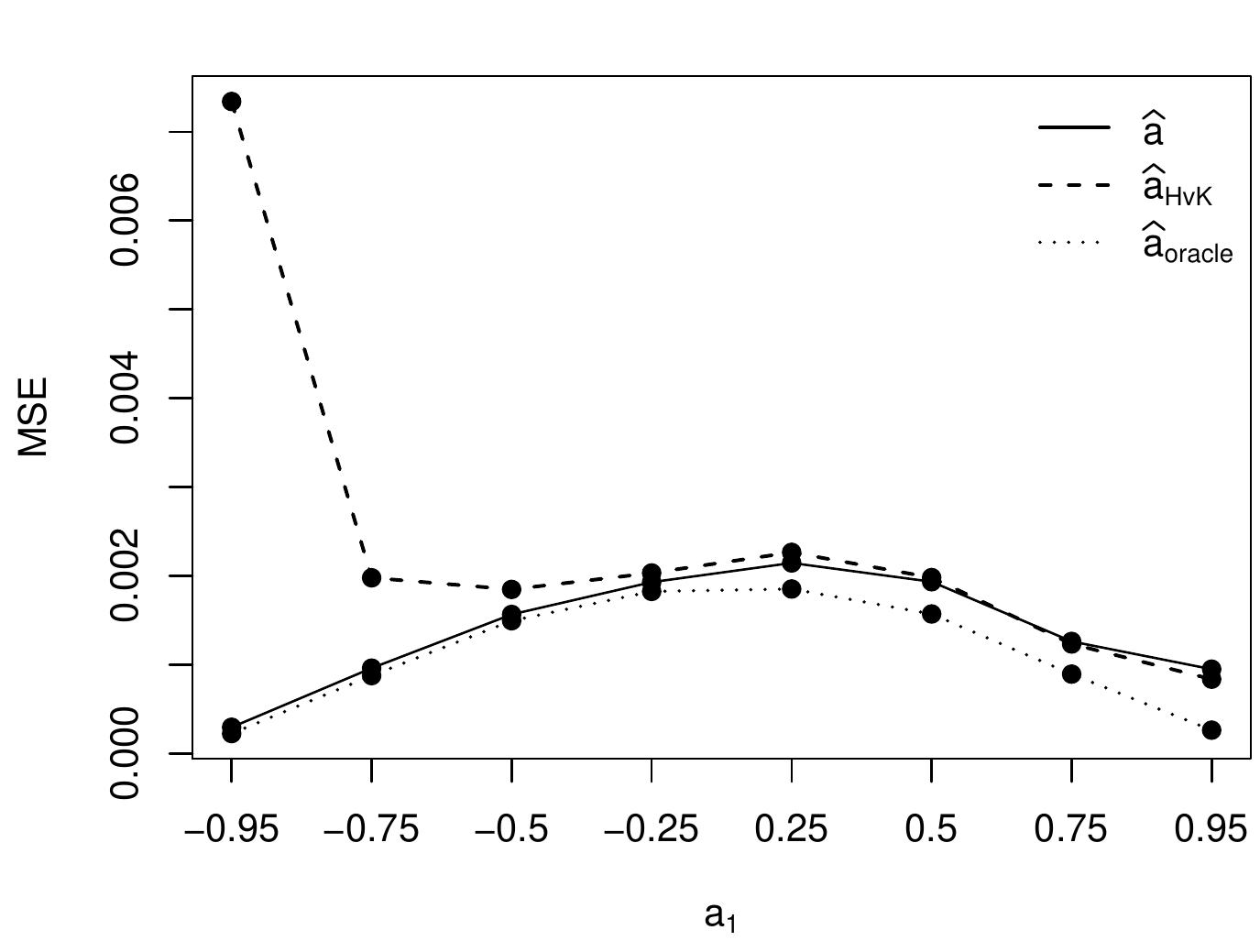}
\end{subfigure}\hspace{0.25cm}
\begin{subfigure}[b]{0.475\textwidth}
\includegraphics[width=\textwidth]{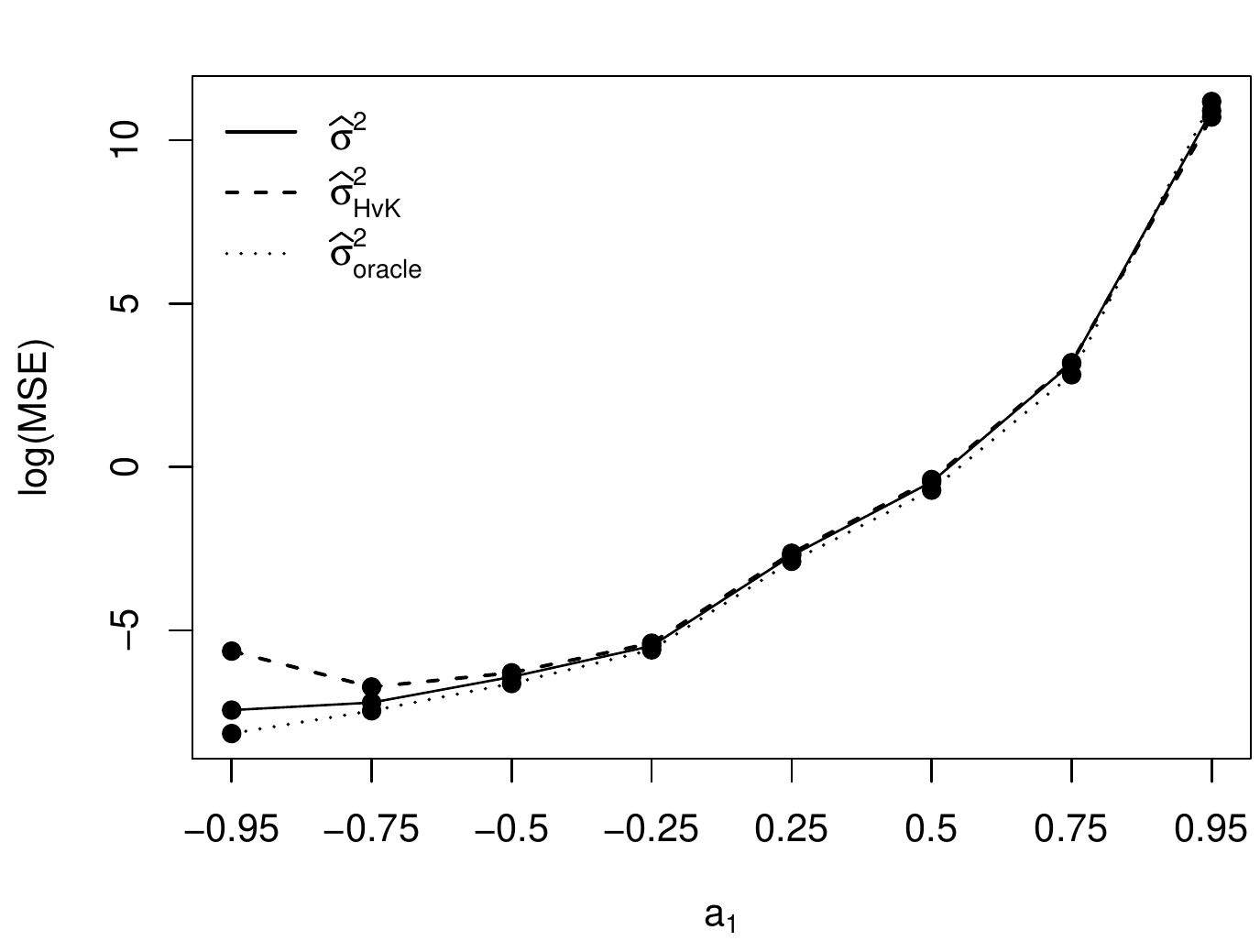}
\end{subfigure}
\caption{MSE values for the estimators $\widehat{a}$, $\widehat{a}_{\text{HvK}}$, $\widehat{a}_{\text{oracle}}$ and $\widehat{\sigma}^2$, $\widehat{\sigma}^2_{\text{HvK}}$, $\widehat{\sigma}^2_{\text{oracle}}$ in the simulation scenarios with a moderate trend ($s_\beta=1$).}\label{fig:MSE_slope1}
\end{figure}

\begin{figure}[t!]
\begin{subfigure}[b]{0.475\textwidth}
\includegraphics[width=\textwidth]{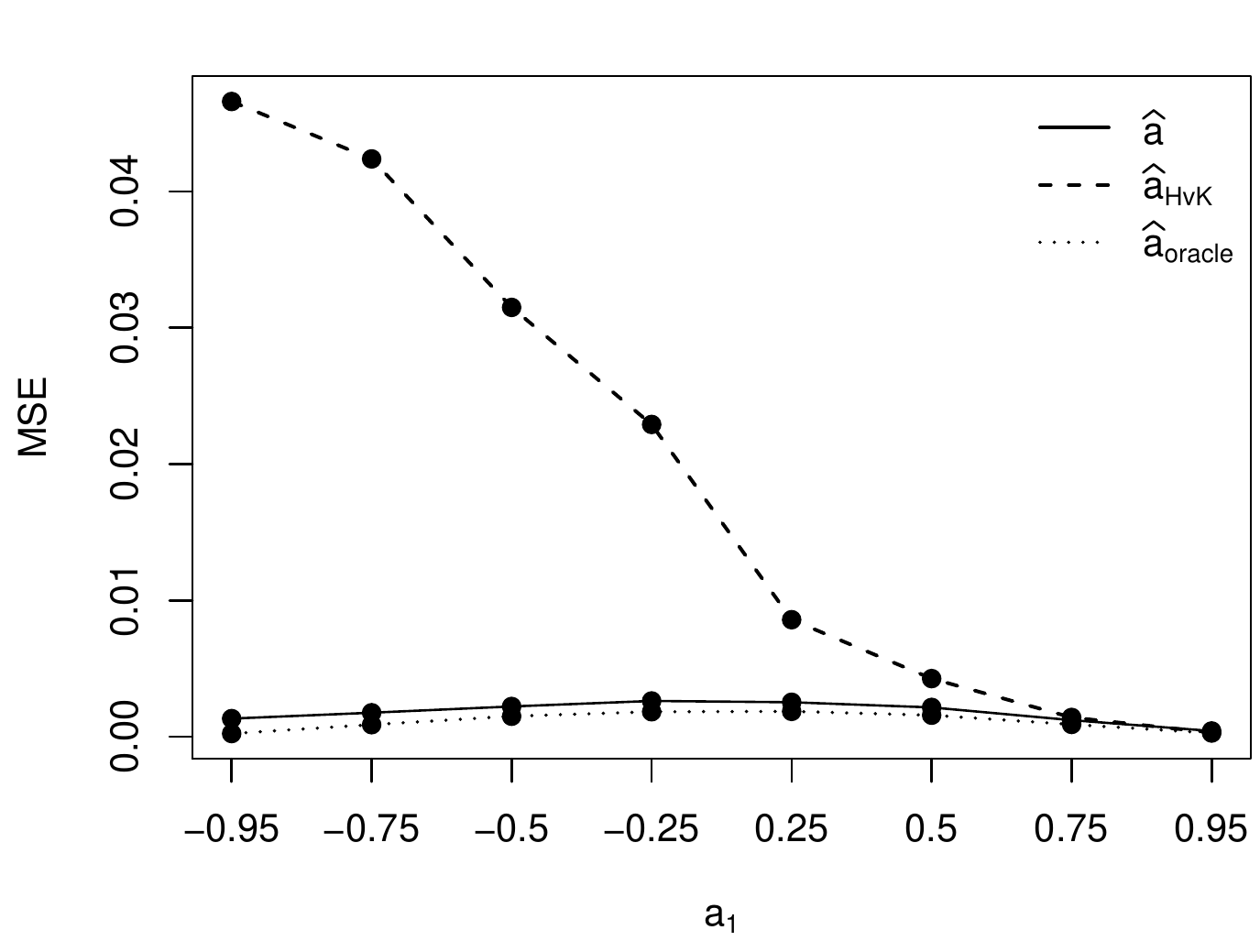}
\end{subfigure}\hspace{0.25cm}
\begin{subfigure}[b]{0.475\textwidth}
\includegraphics[width=\textwidth]{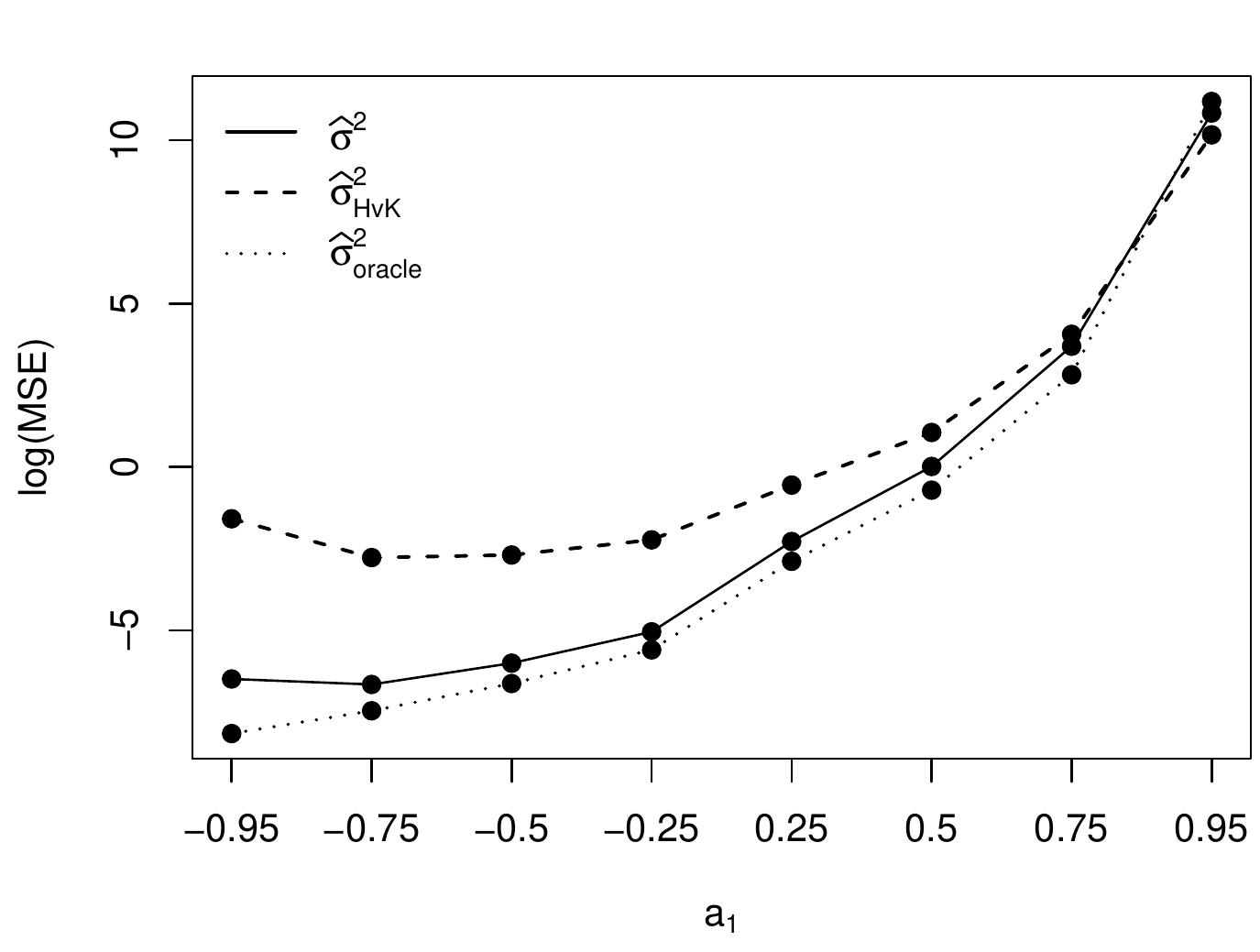}
\end{subfigure}
\caption{MSE values for the estimators $\widehat{a}$, $\widehat{a}_{\text{HvK}}$, $\widehat{a}_{\text{oracle}}$ and $\widehat{\sigma}^2$, $\widehat{\sigma}^2_{\text{HvK}}$, $\widehat{\sigma}^2_{\text{oracle}}$ in the simulation scenarios with a pronounced trend ($s_\beta=10$).}\label{fig:MSE_slope10}
\end{figure}

\begin{figure}[t!]
\centering
\begin{subfigure}[b]{\textwidth}
\includegraphics[width=\textwidth]{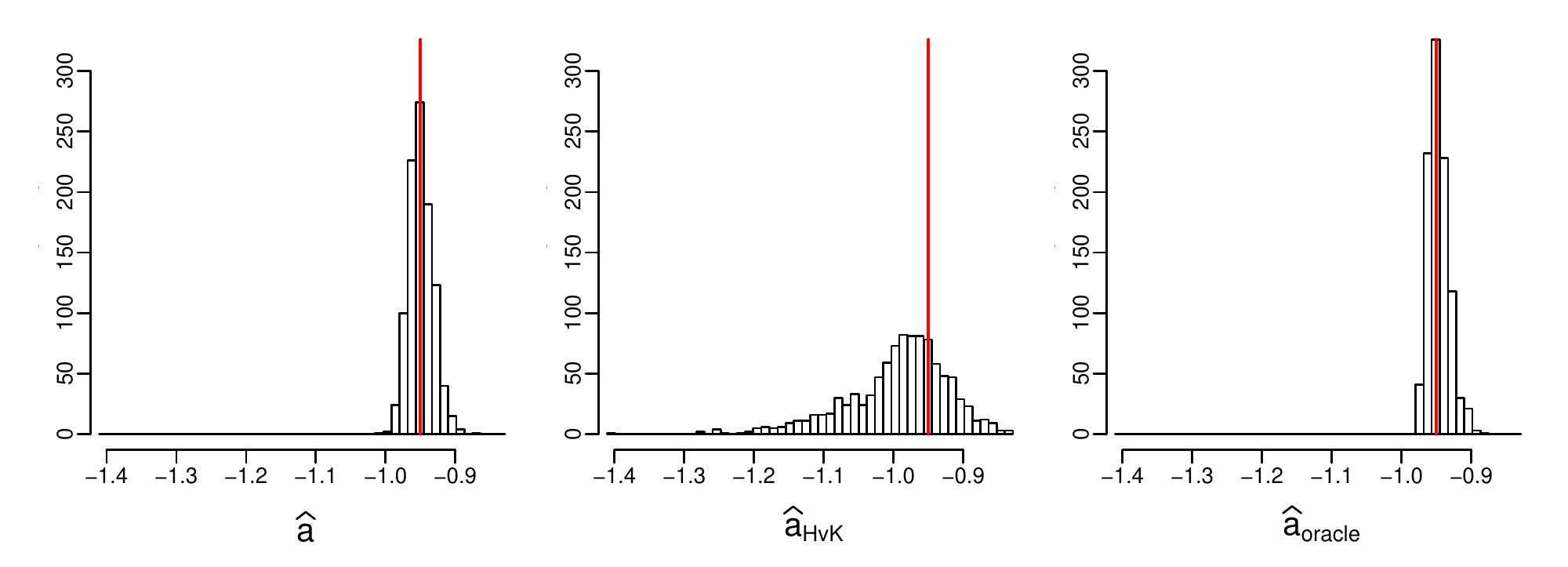}
\end{subfigure}
\begin{subfigure}[b]{\textwidth}
\includegraphics[width=\textwidth]{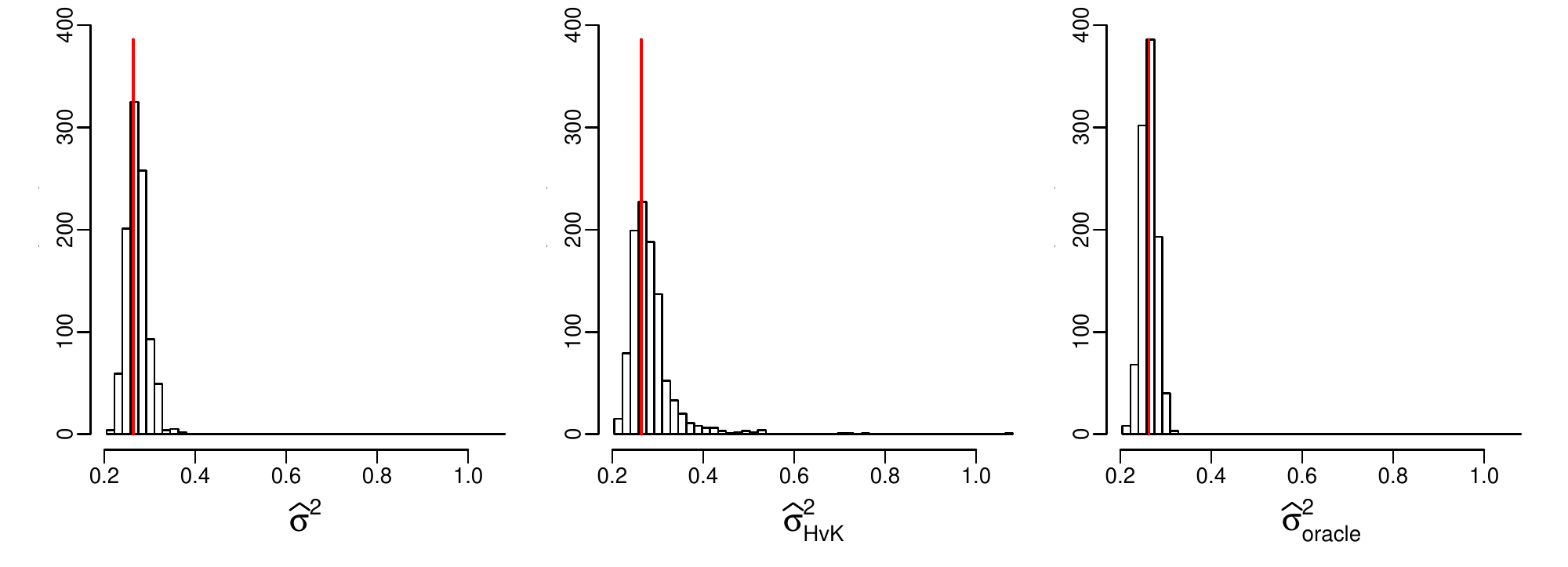}
\end{subfigure}
\caption{Histograms of the simulated values produced by the estimators $\widehat{a}$, $\widehat{a}_{\text{HvK}}$, $\widehat{a}_{\text{oracle}}$ and $\widehat{\sigma}^2$, $\widehat{\sigma}^2_{\text{HvK}}$, $\widehat{\sigma}^2_{\text{oracle}}$ in the scenario with $a_1 = -0.95$ and $s_\beta = 1$. The vertical red lines indicate the true values of $a_1$ and $\sigma^2$.}\label{fig:hist_scenario1} 
\end{figure}

\begin{figure}[t!]
\centering
\begin{subfigure}[b]{\textwidth}
\includegraphics[width=\textwidth]{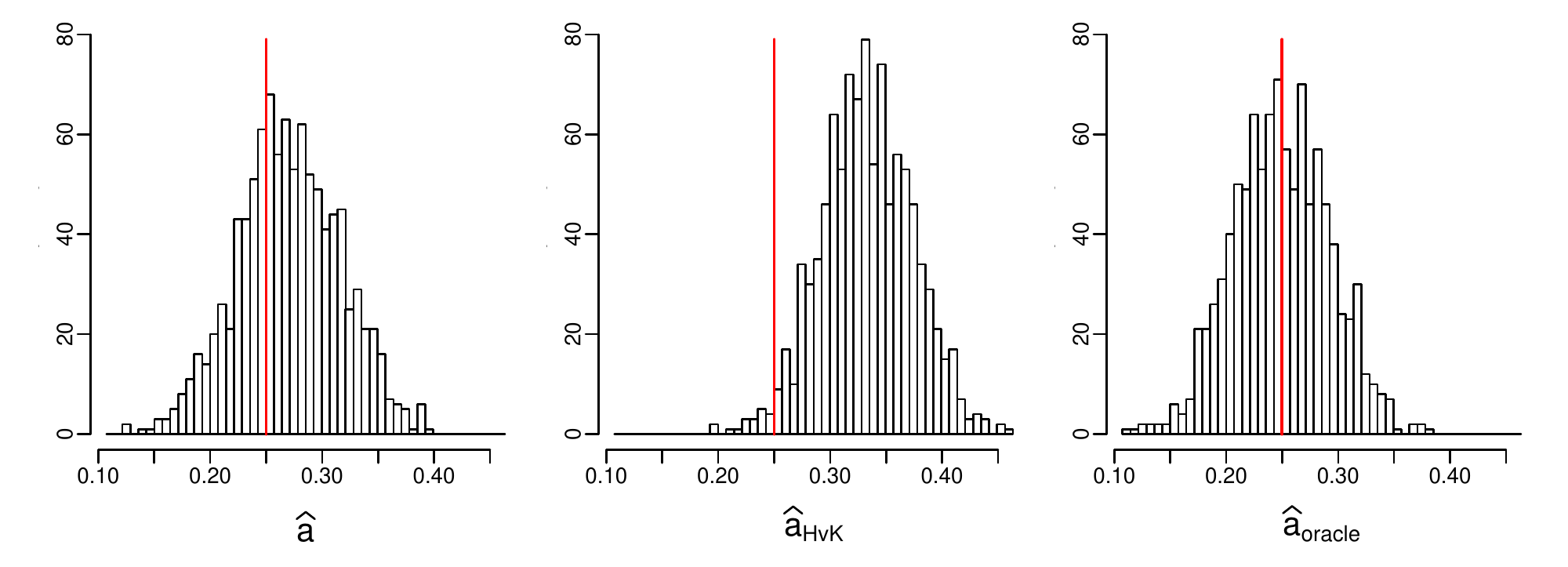}
\end{subfigure}
\begin{subfigure}[b]{\textwidth}
\includegraphics[width=\textwidth]{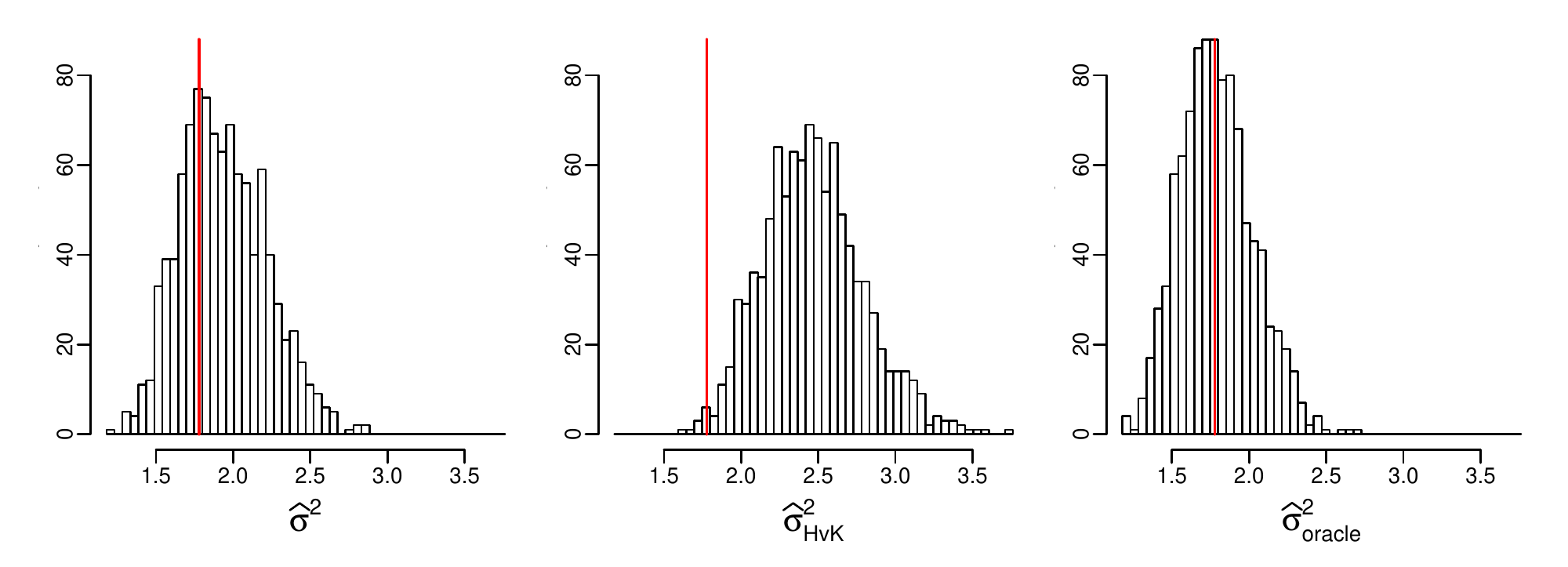}
\end{subfigure}
\caption{Histograms of the simulated values produced by the estimators $\widehat{a}$, $\widehat{a}_{\text{HvK}}$, $\widehat{a}_{\text{oracle}}$ and $\widehat{\sigma}^2$, $\widehat{\sigma}^2_{\text{HvK}}$, $\widehat{\sigma}^2_{\text{oracle}}$ in the scenario with $a_1 = 0.25$ and $s_\beta = 10$. The vertical red lines indicate the true values of $a_1$ and $\sigma^2$.}\label{fig:hist_scenario2} 
\end{figure}

For each estimator $\widehat{a}$, $\widehat{a}_{\text{HvK}}$, $\widehat{a}_{\text{oracle}}$ and $\widehat{\sigma}^2$, $\widehat{\sigma}^2_{\text{HvK}}$, $\widehat{\sigma}^2_{\text{oracle}}$ and for each model specification, the simulation output consists in a vector of length $S=1000$ which contains the $1000$ simulated values of the respective estimator. Figures \ref{fig:MSE_slope1} and \ref{fig:MSE_slope10} report the mean squared error (MSE) of these $1000$ simulated values for each estimator. On the $x$-axis of each plot, the various values of the AR parameter $a_1$ are listed which are considered. The solid line in each plot gives the MSE values of our estimators. The dashed and dotted lines specify the MSE values of the HvK and the oracle estimators, respectively. Note that for the long-run variance estimators, the plots report the logarithm of the MSE rather than the MSE itself since the MSE values are too different across simulation scenarios to obtain a reasonable graphical presentation. In addition to the MSE values presented in Figures \ref{fig:MSE_slope1} and \ref{fig:MSE_slope10}, we depict histograms of the $1000$ simulated values produced by the estimators $\widehat{a}$, $\widehat{a}_{\text{HvK}}$, $\widehat{a}_{\text{oracle}}$ and $\widehat{\sigma}^2$, $\widehat{\sigma}^2_{\text{HvK}}$, $\widehat{\sigma}^2_{\text{oracle}}$ for two specific simulation scenarios in Figures \ref{fig:hist_scenario1} and \ref{fig:hist_scenario2}. The main findings can be summarized as follows:  
\begin{enumerate}[label=(\alph*),leftmargin=0.7cm]

\item In the simulation scenarios with a moderate trend ($s_\beta = 1$), the estimators $\widehat{a}_{\text{HvK}}$ and $\widehat{\sigma}^2_{\text{HvK}}$ of \cite{Hall2003} exhibit a similar performance as our estimators $\widehat{a}$ and $\widehat{\sigma}^2$ as long as the AR parameter $a_1$ is not too close to $-1$. For strongly negative values of $a_1$ (in particular for $a_1 = -0.75$ and $a_1 = -0.95$), the estimators perform much worse than ours. This can be clearly seen from the much larger MSE values of the estimators  $\widehat{a}_{\text{HvK}}$ and $\widehat{\sigma}^2_{\text{HvK}}$ for $a_1 = -0.75$ and $a_1 = -0.95$ in Figure \ref{fig:MSE_slope1}. Figure \ref{fig:hist_scenario1} gives some further insights into what is happening here. It shows the histograms of the simulated values produced by the estimators $\widehat{a}$, $\widehat{a}_{\text{HvK}}$, $\widehat{a}_{\text{oracle}}$ and the corresponding long-run variance estimators in the scenario with $a_1=-0.95$ and $s_\beta = 1$. As can be seen, the estimator $\widehat{a}_{\text{HvK}}$ does not obey the causality restriction $|a_1| \le 1$ but frequently takes values substantially smaller than $-1$. This results in a very large spread of the histogram and thus in a disastrous performance of the estimator.\footnote{One could of course set $\widehat{a}_{\text{HvK}}$ to $-(1 - \delta)$ for some small $\delta > 0$ whenever it takes a value smaller than $-1$. This modified estimator, however, is still far from performing in a satisfying way when $a_1$ is close to $-1$.} A similar point applies to the histogram of the long-run variance estimator $\widehat{\sigma}^2_{\text{HvK}}$. Our estimators $\widehat{a}$ and $\widehat{\sigma}^2$, in contrast, exhibit a stable behaviour in this case. \\ 
Interestingly, the estimator $\widehat{a}_{\text{HvK}}$ (as well as the corresponding long-run variance estimator $\widehat{\sigma}^2_{\text{HvK}}$) performs much worse than ours for large negative values but not for large positive values of $a_1$. This can be explained as follows: In the special case of an AR($1$) process, the estimator $\widehat{a}_{\text{HvK}}$ may produce estimates smaller than $-1$ but it cannot become larger than $1$. This can be easily seen upon inspecting the definition of the estimator. Hence, for large positive values of $a_1$, the estimator $\widehat{a}_{\text{HvK}}$ performs well as it satisfies the causality restriction that the estimated AR parameter should be smaller than $1$. 

\item In the simulation scenarios with a pronounced trend ($s_\beta = 10$), the estimators of \cite{Hall2003} are clearly outperformed by ours for most of the AR parameters $a_1$ under consideration. In particular, their MSE values reported in Figure \ref{fig:MSE_slope10} are much larger than the values produced by our estimators for most parameter values $a_1$. The reason is the following: The HvK estimators have a strong bias since the pronounced trend with $s_\beta = 10$ is not eliminated appropriately by the underlying differencing methods. This point is illustrated by Figure \ref{fig:hist_scenario2} which shows histograms of the simulated values for the estimators $\widehat{a}$, $\widehat{a}_{\text{HvK}}$, $\widehat{a}_{\text{oracle}}$ and the corresponding long-run variance estimators in the scenario with $a_1=0.25$ and $s_\beta = 10$. As can be seen, the histogram produced by our estimator $\widehat{a}$ is approximately centred around the true value $a_1 = 0.25$, whereas that of the estimator $\widehat{a}_{\text{HvK}}$ is strongly biased upwards. A similar picture arises for the long-run variance estimators $\widehat{\sigma}^2$ and $\widehat{\sigma}^2_{\text{HvK}}$. \\
Whereas the methods of \cite{Hall2003} perform much worse than ours for negative and moderately positive values of $a_1$, the performance (in terms of MSE) is fairly similar for large values of $a_1$. This can be explained as follows: When the trend $m$ is not eliminated appropriately by taking differences, this creates spurious persistence in the data. Hence, the estimator $\widehat{a}_{\text{HvK}}$ tends to overestimate the AR parameter $a_1$, that is, $\widehat{a}_{\text{HvK}}$ tends to be larger in absolute value than $a_1$. Very loosely speaking, when the parameter $a_1$ is close to $1$, say $a_1 = 0.95$, there is not much room for overestimation since $\widehat{a}_{\text{HvK}}$ cannot become larger than $1$. Consequently, the effect of not eliminating the trend appropriately has a much smaller impact on $\widehat{a}_{\text{HvK}}$ for large positive values of $a_1$. 

\end{enumerate}

\setcounter{equation}{0}
\section{Application}\label{sec-data}

The analysis of time trends in long temperature records is an important task in climatology. Information on the shape of the trend is needed in order to better understand long-term climate variability. The Central England temperature record is the longest instrumental temperature time series in the world. It is a valuable asset for analysing climate variability over the last few hundred years. The data is publicly available on the webpage of the UK Met Office. A detailed description of the data can be found in \cite{Parker1992}. For our analysis, we use the dataset of yearly mean temperatures which consists of $T=359$ observations covering the years from $1659$ to $2017$.

We assume that the data follow the nonparametric trend model $Y_{t,T} = m(t/T) + \varepsilon_t$, where $m$ is the unknown time trend of interest. The error process $\{ \varepsilon_t \}$ is supposed to have the AR($p$) structure $\varepsilon_t = \sum_{j=1}^p a_j \varepsilon_{t-j} + \eta_t$, where $\eta_t$ are i.i.d.\ innovations with mean $0$ and variance $\nu^2$. As pointed out in \cite{Mudelsee2010} among others, this is the most widely used error model for discrete climate time series. To select the AR order $p$, we proceed as follows: We estimate the AR parameters and the corresponding variance of the innovation terms for different AR orders by our methods from Section \ref{subsec-error-var-AR} and choose $p$ to be the minimizer of the Bayesian information criterion (BIC). This yields the AR order $p = 2$. We then estimate the parameters $\boldsymbol{a} = (a_1,a_2)$ and the long-run error variance $\sigma^2$ by the estimators $\widehat{\boldsymbol{a}} = (\widehat{a}_1,\widehat{a}_2)$ and $\widehat{\sigma}^2$, which gives the values  $\widehat{a}_1 = 0.167$, $\widehat{a}_2 = 0.178$ and $\widehat{\sigma}^2 = 0.749$. To select the AR order $p$ and to produce the estimators $\widehat{\boldsymbol{a}}$ and $\widehat{\sigma}^2$, we set $q = 25$ and $\overline{r} = 10$ as in the simulation study of Section \ref{subsec-sim-1}.\footnote{As a robustness check, we have repeated the process of order selection and parameter estimation for other values of $q$ and $\overline{r}$ as well as for other criteria such as FPE, AIC and AICC, which gave similar results.}

\begin{figure}[t]
\centering
\includegraphics[width=0.8\textwidth]{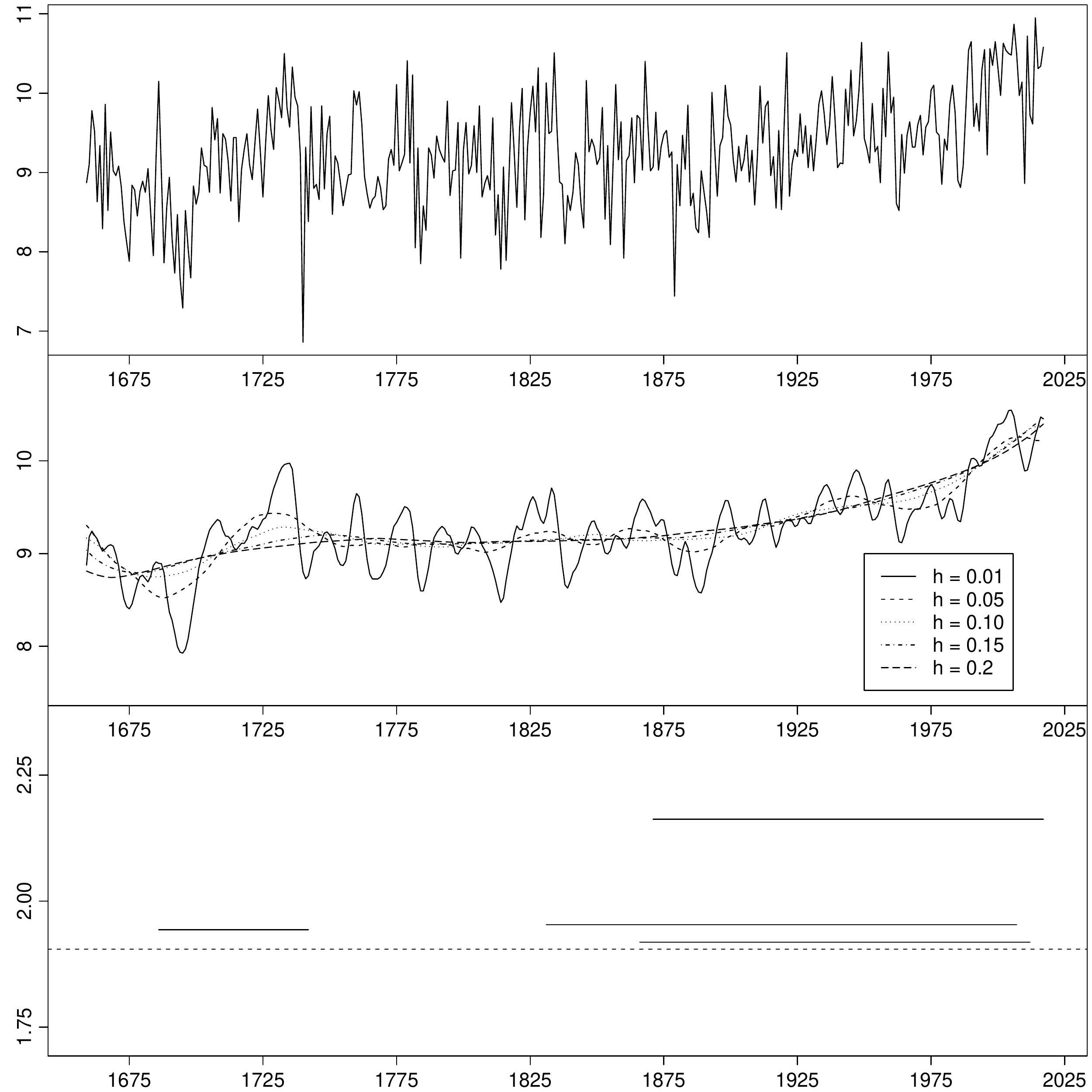}
\caption{Summary of the application results. The upper panel shows the Central England mean temperature time series. The middle panel depicts local linear kernel estimates of the time trend for a number of different bandwidths $h$. The lower panel presents the minimal intervals in the set $\Pi_T^+$ produced by the multiscale test. These are $[1686,1742]$, $[1831,2007]$, $[1866, 2012]$ and $[1871,2017]$.}\label{plot-results-app1}
\end{figure}

\newpage
With the help of our multiscale method from Section \ref{sec-method}, we now test the null hypothesis $H_0$ that $m$ is constant on all intervals $[u-h,u+h]$ with $(u,h) \in \mathcal{G}_T$, where we use the grid $\mathcal{G}_T$ defined in \eqref{grid-sim-app}. To do so, we set the significance level to $\alpha = 0.05$ and implement the test in exactly the same way as in the simulations of Section \ref{subsec-sim-1}. The results are presented in Figure \ref{plot-results-app1}. The upper panel shows the raw temperature time series, whereas the middle panel depicts local linear kernel estimates of the trend $m$ for different bandwidths $h$. As one can see, the shape of the estimated time trend strongly differs with the chosen bandwidth. When the bandwidth is small, there are many local increases and decreases in the estimated trend. When the bandwidth is large, most of these local variations get smoothed out. Hence, by themselves, the nonparametric fits do not give much information on whether the trend $m$ is increasing or decreasing in certain time regions.

Our multiscale test provides this kind of information, which is summarized in the lower panel of Figure \ref{plot-results-app1}. The plot depicts the minimal intervals contained in the set $\Pi_T^+$, which is defined in Section \ref{subsec-method-theo}. The set of intervals $\Pi_T^-$ is empty in the present case. The height at which a minimal interval $I_{u,h} = [u-h,u+h] \in \Pi_T^+$ is plotted indicates the value of the corresponding (additively corrected) test statistic $\widehat{\psi}_T(u,h) / \widehat{\sigma} - \lambda(h)$. The dashed line specifies the critical value $q_T(\alpha)$, where $\alpha = 0.05$ as already mentioned above. According to Proposition \ref{prop-test-3}, we can make the following simultaneous confidence statement about the collection of minimal intervals in $\Pi_T^+$. We can claim, with confidence of about $95\%$, that the trend function $m$ has some increase on each minimal interval. More specifically, we can claim with this confidence that there has been some upward movement in the trend both in the period from around $1680$ to $1740$ and in the period from about $1870$ onwards. Hence, our test in particular provides evidence that there has been some warming trend in the period over approximately the last $150$ years. On the other hand, as the set $\Pi_T^-$ is empty, there is no evidence of any downward movement of the trend.

\bibliographystyle{ims}
{\small
\setlength{\bibsep}{0.55em}
\bibliography{bibliography}}

\newpage
\renewcommand{\baselinestretch}{1.0}\normalsize

\headingSupplement{Supplement to}{``Multiscale Inference and}{Long-Run Variance Estimation}{in Nonparametric Regression}{with Time Series Errors''}
\authors{Marina Khismatullina}{University of Bonn}{Michael Vogt}{University of Bonn} 

\vspace{-1.0cm}

\renewcommand{\abstractname}{}
\begin{abstract}
\noindent In this supplement, we provide the technical details and proofs that are omitted in the paper. In addition, we report the results of some robustness checks which complement the simulation exercises in Section \ref{sec-sim} of the paper. 
\end{abstract}

\renewcommand{\baselinestretch}{1.2}\normalsize
\renewcommand{\theequation}{S.\arabic{equation}}
\def\thesection{S.\arabic{section}}
\def\thefigure{S.\arabic{figure}}
\def\thetable{S.\arabic{table}}
\setcounter{section}{0}
\setcounter{equation}{0}
\setcounter{figure}{0}
\setcounter{table}{0}
\allowdisplaybreaks[4]

\section{Proofs of the results from Section \ref{sec-method}}\label{sec-supp-proofs1}

In this section, we prove the theoretical results from Section \ref{sec-method}. We use the following notation: The symbol $C$ denotes a universal real constant which may take a different value on each occurrence. For $a,b \in \reals$, we write $a_+ = \max \{0,a\}$ and $a \vee b = \max\{a,b\}$. For any set $A$, the symbol $|A|$ denotes the cardinality of $A$. The notation $X \stackrel{\mathcal{D}}{=} Y$ means that the two random variables $X$ and $Y$ have the same distribution. Finally, $f_0(\cdot)$ and $F_0(\cdot)$ denote the density and distribution function of the standard normal distribution, respectively.

\subsection*{Auxiliary results using strong approximation theory}

The main purpose of this section is to prove that there is a version of the multiscale statistic $\widehat{\Phi}_T$ defined in \eqref{Phi-hat-statistic} which is close to a Gaussian statistic whose distribution is known. More specifically, we prove the following result. 
\begin{propA}\label{propA-strong-approx}
Under the conditions of Theorem \ref{theo-stat}, there exist statistics $\widetilde{\Phi}_T$ for $T = 1,2,\ldots$ with the following two properties: (i) $\widetilde{\Phi}_T$ has the same distribution as $\widehat{\Phi}_T$ for any $T$, and (ii)
\[ \big| \widetilde{\Phi}_T - \Phi_T \big| = o_p \Big( \frac{T^{1/q}}{\sqrt{T h_{\min}}} + \rho_T \sqrt{\log T} \Big), \]
where $\Phi_T$ is a Gaussian statistic as defined in \eqref{Phi-statistic}. 
\end{propA}
\begin{proof}[\textnormal{\textbf{Proof of Proposition \ref{propA-strong-approx}}}] 
For the proof, we draw on strong approximation theory for stationary processes $\{\varepsilon_t\}$ that fulfill the conditions \ref{C-err1}--\ref{C-err3}. By Theorem 2.1 and Corollary 2.1 in \cite{BerkesLiuWu2014}, the following strong approximation result holds true: On a richer probability space, there exist a standard Brownian motion $\mathbb{B}$ and a sequence $\{ \widetilde{\varepsilon}_t: t \in \naturals \}$ such that $[\widetilde{\varepsilon}_1,\ldots,\widetilde{\varepsilon}_T] \stackrel{\mathcal{D}}{=} [\varepsilon_1,\ldots,\varepsilon_T]$ for each $T$ and 
\begin{equation}\label{eq-strongapprox-dep}
\max_{1 \le t \le T} \Big| \sum\limits_{s=1}^t \widetilde{\varepsilon}_s - \sigma \mathbb{B}(t) \Big| = o\big( T^{1/q} \big) \quad \text{a.s.},  
\end{equation}
where $\sigma^2 = \sum_{k \in \integers} \cov(\varepsilon_0, \varepsilon_k)$ denotes the long-run error variance. To apply this result, we define 
\[ \widetilde{\Phi}_T = \max_{(u,h) \in \mathcal{G}_T} \Big\{ \Big|\frac{\widetilde{\phi}_T(u,h)}{\widetilde{\sigma}}\Big| - \lambda(h) \Big\}, \]
where $\widetilde{\phi}_T(u,h) = \sum\nolimits_{t=1}^T w_{t,T}(u,h) \widetilde{\varepsilon}_t$ and $\widetilde{\sigma}^2$ is the same estimator as $\widehat{\sigma}^2$ with $Y_{t,T} = m(t/T) + \varepsilon_t$ replaced by $\widetilde{Y}_{t,T} = m(t/T) + \widetilde{\varepsilon}_t$ for $1 \le t \le T$. In addition, we let
\begin{align*}
\Phi_T & = \max_{(u,h) \in \mathcal{G}_T} \Big\{ \Big|\frac{\phi_T(u,h)}{\sigma}\Big| - \lambda(h) \Big\} \\
\Phi_T^{\diamond} & = \max_{(u,h) \in \mathcal{G}_T} \Big\{ \Big|\frac{\phi_T(u,h)}{\widetilde{\sigma}}\Big| - \lambda(h) \Big\} 
\end{align*}
with $\phi_T(u,h) = \sum\nolimits_{t=1}^T w_{t,T}(u,h) \sigma Z_t$ and $Z_t = \mathbb{B}(t) - \mathbb{B}(t-1)$. With this notation, we can write 
\begin{equation}\label{eq-strongapprox-bound1}
\big| \widetilde{\Phi}_T - \Phi_T \big| \le \big| \widetilde{\Phi}_T - \Phi_T^{\diamond} \big| + \big| \Phi_T^{\diamond} - \Phi_T \big| = \big| \widetilde{\Phi}_T - \Phi_T^{\diamond} \big| + o_p \big( \rho_T \sqrt{\log T} \big), 
\end{equation}
where the last equality follows by taking into account that $\phi_T(u,h) \sim \normal(0,\sigma^2)$ for all $(u,h) \in \mathcal{G}_T$, $|\mathcal{G}_T| = O(T^\theta)$ for some large but fixed constant $\theta$ and $\widetilde{\sigma}^2 = \sigma^2 + o_p(\rho_T)$. Straightforward calculations yield that 
\[ \big| \widetilde{\Phi}_T - \Phi_T^{\diamond} \big| \le \widetilde{\sigma}^{-1} \max_{(u,h) \in \mathcal{G}_T} \big| \widetilde{\phi}_T(u,h) - \phi_T(u,h) \big|. \]
Using summation by parts, we further obtain that 
\begin{align*}
\big| \widetilde{\phi}_T(u,h) - \phi_T(u,h) \big| 
 & \le W_T(u,h) \max_{1 \le t \le T} \Big| \sum\limits_{s=1}^t \widetilde{\varepsilon}_s - \sigma \sum\limits_{s=1}^t \big\{ \mathbb{B}(s) - \mathbb{B}(s-1) \big\} \Big| \\
 & = W_T(u,h) \max_{1 \le t \le T} \Big| \sum\limits_{s=1}^t \widetilde{\varepsilon}_s - \sigma \mathbb{B}(t) \Big|,
\end{align*}
where
\[ W_T(u,h) = \sum\limits_{t=1}^{T-1} |w_{t+1,T}(u,h) - w_{t,T}(u,h)| + |w_{T,T}(u,h)|. \]
Standard arguments show that $\max_{(u,h) \in \mathcal{G}_T} W_T(u,h) = O( 1/\sqrt{Th_{\min}} )$. Applying the strong approximation result \eqref{eq-strongapprox-dep}, we can thus infer that 
\begin{align}
\big| \widetilde{\Phi}_T - \Phi_T^{\diamond} \big| 
 & \le \widetilde{\sigma}^{-1} \max_{(u,h) \in \mathcal{G}_T} \big| \widetilde{\phi}_T(u,h) - \phi_T(u,h) \big| \nonumber \\
 & \le \widetilde{\sigma}^{-1} \max_{(u,h) \in \mathcal{G}_T} W_T(u,h) \max_{1 \le t \le T} \Big| \sum\limits_{s=1}^t \widetilde{\varepsilon}_s - \sigma \mathbb{B}(t) \Big| 
   = o_p \Big( \frac{T^{1/q}}{\sqrt{Th_{\min}}} \Big). \label{eq-strongapprox-bound2}
\end{align}
Plugging \eqref{eq-strongapprox-bound2} into \eqref{eq-strongapprox-bound1} completes the proof.
\end{proof}

\subsection*{Auxiliary results using anti-concentration bounds}

In this section, we establish some properties of the Gaussian statistic $\Phi_T$ defined in \eqref{Phi-statistic}. We in particular show that $\Phi_T$ does not concentrate too strongly in small regions of the form $[x-\delta_T,x+\delta_T]$ with $\delta_T$ converging to zero.  
\begin{propA}\label{propA-anticon}
Under the conditions of Theorem \ref{theo-stat}, it holds that 
\[ \sup_{x \in \reals} \pr \Big( | \Phi_T - x | \le \delta_T \Big) = o(1), \]
where $\delta_T = T^{1/q} / \sqrt{T h_{\min}} + \rho_T \sqrt{\log T}$.
\end{propA}
\begin{proof}[\textnormal{\textbf{Proof of Proposition \ref{propA-anticon}}}] 
The main technical tool for proving Proposition \ref{propA-anticon} are anti-concentration bounds for Gaussian random vectors. The following proposition slightly generalizes anti-concentration results derived in \cite{Chernozhukov2015}, in particular Theorem 3 therein. 
\begin{propA}\label{theo-anticon}
Let $(X_1,\ldots,X_p)^\top$ be a Gaussian random vector in $\reals^p$ with $\ex[X_j] = \mu_j$ and $\var(X_j) = \sigma_j^2 > 0$ for $1 \le j \le p$. Define $\overline{\mu} = \max_{1 \le j \le p} |\mu_j|$ together with $\underline{\sigma} = \min_{1 \le j \le p} \sigma_j$ and $\overline{\sigma} = \max_{1 \le j \le p} \sigma_j$. Moreover, set $a_p = \ex[ \max_{1 \le j \le p} (X_j-\mu_j)/\sigma_j ]$ and $b_p = \ex[ \max_{1 \le j \le p} (X_j-\mu_j) ]$. For every $\delta > 0$, it holds that
\[ \sup_{x \in \reals} \pr \Big( \big| \max_{1 \le j \le p} X_j - x \big| \le \delta \Big) \le C \delta \big\{ \overline{\mu} + a_p + b_p + \sqrt{1 \vee \log(\underline{\sigma}/\delta)} \big\}, \]
where $C > 0$ depends only on $\underline{\sigma}$ and $\overline{\sigma}$. 
\end{propA} 
The proof of Proposition \ref{theo-anticon} is provided at the end of this section for completeness. To apply Proposition \ref{theo-anticon} to our setting at hand, we introduce the following notation: We write $x = (u,h)$ along with $\mathcal{G}_T = \{ x : x \in \mathcal{G}_T \} = \{x_1,\ldots,x_p\}$, where $p := |\mathcal{G}_T| \le O(T^\theta)$ for some large but fixed $\theta > 0$ by our assumptions. Moreover, for $j = 1,\ldots,p$, we set 
\begin{align*}
X_{2j-1} & = \frac{\phi_T(x_{j1},x_{j2})}{\sigma} - \lambda(x_{j2}) \\
X_{2j} & = -\frac{\phi_T(x_{j1},x_{j2})}{\sigma} - \lambda(x_{j2}) 
\end{align*}
with $x_j = (x_{j1},x_{j2})$. This notation allows us to write
\[ \Phi_T = \max_{1 \le j \le 2p} X_j, \]
where $(X_1,\ldots,X_{2p})^\top$ is a Gaussian random vector with the following properties: (i) $\mu_j := \ex[X_j] = - \lambda(x_{j2})$ and thus $\overline{\mu} = \max_{1 \le j \le 2p} |\mu_j| \le C \sqrt{\log T}$, and (ii) $\sigma_j^2 := \var(X_j) = 1$ for all $j$. Since $\sigma_j = 1$ for all $j$, it holds that $a_{2p} = b_{2p}$. Moreover, as the variables $(X_j - \mu_j)/\sigma_j$ are standard normal, we have that $a_{2p} = b_{2p} \le \sqrt{2 \log (2p)} \le C \sqrt{\log T}$. With this notation at hand, we can apply Proposition \ref{theo-anticon} to obtain that 
\[ \sup_{x \in \reals} \pr \Big( \big| \Phi_T - x \big| \le \delta_T \Big) \le C \delta_T \Big[ \sqrt{\log T} + \sqrt{ \log(1/\delta_T) } \Big] = o(1) \]
with $\delta_T = T^{1/q} / \sqrt{T h_{\min}} + \rho_T \sqrt{\log T}$, which is the statement of Proposition \ref{propA-anticon}.
\end{proof}

\subsection*{Proof of Theorem \ref{theo-stat}}

To prove Theorem \ref{theo-stat}, we make use of the two auxiliary results derived above. By Proposition \ref{propA-strong-approx}, there exist statistics $\widetilde{\Phi}_T$ for $T = 1,2,\ldots$ which are distributed as $\widehat{\Phi}_T$ for any $T \ge 1$ and which have the property that 
\begin{equation}\label{statement-propA-strong-approx}
\big| \widetilde{\Phi}_T - \Phi_T \big| = o_p \Big( \frac{T^{1/q}}{\sqrt{T h_{\min}}} + \rho_T \sqrt{\log T} \Big), 
\end{equation}
where $\Phi_T$ is a Gaussian statistic as defined in \eqref{Phi-statistic}. The approximation result \eqref{statement-propA-strong-approx} allows us to replace the multiscale statistic $\widehat{\Phi}_T$ by an identically distributed version $\widetilde{\Phi}_T$ which is close to the Gaussian statistic $\Phi_T$. In the next step, we show that  
\begin{equation}\label{eq-theo-stat-step2}
\sup_{x \in \reals} \big| \pr(\widetilde{\Phi}_T \le x) - \pr(\Phi_T \le x) \big| = o(1), 
\end{equation}
which immediately implies the statement of Theorem \ref{theo-stat}. For the proof of \eqref{eq-theo-stat-step2}, we use the following simple lemma: 
\begin{lemmaA}\label{lemma1-theo-stat}
Let $V_T$ and $W_T$ be real-valued random variables for $T = 1,2,\ldots$ such that $V_T - W_T = o_p(\delta_T)$ with some $\delta_T = o(1)$. If 
\begin{equation}\label{eq-lemma1-cond}
\sup_{x \in \reals} \pr(|V_T - x| \le \delta_T) = o(1), 
\end{equation}
then 
\begin{equation}\label{eq-lemma1-statement}
\sup_{x \in \reals} \big| \pr(V_T \le x) - \pr(W_T \le x) \big| = o(1). 
\end{equation}
\end{lemmaA}
The statement of Lemma \ref{lemma1-theo-stat} can be summarized as follows: If $W_T$ can be approximated by $V_T$ in the sense that $V_T - W_T = o_p(\delta_T)$ and if $V_T$ does not concentrate too strongly in small regions of the form $[x - \delta_T,x+\delta_T]$ as assumed in \eqref{eq-lemma1-cond}, then the distribution of $W_T$ can be approximated by that of $V_T$ in the sense of \eqref{eq-lemma1-statement}.
\begin{proof}[\textnormal{\textbf{Proof of Lemma \ref{lemma1-theo-stat}}}] 
It holds that 
\begin{align*}
 & \big| \pr(V_T \le x) - \pr(W_T \le x) \big| \\
 & = \big| \ex \big[ 1(V_T \le x) - 1(W_T \le x) \big] \big| \\
 & \le \big| \ex \big[ \big\{ 1(V_T \le x) - 1(W_T \le x) \big\} 1(|V_T - W_T| \le \delta_T) \big] \big| + \big| \ex \big[ 1(|V_T - W_T| > \delta_T) \big] \big| \\
 & \le \ex \big[ 1(|V_T - x| \le \delta_T, |V_T - W_T| \le \delta_T) \big] + o(1) \\
 & \le \pr (|V_T - x| \le \delta_T) + o(1). \qedhere
\end{align*}
\end{proof}
We now apply this lemma with $V_T = \Phi_T$, $W_T = \widetilde{\Phi}_T$ and $\delta_T = T^{1/q} / \sqrt{T h_{\min}} + \rho_T \sqrt{\log T}$: From \eqref{statement-propA-strong-approx}, we already know that $\widetilde{\Phi}_T - \Phi_T = o_p(\delta_T)$. Moreover, by Proposition \ref{propA-anticon}, it holds that 
\begin{equation}\label{statement-propA-anticon}
\sup_{x \in \reals} \pr \Big( | \Phi_T - x | \le \delta_T \Big) = o(1). 
\end{equation}
Hence, the conditions of Lemma \ref{lemma1-theo-stat} are satisfied. Applying the lemma, we obtain \eqref{eq-theo-stat-step2}, which completes the proof of Theorem \ref{theo-stat}.

\subsection*{Proof of Proposition \ref{prop-test-2}}

To start with, we introduce the notation $\widehat{\psi}_T(u,h) = \widehat{\psi}_T^A(u,h) + \widehat{\psi}_T^B(u,h)$ with $\widehat{\psi}_T^A(u,h) = \sum\nolimits_{t=1}^T w_{t,T}(u,h) \varepsilon_t$ and $\widehat{\psi}_T^B(u,h) = \sum\nolimits_{t=1}^T w_{t,T}(u,h) m_T(\frac{t}{T})$. By assumption, there exists $(u_0,h_0) \in \mathcal{G}_T$ with $[u_0-h_0,u_0+h_0] \subseteq [0,1]$ such that $m_T^\prime(w) \ge c_T \sqrt{\log T/(Th_0^3)}$ for all $w \in [u_0-h_0,u_0+h_0]$. (The case that $-m_T^\prime(w) \ge c_T \sqrt{\log T/(Th_0^3)}$ for all $w$ can be treated analogously.) Below, we prove that under this assumption, 
\begin{equation}\label{eq-psiB}
\widehat{\psi}_T^B(u_0,h_0) \ge \frac{\kappa c_T \sqrt{\log T}}{2} 
\end{equation}
for sufficiently large $T$, where $\kappa = (\int K(\varphi) \varphi^2 d\varphi) / (\int K^2(\varphi) \varphi^2 d\varphi)^{1/2}$. Moreover, by arguments very similar to those for the proof of Proposition \ref{propA-strong-approx}, it follows that
\begin{equation}\label{eq-psiA}
\max_{(u,h) \in \mathcal{G}_T} |\widehat{\psi}_T^A(u,h)| = O_p(\sqrt{\log T}). 
\end{equation}
With the help of \eqref{eq-psiB}, \eqref{eq-psiA} and the fact that $\lambda(h) \le \lambda(h_{\min}) \le C \sqrt{\log T}$, we can infer that
\begin{align}
\widehat{\Psi}_T 
 & \ge \max_{(u,h) \in \mathcal{G}_T} \frac{|\widehat{\psi}_T^B(u,h)|}{\widehat{\sigma}} - \max_{(u,h) \in \mathcal{G}_T} \Big\{ \frac{|\widehat{\psi}_T^A(u,h)|}{\widehat{\sigma}} + \lambda(h) \Big\} \nonumber \\
 & = \max_{(u,h) \in \mathcal{G}_T} \frac{|\widehat{\psi}_T^B(u,h)|}{\widehat{\sigma}} + O_p(\sqrt{\log T}) \nonumber \\
 & \ge \frac{\kappa c_T \sqrt{\log T}}{2 \widehat{\sigma}} + O_p(\sqrt{\log T}) \label{eq-proof-prop-test-2-conclusion}
\end{align}  
for sufficiently large $T$. Since $q_T(\alpha) = O(\sqrt{\log T})$ for any fixed $\alpha \in (0,1)$, \eqref{eq-proof-prop-test-2-conclusion} immediately yields that $\pr(\widehat{\Psi}_T \le q_T(\alpha)) = o(1)$, which is the statement of Proposition \ref{prop-test-2}.

\begin{proof}[\textnormal{\textbf{Proof of (\ref{eq-psiB})}}] 
Write $m_T(\frac{t}{T}) = m_T(u_0) + m_T^\prime(\xi_{u_0,t,T})(\frac{t}{T} - u_0)$, where $\xi_{u_0,t,T}$ is an intermediate point between $u_0$ and $t/T$. The local linear weights $w_{t,T}(u_0,h_0)$ are constructed such that $\sum_{t=1}^T w_{t,T}(u_0,h_0) = 0$. We thus obtain that 
\begin{equation}\label{eq-psiB-exp}
\widehat{\psi}_T^B(u_0,h_0) = \sum\limits_{t=1}^T w_{t,T}(u_0,h_0) \Big(\frac{\frac{t}{T} - u_0}{h_0}\Big) h_0 m_T^\prime(\xi_{u_0,t,T}).
\end{equation}
Moreover, since the kernel $K$ is symmetric and $u_0 = t/T$ for some $t$, it holds that $S_{T,1}(u_0,h_0) = 0$, which in turn implies that 
\begin{align}
w_{t,T} & (u_0,h_0) \Big( \frac{\frac{t}{T} - u_0}{h_0} \Big) \nonumber \\* & = K\Big(\frac{\frac{t}{T}-u_0}{h_0}\Big) \Big( \frac{\frac{t}{T} - u_0}{h_0} \Big)^2 \Big/ \Big\{ \sum_{t=1}^T K^2\Big(\frac{\frac{t}{T}-u_0}{h_0}\Big) \Big( \frac{\frac{t}{T} - u_0}{h_0} \Big)^2 \Big\}^{1/2} \ge 0. \label{weight-interior} 
\end{align}
From \eqref{eq-psiB-exp}, \eqref{weight-interior} and the assumption that $m_T^\prime(w) \ge c_T \sqrt{\log T/(Th_0^3)}$ for all $w \in [u_0-h_0,u_0+h_0]$, we get that 
\begin{equation}\label{eq-psiB-bound}
\widehat{\psi}_T^B(u_0,h_0) \ge c_T \sqrt{\frac{\log T}{T h_0}} \sum\limits_{t=1}^T w_{t,T}(u_0,h_0) \Big( \frac{\frac{t}{T} - u_0}{h_0} \Big).  
\end{equation}
Standard calculations exploiting the Lipschitz continuity of the kernel $K$ show that for any $(u,h) \in \mathcal{G}_T$ and any given natural number $\ell$, 
\begin{equation}\label{eq-riemann-sum}
\Big| \frac{1}{Th} \sum\limits_{t=1}^T K\Big(\frac{\frac{t}{T}-u}{h}\Big) \Big(\frac{\frac{t}{T}-u}{h}\Big)^\ell - \int_0^1 \frac{1}{h} K\Big(\frac{w-u}{h}\Big) \Big(\frac{w-u}{h}\Big)^\ell dw \Big| \le \frac{C}{Th}, 
\end{equation}
where the constant $C$ does not depend on $u$, $h$ and $T$. With the help of \eqref{weight-interior} and \eqref{eq-riemann-sum}, we obtain that for any $(u,h) \in \mathcal{G}_T$ with $[u-h,u+h] \subseteq [0,1]$, 
\begin{equation}\label{eq-psiB-weight-bound}
\Big| \sum\limits_{t=1}^T w_{t,T}(u,h) \Big(\frac{\frac{t}{T} - u}{h}\Big) - \kappa \sqrt{Th} \Big| \le \frac{C}{\sqrt{Th}}, 
\end{equation}
where the constant $C$ does once again not depend on $u$, $h$ and $T$. \eqref{eq-psiB-weight-bound} implies that $\sum\nolimits_{t=1}^T w_{t,T}(u,h) (\frac{t}{T} - u)/h \ge \kappa \sqrt{Th} / 2$ for sufficiently large $T$ and any $(u,h) \in \mathcal{G}_T$ with $[u-h,u+h] \subseteq [0,1]$. Using this together with \eqref{eq-psiB-bound}, we immediately obtain \eqref{eq-psiB}.
\end{proof}

\subsection*{Proof of Proposition \ref{prop-test-3}}

In what follows, we show that 
\begin{equation}\label{claim-prop-test-3}
\pr(E_T^+) \ge (1-\alpha) + o(1). 
\end{equation}
The other statements of Proposition \ref{prop-test-3} can be verified by analogous arguments. \eqref{claim-prop-test-3} is a consequence of the following two observations:  
\begin{enumerate}[label=(\roman*),leftmargin=0.75cm]

\item For all $(u,h) \in \mathcal{G}_T$ with   
\[ \Big|\frac{\widehat{\psi}_T(u,h) - \ex \widehat{\psi}_T(u,h)}{\widehat{\sigma}}\Big| - \lambda(h) \le q_T(\alpha) \quad \text{and} \quad \frac{\widehat{\psi}_T(u,h)}{\widehat{\sigma}} - \lambda(h) > q_T(\alpha), \]
it holds that $\ex[\widehat{\psi}_T(u,h)] > 0$. 

\item For all $(u,h) \in \mathcal{G}_T$ with $[u-h,u+h] \subseteq [0,1]$,  $\ex[\widehat{\psi}_T(u,h)] > 0$ implies that $m^\prime(v) > 0$ for some $v \in [u-h,u+h]$. 

\end{enumerate}
Observation (i) is trivial, (ii) can be seen as follows: Let $(u,h)$ be any point with $(u,h) \in \mathcal{G}_T$ and $[u-h,u+h] \subseteq [0,1]$. It holds that $\ex[\widehat{\psi}_T(u,h)] = \widehat{\psi}_T^B(u,h)$, where $\widehat{\psi}_T^B(u,h)$ has been defined in the proof of Proposition \ref{prop-test-2}. As already shown in \eqref{eq-psiB-exp},  
\[ \widehat{\psi}_T^B(u,h) = \sum\limits_{t=1}^T w_{t,T}(u,h) \Big( \frac{\frac{t}{T} - u}{h} \Big) \, h m^\prime(\xi_{u,t,T}), \]
where $\xi_{u,t,T}$ is some intermediate point between $u$ and $t/T$. Moreover, by \eqref{weight-interior}, it holds that $w_{t,T}(u,h) (\frac{t}{T} - u)/h \ge 0$ for any $t$. Hence, $\ex[\widehat{\psi}_T(u,h)] = \widehat{\psi}_T^B(u,h)$ can only take a positive value if $m^\prime(v) > 0$ for some $v \in [u-h,u+h]$.

From observations (i) and (ii), we can draw the following conclusions: On the event 
\[ \big\{ \widehat{\Phi}_T \le q_T(\alpha) \big\} = \Big\{ \max_{(u,h) \in \mathcal{G}_T} \Big( \Big|\frac{\widehat{\psi}_T(u,h) - \ex \widehat{\psi}_T(u,h)}{\widehat{\sigma}}\Big| - \lambda(h) \Big) \le q_T(\alpha) \Big\}, \]
it holds that for all $(u,h) \in \mathcal{A}_T^+$ with $[u-h,u+h] \subseteq [0,1]$, $m^\prime(v) > 0$ for some $v \in I_{u,h} = [u-h,u+h]$. We thus obtain that $\{ \widehat{\Phi}_T \le q_T(\alpha) \} \subseteq E_T^+$. This in turn implies that 
\[ \pr(E_T^+) \ge \pr \big(  \widehat{\Phi}_T \le q_T(\alpha) \big) = (1-\alpha) + o(1), \]
where the last equality holds by Theorem \ref{theo-stat}.

\newpage
\subsection*{Proof of Proposition \ref{theo-anticon}}

The proof makes use of the following three lemmas, which correspond to Lemmas 5--7 in \cite{Chernozhukov2015}. 
\begin{lemmaA}\label{lemma1-anticon}
Let $(W_1,\ldots,W_p)^\top$ be a (not necessarily centred) Gaussian random vector in $\reals^p$ with $\var(W_j) = 1$ for all $1 \le j \le p$. Suppose that $\textnormal{Corr}(W_j,W_k) < 1$ whenever $j \ne k$. Then the distribution of $\max_{1 \le j \le p} W_j$ is absolutely continuous with respect to Lebesgue measure and a version of the density is given by 
\[ f(x) = f_0(x) \sum\limits_{j=1}^p e^{\ex[W_j]x - \ex[W_j]^2/2} \, \pr \big(W_k \le x \text{ for all } k \ne j \, \big| \, W_j = x \big). \]
\end{lemmaA}
\begin{lemmaA}\label{lemma2-anticon}
Let $(W_0,W_1,\ldots,W_p)^\top$ be a (not necessarily centred) Gaussian random vector with $\var(W_j) = 1$ for all $0 \le j \le p$. Suppose that $\ex[W_0] \ge 0$. Then the map 
\[ x \mapsto  e^{\ex[W_0]x - \ex[W_0]^2/2} \, \pr \big(W_j \le x \text{ for } 1 \le j \le p \, \big| \, W_0 = x \big) \]
is non-decreasing on $\reals$. 
\end{lemmaA}
\begin{lemmaA}\label{lemma3-anticon}
Let $(X_1,\ldots,X_p)^\top$ be a centred Gaussian random vector in $\reals^p$ with $\max_{1 \le j \le p} \ex[X_j^2] \le \sigma_X^2$ for some $\sigma_X^2 > 0$. Then for any $r > 0$, 
\[ \pr \Big( \max_{1 \le j \le p} X_j \ge \ex \Big[ \max_{1 \le j \le p} X_j \Big] + r \Big) \le e^{-r^2/(2\sigma_X^2)}. \]
\end{lemmaA} 
The proof of Lemmas \ref{lemma1-anticon} and \ref{lemma2-anticon} can be found in \cite{Chernozhukov2015}. Lemma \ref{lemma3-anticon} is a standard result on Gaussian concentration whose proof is given e.g.\ in \cite{Ledoux2001}; see Theorem 7.1 therein. We now closely follow the arguments for the proof of Theorem 3 in \cite{Chernozhukov2015}. The proof splits up into three steps. 
\vspace{7pt}

\textit{Step 1.} Pick any $x \ge 0$ and set 
\[ W_j = \frac{X_j - x}{\sigma_j} + \frac{\overline{\mu} + x}{\underline{\sigma}}. \]
By construction, $\ex[W_j] \ge 0$ and $\var(W_j) = 1$. Defining $Z = \max_{1 \le j \le p} W_j$, it holds that  
\begin{align*}
\pr \Big( \Big| \max_{1 \le j \le p} X_j - x \Big| \le \delta \Big) 
 & \le \pr \Big( \Big| \max_{1 \le j \le p} \frac{X_j - x}{\sigma_j} \Big| \le \frac{\delta}{\underline{\sigma}} \Big) \\
 & \le \sup_{y \in \reals} \pr \Big( \Big| \max_{1 \le j \le p} \frac{X_j - x}{\sigma_j} + \frac{\overline{\mu} + x}{\underline{\sigma}} - y \Big| \le \frac{\delta}{\underline{\sigma}} \Big) \\
 & = \sup_{y \in \reals} \pr \Big( |Z - y| \le \frac{\delta}{\underline{\sigma}} \Big). 
\end{align*}

\textit{Step 2.} We now bound the density of $Z$. Without loss of generality, we assume that $\text{Corr}(W_j,W_k) < 1$ for $k \ne j$. The marginal distribution of $W_j$ is $\normal(\nu_j,1)$ with $\nu_j = \ex[W_j] = (\mu_j/\sigma_j + \overline{\mu}/{\underline{\sigma}}) + (x/\underline{\sigma} - x/\sigma_j) \ge 0$. Hence, by Lemmas \ref{lemma1-anticon} and \ref{lemma2-anticon}, the random variable $Z$ has a density of the form
\begin{equation}\label{eq-dens-Z}
f_p(z) = f_0(z) G_p(z), 
\end{equation}
where the map $z \mapsto G_p(z)$ is non-decreasing. Define $\overline{Z} = \max_{1 \le j \le p} (W_j - \ex[W_j])$ and set $\overline{z} = 2 \overline{\mu}/\underline{\sigma} + x(1/\underline{\sigma} - 1/\overline{\sigma})$ such that $\ex[W_j] \le \overline{z}$ for any $1 \le j \le p$. With these definitions at hand, we obtain that  
\begin{align*}
\int_z^{\infty} f_0(u)du \, G_p(z) & \le \int_z^{\infty} f_0(u) G_p(u) du = \pr(Z > z) \\ 
 & \le P(\overline{Z} > z - \overline{z}) \le \exp \Big( - \frac{(z - \overline{z} - \ex[\overline{Z}])^2_+}{2} \Big), 
\end{align*}
where the last inequality follows from Lemma \ref{lemma3-anticon}. Since $W_j - \ex[W_j] = (X_j - \mu_j)/\sigma_j$, it holds that 
\[ \ex[\overline{Z}] = \ex \Big[ \max_{1 \le j \le p} \Big\{ \frac{X_j-\mu_j}{\sigma_j} \Big\} \Big] =: a_p. \]
Hence, for every $z \in \reals$, 
\begin{equation}\label{eq-bound-Gp}
G_p(z) \le \frac{1}{1 - F_0(z)} \exp\Big( - \frac{(z - \overline{z} - a_p)_+^2}{2} \Big). 
\end{equation}
Mill's inequality states that for $z > 0$, 
\[ z \le \frac{f_0(z)}{1-F_0(z)} \le z \frac{1+z^2}{z^2}. \]
Since $(1+z^2)/z^2 \le 2$ for $z \ge 1$ and $f_0(z)/\{1-F_0(z)\} \le 1.53 \le 2$ for $z \in (-\infty,1)$, we can infer that
\[ \frac{f_0(z)}{1-F_0(z)} \le 2 (z \vee 1) \quad \text{for any } z \in \reals. \]
This together with \eqref{eq-dens-Z} and \eqref{eq-bound-Gp} yields that
\[ f_p(z) \le 2 (z \vee 1)  \exp\Big( - \frac{(z - \overline{z} - a_p)_+^2}{2} \Big) \quad \text{for any } z \in \reals. \]

\textit{Step 3.} By Step 2, we get that for any $y \in \reals$ and $u > 0$, 
\[ \pr( |Z - y| \le u) = \int_{y - u}^{y + u} f_p(z) dz \le 2u \max_{z \in [y-u,y+u]} f_p(z) \le 4u (\overline{z} + a_p + 1), \] 
where the last inequality follows from the fact that the map $z \mapsto z e^{-(z-a)^2/2}$ (with $a > 0$) is non-increasing on $[a+1,\infty)$. Combining this bound with Step 1, we further obtain that for any $x \ge 0$ and $\delta > 0$, 
\begin{equation}\label{eq-bound1-Levy}
\pr \Big( \Big| \max_{1 \le j \le p} X_j - x \Big| \le \delta \Big) \le 4\delta \Big\{ \frac{2\overline{\mu}}{\underline{\sigma}} + |x| \Big(\frac{1}{\underline{\sigma}} - \frac{1}{\overline{\sigma}}\Big) + a_p + 1 \Big\} \big/ \underline{\sigma}. 
\end{equation} 
This inequality also holds for $x < 0$ by an analogous argument, and hence for all $x \in \reals$.

Now let $0 < \delta \le \underline{\sigma}$ and define $b_p = \ex \max_{1 \le j \le p} \{X_j - \mu_j\}$. For any $|x| \le \delta + \overline{\mu} + b_p + \overline{\sigma} \sqrt{2\log(\underline{\sigma}/\delta)}$, \eqref{eq-bound1-Levy} yields that 
\begin{align}
\pr \Big( \Big| \max_{1 \le j \le p} X_j - x \Big| \le \delta \Big) 
 & \le \frac{4 \delta}{\underline{\sigma}} \Big\{ \overline{\mu} \Big( \frac{3}{\underline{\sigma}} - \frac{1}{\overline{\sigma}} \Big) + a_p + \Big( \frac{1}{\underline{\sigma}} - \frac{1}{\overline{\sigma}} \Big) b_p \nonumber \\ & \phantom{\le \frac{4 \delta}{\underline{\sigma}} \Big\{} + \Big( \frac{\overline{\sigma}}{\underline{\sigma}} - 1 \Big) \sqrt{2\log\Big(\frac{\underline{\sigma}}{\delta}\Big)} + 2 - \frac{\underline{\sigma}}{\overline{\sigma}} \Big\} \nonumber \\[0.2cm]
 & \le C \delta \big\{ \overline{\mu} + a_p + b_p + \sqrt{1 \vee \log(\underline{\sigma}/\delta)} \big\} \label{eq-bound2-Levy}
\end{align}
with a sufficiently large constant $C > 0$ that depends only on $\underline{\sigma}$ and $\overline{\sigma}$. For $|x| \ge \delta + \overline{\mu} + b_p + \overline{\sigma}\sqrt{2\log(\underline{\sigma}/\delta)}$, we obtain that 
\begin{equation}\label{eq-bound3-Levy}
\pr \Big( \Big| \max_{1 \le j \le p} X_j - x \Big| \le \delta \Big) \le \frac{\delta}{\underline{\sigma}}, 
\end{equation}
which can be seen as follows: If $x > \delta + \overline{\mu}$, then $|\max_j X_j - x| \le \delta$ implies that $|x| - \delta \le \max_j X_j \le \max_j \{ X_j - \mu_j \} + \overline{\mu}$ and thus $\max_j \{ X_j - \mu_j \} \ge |x| - \delta - \overline{\mu}$. Hence, it holds that 
\begin{equation}\label{eq-bound3-Levy-prep1}
\pr \Big( \Big| \max_{1 \le j \le p} X_j - x \Big| \le \delta \Big) \le \pr \Big( \max_{1 \le j \le p} \big\{ X_j - \mu_j \} \ge |x| - \delta - \overline{\mu} \Big). 
\end{equation}
If $x < - (\delta + \overline{\mu})$, then $|\max_j X_j - x| \le \delta$ implies that $\max_j \{ X_j - \mu_j \} \le -|x| + \delta + \overline{\mu}$. Hence, in this case,
\begin{align}
\pr \Big( \Big| \max_{1 \le j \le p} X_j - x \Big| \le \delta \Big) 
 & \le \pr \Big( \max_{1 \le j \le p} \big\{ X_j - \mu_j \} \le -|x| + \delta + \overline{\mu} \Big) \nonumber \\
 & \le \pr \Big( \max_{1 \le j \le p} \big\{ X_j - \mu_j \} \ge |x| - \delta - \overline{\mu} \Big), \label{eq-bound3-Levy-prep2}
\end{align}
where the last inequality follows from the fact that for centred Gaussian random variables $V_j$ and $v > 0$, $\pr(\max_j V_j \le -v) \le \pr(V_1 \le -v) = P(V_1 \ge v) \le \pr(\max_j V_j \ge v)$. With \eqref{eq-bound3-Levy-prep1} and \eqref{eq-bound3-Levy-prep2}, we obtain that for any $|x| \ge \delta + \overline{\mu} + b_p + \overline{\sigma}\sqrt{2\log(\underline{\sigma}/\delta)}$,
\begin{align*} 
\pr \Big( & \Big| \max_{1 \le j \le p} X_j - x \Big| \le \delta \Big) \le \pr \Big( \max_{1 \le j \le p} \big\{ X_j - \mu_j \} \ge |x| - \delta - \overline{\mu} \Big) \\
 & \le \pr \Big( \max_{1 \le j \le p} \big\{ X_j - \mu_j \big\} \ge \ex \Big[ \max_{1 \le j \le p} \big\{ X_j-\mu_j \big\} \Big] + \overline{\sigma} \sqrt{2\log(\underline{\sigma}/\delta)} \Big) \le \frac{\delta}{\underline{\sigma}}, 
\end{align*}
the last inequality following from Lemma \ref{lemma3-anticon}. To sum up, we have established that for any $0 < \delta \le \underline{\sigma}$ and any $x \in \reals$, 
\begin{equation}\label{claim-prop-anticon}
\pr \Big( \Big| \max_{1 \le j \le p} X_j - x \Big| \le \delta \Big) \le C \delta \big\{ \overline{\mu} + a_p + b_p + \sqrt{1 \vee \log(\underline{\sigma}/\delta)} \big\} 
\end{equation}
with some constant $C > 0$ that does only depend on $\underline{\sigma}$ and $\overline{\sigma}$. For $\delta > \underline{\sigma}$, \eqref{claim-prop-anticon} trivially follows upon setting $C \ge 1/\underline{\sigma}$. This completes the proof.

\section{Proofs of the results from Section \ref{sec-error-var}}\label{sec-supp-proofs2}

In what follows, we prove Proposition \ref{prop-lrv} from Section \ref{sec-error-var}. The notation is the same as in the previous section. In particular, we use the symbol $C$ to denote a generic constant which may take a different value on each occurrence.

\enlargethispage{0.1cm}
\subsection*{Auxiliary results}

To start with, we derive some auxiliary results needed for the proof of Proposition \ref{prop-lrv}. The first lemma analyses the term 
\[ \xi(\ell_1,\ell_2,L) =  \frac{1}{T-L} \sum\limits_{t=L+1}^{T} \varepsilon_{t-\ell_1} \varepsilon_{t-\ell_2}, \]
where $\ell_1,\ell_2$ and $L$ are natural numbers with $0 \le \ell_1, \ell_2 \le L$ that may depend on the sample size $T$, that is, $L = L_T$ as well as $\ell_1 = \ell_{1,T}$ and $\ell_2 = \ell_{2,T}$. 
\begin{lemmaA}\label{lemma-lrv-1} 
For any $L = L_T$ with $L_T/T \rightarrow 0$, it holds that 
\[ \ex \Big[ \big\{ \xi(\ell_1,\ell_2,L) - \gamma_\varepsilon(\ell_2-\ell_1) \big\}^2 \Big] = O(T^{-1}), \]
where $\gamma_\varepsilon(\ell) = \cov(\varepsilon_t,\varepsilon_{t-\ell})$.
\end{lemmaA}
\begin{proof}[\textnormal{\textbf{Proof of Lemma \ref{lemma-lrv-1}}}] 
Since the variables $\varepsilon_t$ have the expansion $\varepsilon_t = \sum_{k=0}^{\infty} c_k \eta_{t-k}$ and $\gamma_\varepsilon(\ell) = (\sum_{k=0}^\infty c_k c_{k+\ell}) \nu^2$, it holds that 
\[ \ex \big[ \xi^2(\ell_1,\ell_2,L)\big] = \frac{1}{(T-L)^2} \sum\limits_{t,t^\prime=L+1}^T \ex \big[ \varepsilon_{t-\ell_1} \varepsilon_{t-\ell_2} \varepsilon_{t^\prime-\ell_1} \varepsilon_{t^\prime-\ell_2} \big], \]
where
\begin{align*} 
 & \ex \big[ \varepsilon_{t-\ell_1} \varepsilon_{t-\ell_2} \varepsilon_{t^\prime-\ell_1} \varepsilon_{t^\prime-\ell_2} \big] \\*
 & = \Big( \sum_{k=0}^{\infty} c_k c_{k+\ell_1-\ell_2} c_{k+t^\prime-t} c_{k+t^\prime-t+\ell_1-\ell_2} \Big) \kappa +  \Big( \sum_{k=0}^{\infty} c_k c_{k+\ell_1-\ell_2} \Big)^2 \nu^4 \\*
 & \quad + \Big( \sum_{k=0}^{\infty} c_k c_{k+t^\prime-t} \Big)^2 \nu^4 + \Big( \sum_{k=0}^{\infty} c_k c_{k+t^\prime-t-\ell_1+\ell_2} \Big) \Big( \sum_{k=0}^{\infty} c_k c_{k+t^\prime-t+\ell_1-\ell_2} \Big) \nu^4 \\
 & = \Big( \sum_{k=0}^{\infty} c_k c_{k+\ell_1-\ell_2} c_{k+t^\prime-t} c_{k+t^\prime-t+\ell_1-\ell_2} \Big) \kappa + \gamma_\varepsilon^2(\ell_1-\ell_2) \\*
 & \quad + \gamma_\varepsilon^2(t^\prime-t) + \gamma_\varepsilon(t^\prime-t-\ell_1+\ell_2) \gamma_\varepsilon(t^\prime-t+\ell_1-\ell_2)
\end{align*}
with $\kappa=\ex[\eta_0^4] - 3 \nu^4$ and $c_k=0$ for $k < 0$. Noting that 
\[ \ex \Big[ \big\{ \xi(\ell_1,\ell_2,L) - \gamma_\varepsilon(\ell_1-\ell_2) \big\}^2 \Big] = \ex \big[\xi^2(\ell_1,\ell_2,L)\big] - \gamma_\varepsilon^2(\ell_1-\ell_2), \] 
we can infer that
\begin{align*} 
 & \ex \Big[ \big\{ \xi(\ell_1,\ell_2,L) - \gamma_\varepsilon(\ell_1-\ell_2) \big\}^2 \Big] \\
 & = \frac{1}{(T-L)^2} \sum\limits_{t,t^\prime=L+1}^T \Big( \sum_{k=0}^{\infty} c_k c_{k+\ell_1-\ell_2} c_{k+t^\prime-t} c_{k+t^\prime-t+\ell_1-\ell_2} \Big) \kappa + \frac{1}{(T-L)^2} \sum\limits_{t,t^\prime=L+1}^T \gamma_\varepsilon^2(t^\prime-t) \\
 & \quad + \frac{1}{(T-L)^2} \sum\limits_{t,t^\prime=L+1}^T  \gamma_\varepsilon(t^\prime-t-\ell_1+\ell_2) \gamma_\varepsilon(t^\prime-t+\ell_1-\ell_2) \\
 & = O(T^{-1}),
\end{align*} 
the last equality following from the fact that the autocovariances $\gamma_\varepsilon(\ell)$ are absolutely summable and the coefficients $c_k$ decay exponentially fast to zero. 
\end{proof}

We next show that the empirical autocovariances 
\[ \widehat{\gamma}_q(\ell) = \frac{1}{T-q} \sum_{t=q+\ell+1}^T \Delta_q Y_{t,T} \, \Delta_q Y_{t-\ell,T} \]
of the process $\{ \Delta_q Y_{t,T} \}$ have the following property. 
\begin{lemmaA}\label{lemma-lrv-2}
For any $q = q_T$ with $q_T/\sqrt{T} \rightarrow 0$ and any $1 \le \ell \le p + 1$, it holds that  
\[ \widehat{\gamma}_q(\ell) - \gamma_q(\ell) = O_p(T^{-1/2}), \]
where $\gamma_q(\ell) = \cov(\Delta_q \varepsilon_t,\Delta_q \varepsilon_{t-\ell})$. 
\end{lemmaA}
\begin{proof}[\textnormal{\textbf{Proof of Lemma \ref{lemma-lrv-2}}}] 
To analyse the term $\widehat{\gamma}_q(\ell)$, we decompose it as follows:  
\[ \widehat{\gamma}_q(\ell) = \widehat{\gamma}_q^*(\ell) + R_A + R_B + R_C, \]
where 
\[ \widehat{\gamma}_q^*(\ell) =  \frac{1}{T-q} \sum\limits_{t=q+\ell+1}^T \Delta_q \varepsilon_t \, \Delta_q \varepsilon_{t-\ell} \]
as well as $R_A = (T-q)^{-1} \sum_{t=q+\ell+1}^T \Delta_q m_t \Delta_q \varepsilon_{t-\ell}$, $R_B = (T-q)^{-1} \sum_{t=q+\ell+1}^T \Delta_q \varepsilon_t \Delta_q m_{t-\ell}$ and $R_C = (T-q)^{-1} \sum_{t=q+\ell+1}^T \Delta_q m_t \Delta_q m_{t-\ell}$ with $\Delta_q m_t = m(\frac{t}{T}) - m(\frac{t-q}{T})$. With the help of Lemma \ref{lemma-lrv-1}, it is straightforward to show that  
\[ \widehat{\gamma}_q^*(\ell) - \gamma_q(\ell) = O_p(T^{-1/2}). \]
Moreover, the Cauchy-Schwarz inequality yields that 
\[ \ex [R_A^2] \le \Big\{ \frac{1}{T-q} \sum\limits_{t=q+\ell+1}^T (\Delta_q m_t)^2 \Big\} \ex \Big[ \frac{1}{T-q} \sum\limits_{t=q+\ell+1}^T (\Delta_q \varepsilon_{t-\ell})^2 \Big]. \]
Since $m$ is Lipschitz by assumption, we get that $(T-q)^{-1} \sum_{t=q+\ell+1}^T (\Delta_q m_t)^2 \le C (q/T)^2$. In addition, it obviously holds that 
$\ex [ (T-q)^{-1} \sum_{t=q+\ell+1}^T (\Delta_q \varepsilon_{t-\ell})^2 ] = O(1)$.  
Hence, we can infer that 
\[ \ex [R_A^2] = O\Big( \Big\{ \frac{q}{T} \Big\}^2 \Big), \]
which implies that $R_A = o_p(T^{-1/2})$. Similar arguments yield that $R_j = o_p(T^{-1/2})$ for $j = B,C$ as well. Putting everything together, we arrive at the statement of Lemma \ref{lemma-lrv-2}. 
\end{proof}

\subsection*{Proof of Proposition \ref{prop-lrv}}

We first show that the pilot estimator $\widetilde{\boldsymbol{a}}_q$ converges to $\boldsymbol{a}$. In particular, we verify that $\widetilde{\boldsymbol{a}}_q - \boldsymbol{a} = O_p(T^{-1/2})$. By Lemma \ref{lemma-lrv-2}, it holds that $\widehat{\boldsymbol{\Gamma}}_q = \boldsymbol{\Gamma}_q + O_p(T^{-1/2})$ and $\widehat{\boldsymbol{\gamma}}_q = \boldsymbol{\gamma}_q + O_p(T^{-1/2})$. Since $\boldsymbol{\Gamma}_q$ is invertible, this implies that 
\[ \widetilde{\boldsymbol{a}}_q = \boldsymbol{\Gamma}_q^{-1} \boldsymbol{\gamma}_q + O_p(T^{-1/2}). \]
With the help of equation \eqref{YW-eq}, we can further infer that 
\[ \widetilde{\boldsymbol{a}}_q - \boldsymbol{a} = -\nu^2 \boldsymbol{\Gamma}_q^{-1} \boldsymbol{c}_q + O_p(T^{-1/2}). \]
As already noted in Section \ref{subsec-error-var-AR}, the entries of the vector $\boldsymbol{c}_q = (c_{q-1},\ldots,c_{c-p})^\top$ decay exponentially fast to zero, that is, $|c_k| \le C \rho^k$ for some $0 < \rho < 1$. Moreover, it holds that $\gamma_q(\ell) \rightarrow 2 \gamma_\varepsilon(\ell)$ for any fixed $\ell$ as $q \rightarrow \infty$. Consequently, $\| \nu^2 \boldsymbol{\Gamma}_q^{-1} \boldsymbol{c}_q \|_\infty = o(T^{-1/2})$, where $\| \cdot \|_\infty$ denotes the usual supremum norm for vectors. As a result, we obtain that $\widetilde{\boldsymbol{a}}_q - \boldsymbol{a} = O_p(T^{-1/2})$.

We next show that $\widehat{\boldsymbol{a}}_r - \boldsymbol{a} = O_p(T^{-1/2})$, where $r \ge 1$ is any fixed integer that does not grow with the sample size $T$. By definition, it holds that $\widehat{\boldsymbol{a}}_r = \widehat{\boldsymbol{\Gamma}}_r^{-1} (\widehat{\boldsymbol{\gamma}}_r + \widetilde{\nu}^2 \widetilde{\boldsymbol{c}}_r)$.  From Lemma \ref{lemma-lrv-2}, it follows that $\widehat{\boldsymbol{\Gamma}}_r^{-1} = \boldsymbol{\Gamma}_r^{-1} + O_p(T^{-1/2})$ and $\widehat{\boldsymbol{\gamma}}_r = \boldsymbol{\gamma}_r + O_p(T^{-1/2})$. Moreover, with the help of the fact that $\widetilde{\boldsymbol{a}}_q - \boldsymbol{a} = O_p(T^{-1/2})$, it is straightforward to verify that $\widetilde{\nu}^2 - \nu^2 = O_p(T^{-1/2})$ and $\widetilde{\boldsymbol{c}}_r - \boldsymbol{c}_r =  O_p(T^{-1/2})$. Hence, we arrive at  
\begin{equation}\label{eq-conv-est-AR-SS}
\widehat{\boldsymbol{a}}_r = \boldsymbol{\Gamma}_r^{-1} (\boldsymbol{\gamma}_r + \nu^2 \boldsymbol{c}_r) + O_p(T^{-1/2}) = \boldsymbol{a} + O_p(T^{-1/2}), 
\end{equation}
where the last equality is due to equation \eqref{YW-eq}.

From \eqref{eq-conv-est-AR-SS}, it immediately follows that $\widehat{\boldsymbol{a}} - \boldsymbol{a} = O_p(T^{-1/2})$, which in turn allows us to infer that $\widehat{\nu}^2 - \nu^2 = O_p(T^{-1/2})$ and $\widehat{\sigma}^2 = \sigma^2 + O_p(T^{-1/2})$ by straightforward arguments.

\section{Robustness checks and implementation details for the simulations in Section \ref{sec-sim}}\label{sec-supp-sim}

\subsection*{Robustness checks for Section \ref{subsec-sim-3}}

In what follows, we carry out some robustness checks to assess how sensitive the estimators $\widehat{a}$ and $\widehat{\sigma}^2$ are to the choice of the tuning parameters $q$ and $\overline{r}$. To do so, we repeat the simulation exercises of Section \ref{subsec-sim-3} for different values of $q$ and $\overline{r}$. In addition, we consider different choices of the tuning parameters $(m_1,m_2)$ on which the estimators of \cite{Hall2003} depend. As in Section \ref{subsec-sim-3}, we choose $m_1$ and $m_2$ such that $q$ lies between these values. We thus keep the parameters $q$ and $(m_1,m_2)$ roughly comparable.

To start with, we consider the simulation scenarios with a moderate trend ($s_\beta = 1$). The MSE values of the estimators $\widehat{a}$, $\widehat{a}_{\text{HvK}}$, $\widehat{a}_{\text{oracle}}$ and $\widehat{\sigma}^2$, $\widehat{\sigma}^2_{\text{HvK}}$, $\widehat{\sigma}^2_{\text{oracle}}$ for these scenarios are presented in Figure \ref{fig:MSE_slope1} of Section \ref{subsec-sim-3}. These MSEs are re-calculated in Figures \ref{fig:MSE_slope1_AR_robust} and \ref{fig:MSE_slope1_lrv_robust} for a range of different choices of $q$, $\overline{r}$ and $(m_1,m_2)$. As one can see, the MSEs in the different plots of Figures \ref{fig:MSE_slope1_AR_robust} and \ref{fig:MSE_slope1_lrv_robust} are very similar. Hence, the MSE results reported in Section \ref{subsec-sim-3} for the scenarios with a moderate trend appear to be fairly robust to different choices of the tuning parameters. In particular, our estimators $\widehat{a}$ and $\widehat{\sigma}^2$ seem to be quite insensitive to the choice of tuning parameters, at least as far as their MSEs are concerned.

We next turn to the simulation designs with a pronounced trend ($s_\beta = 10$). The MSE values of the estimators in these scenarios are reported in Figure \ref{fig:MSE_slope10} of Section \ref{subsec-sim-3}. Analogously as before, we re-calculate these MSEs for different tuning parameters in Figures \ref{fig:MSE_slope10_AR_robust}--\ref{fig:MSE_slope10_lrv_robust}. Figure \ref{fig:MSE_slope10_AR_zoom_robust} is a zoomed-in version of Figure \ref{fig:MSE_slope10_AR_robust} which is added for better visibility. As can be seen, our estimators appear to be barely influenced by the choice of $q$. However, the MSE values become somewhat larger when $\overline{r}$ is chosen bigger. This is of course not very surprising: The main reason why the estimator $\widehat{a}$ works well in the presence of a strong trend is that it is only based on differences of small orders. If we increase $\overline{r}$, we use larger differences to compute $\widehat{a}$, which results in not eliminating the trend $m$ appropriately any more. This becomes visible in somewhat larger MSE values. Nevertheless, overall, our estimators appear not to be strongly influenced by the choice of tuning parameters (in terms of MSE) as long as these are chosen within reasonable bounds.

\begin{figure}[h!]
\begin{subfigure}[b]{0.45\textwidth}
\includegraphics[width=\textwidth]{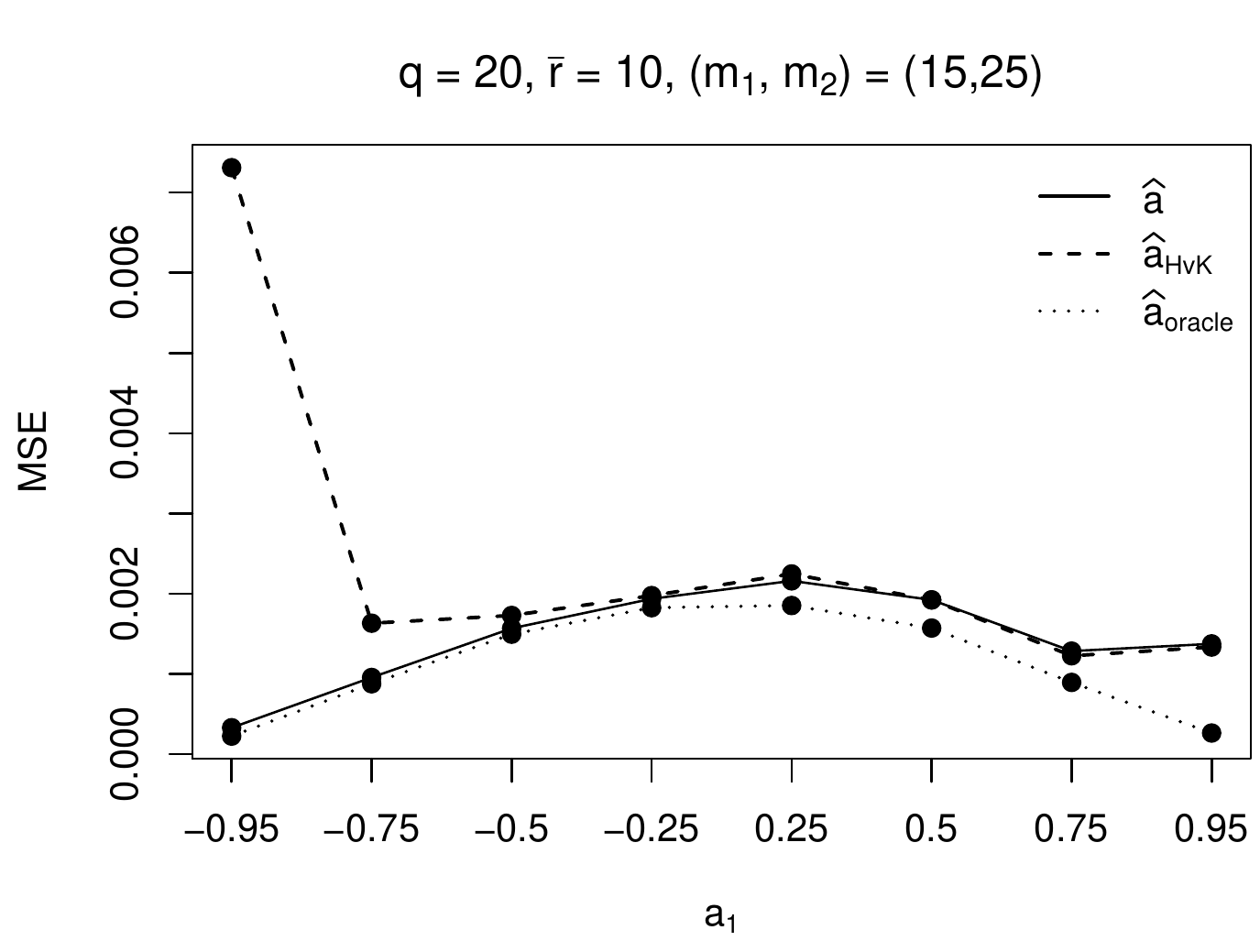}
\end{subfigure}
\hspace{0.25cm}
\begin{subfigure}[b]{0.45\textwidth}
\includegraphics[width=\textwidth]{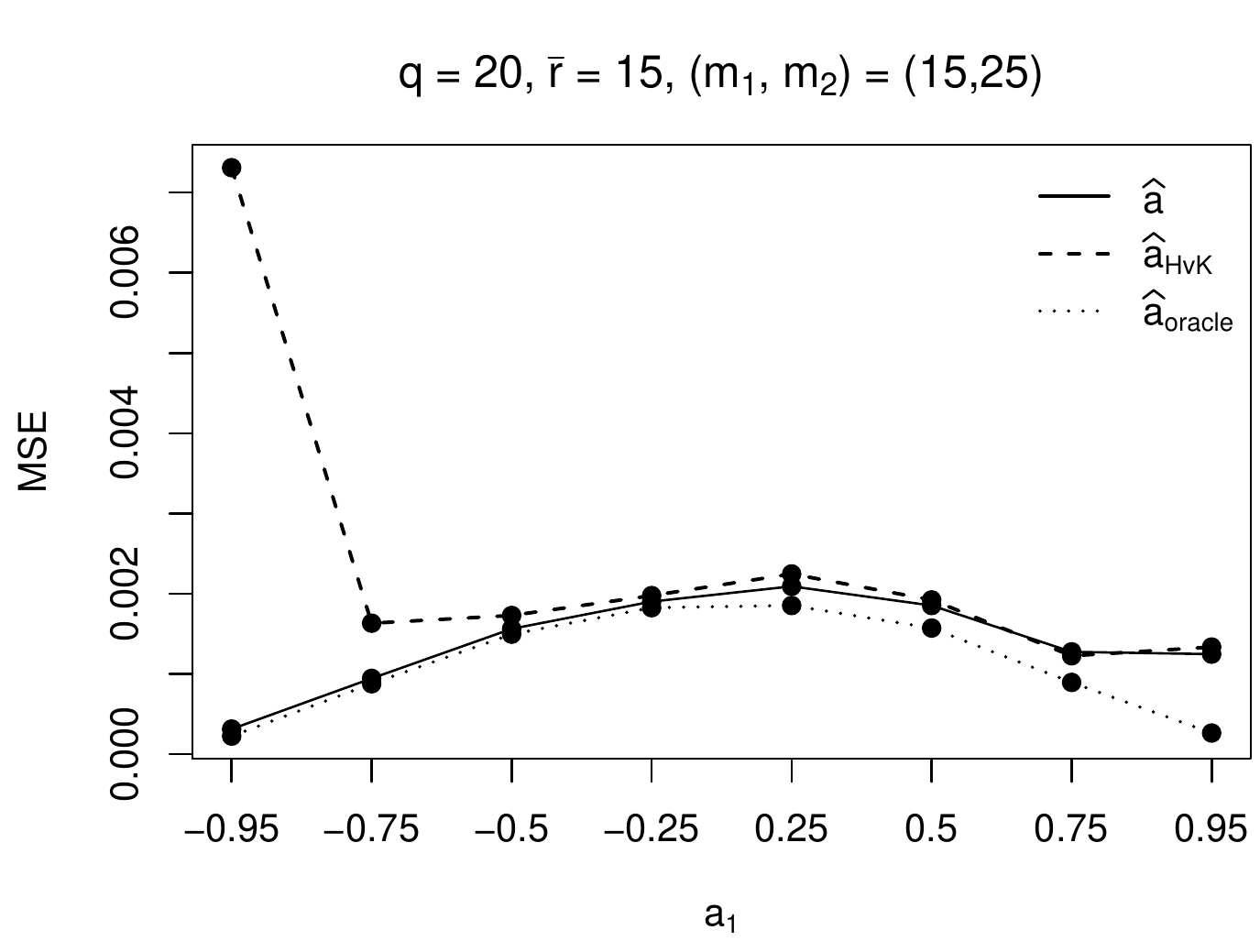}
\end{subfigure}

\begin{subfigure}[b]{0.45\textwidth}
\includegraphics[width=\textwidth]{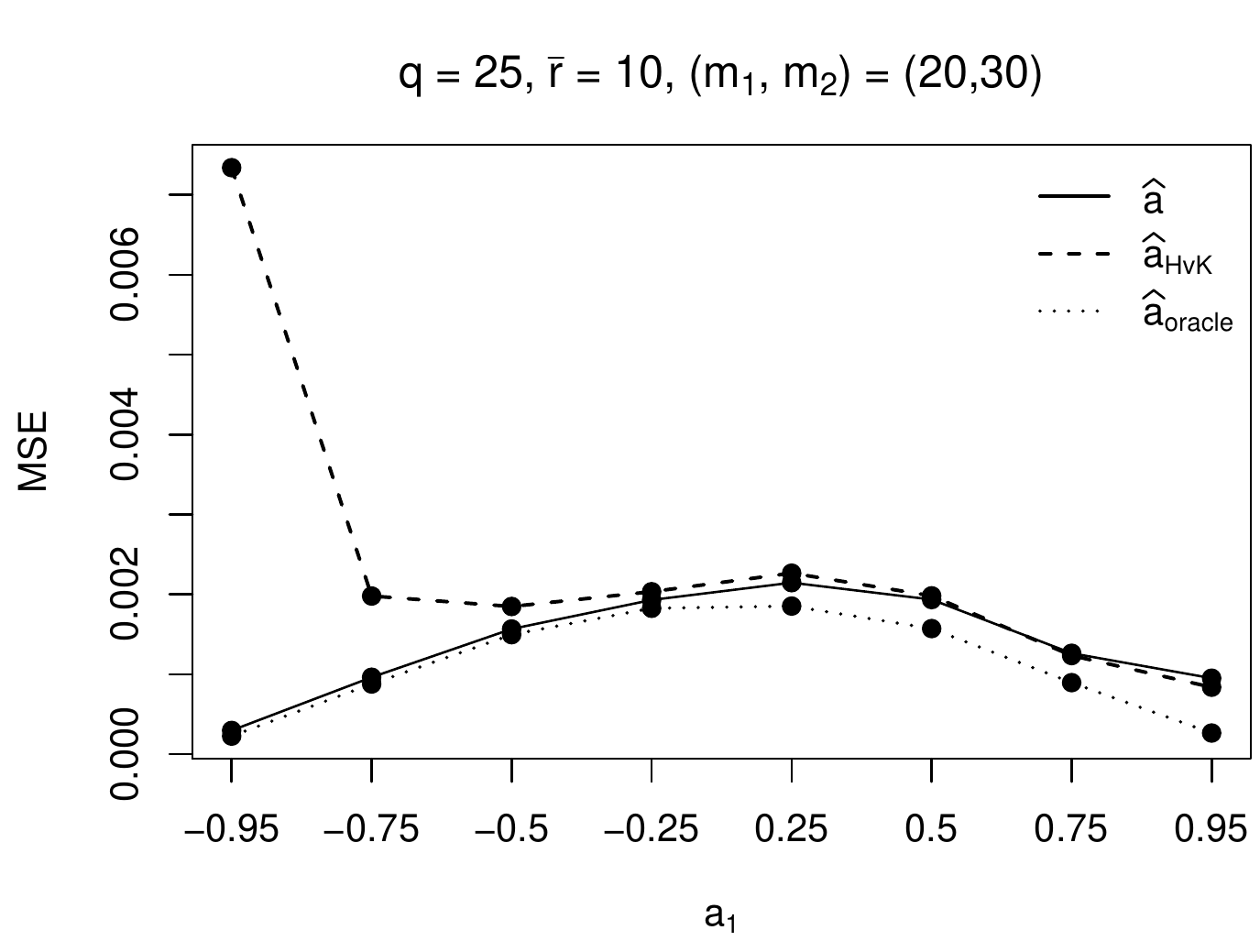}
\end{subfigure}
\hspace{0.25cm}
\begin{subfigure}[b]{0.45\textwidth}
\includegraphics[width=\textwidth]{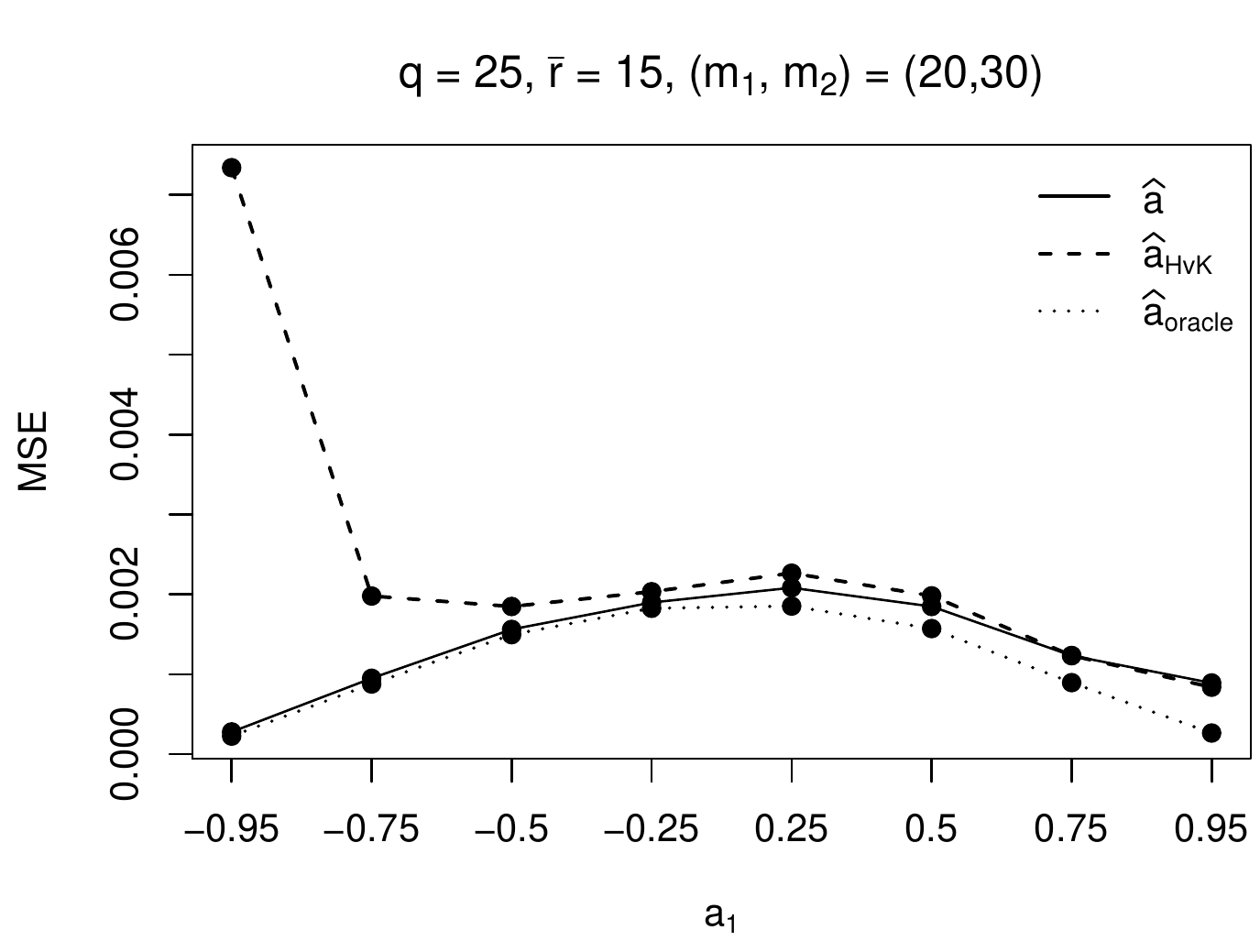}
\end{subfigure}

\begin{subfigure}[b]{0.45\textwidth}
\includegraphics[width=\textwidth]{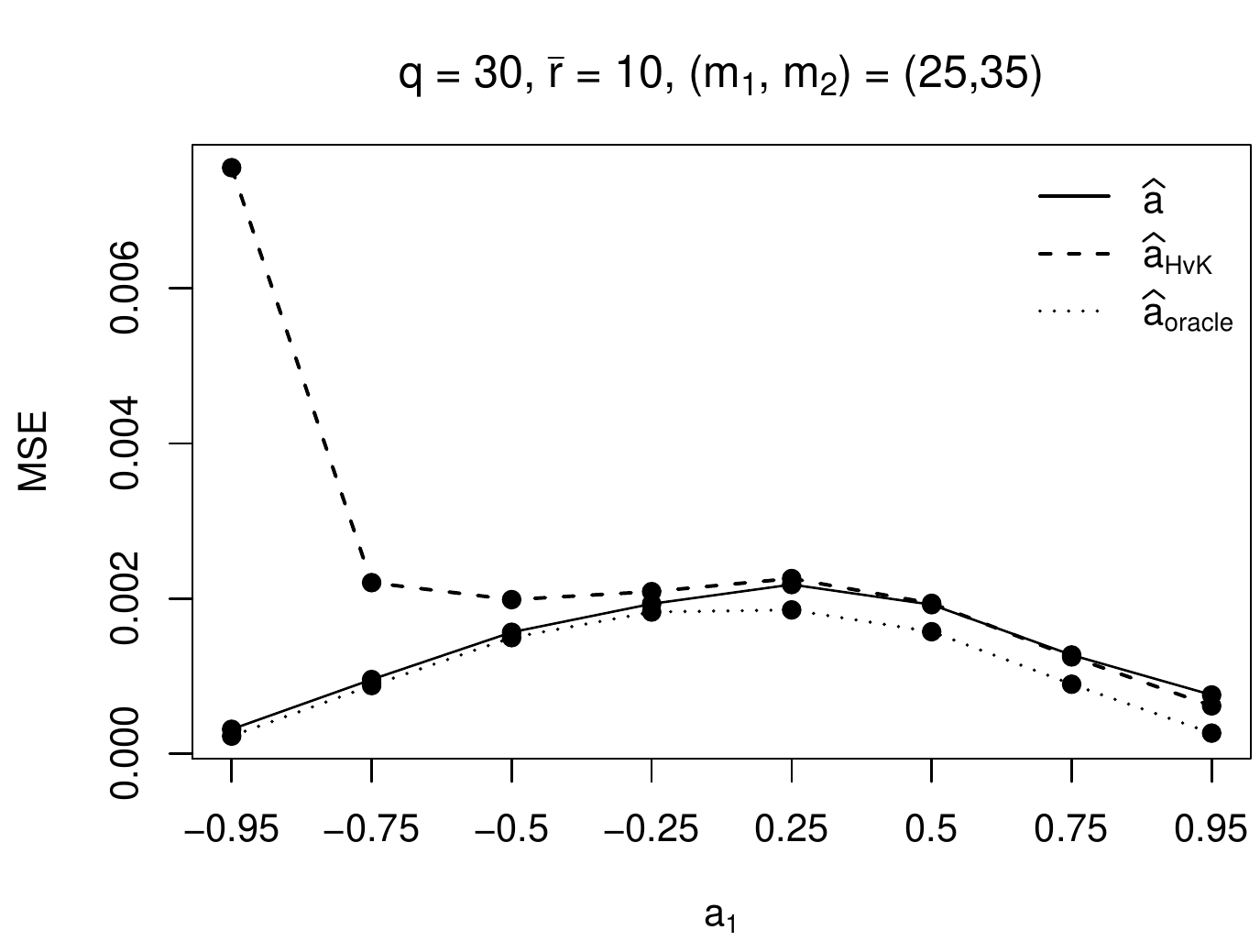}
\end{subfigure}
\hspace{0.25cm}
\begin{subfigure}[b]{0.45\textwidth}
\includegraphics[width=\textwidth]{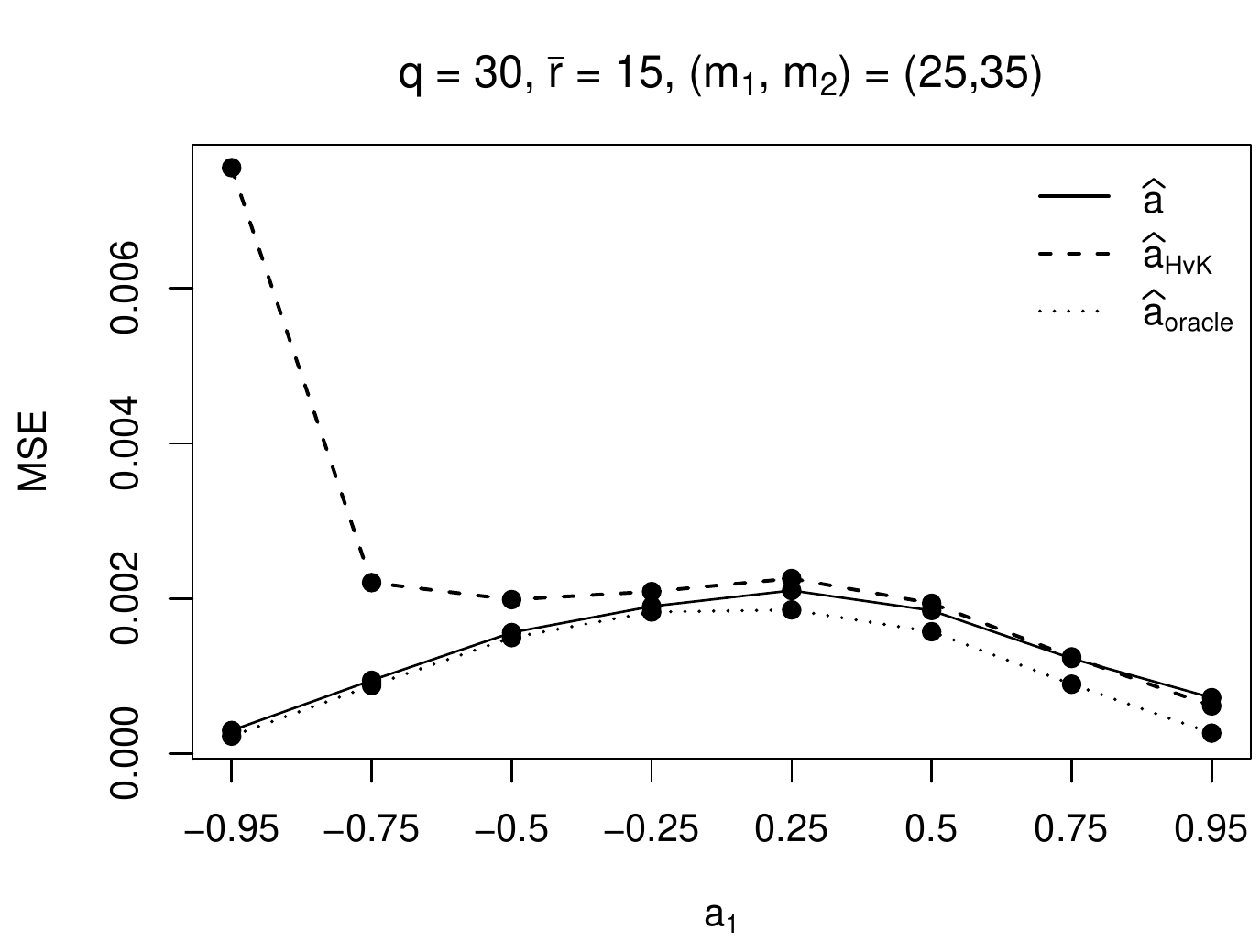}
\end{subfigure}

\begin{subfigure}[b]{0.45\textwidth}
\includegraphics[width=\textwidth]{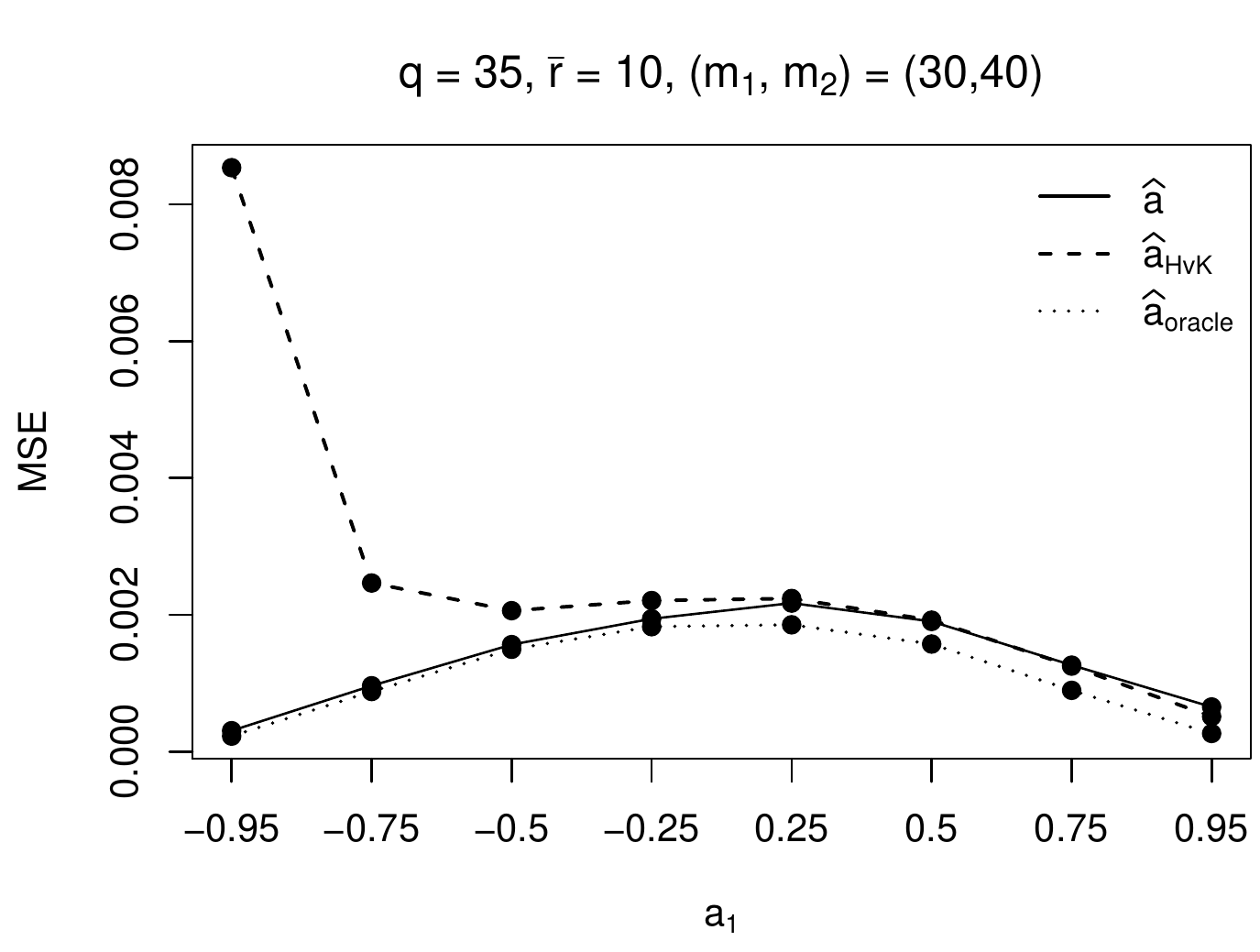}
\end{subfigure}
\hspace{0.25cm}
\begin{subfigure}[b]{0.45\textwidth}
\includegraphics[width=\textwidth]{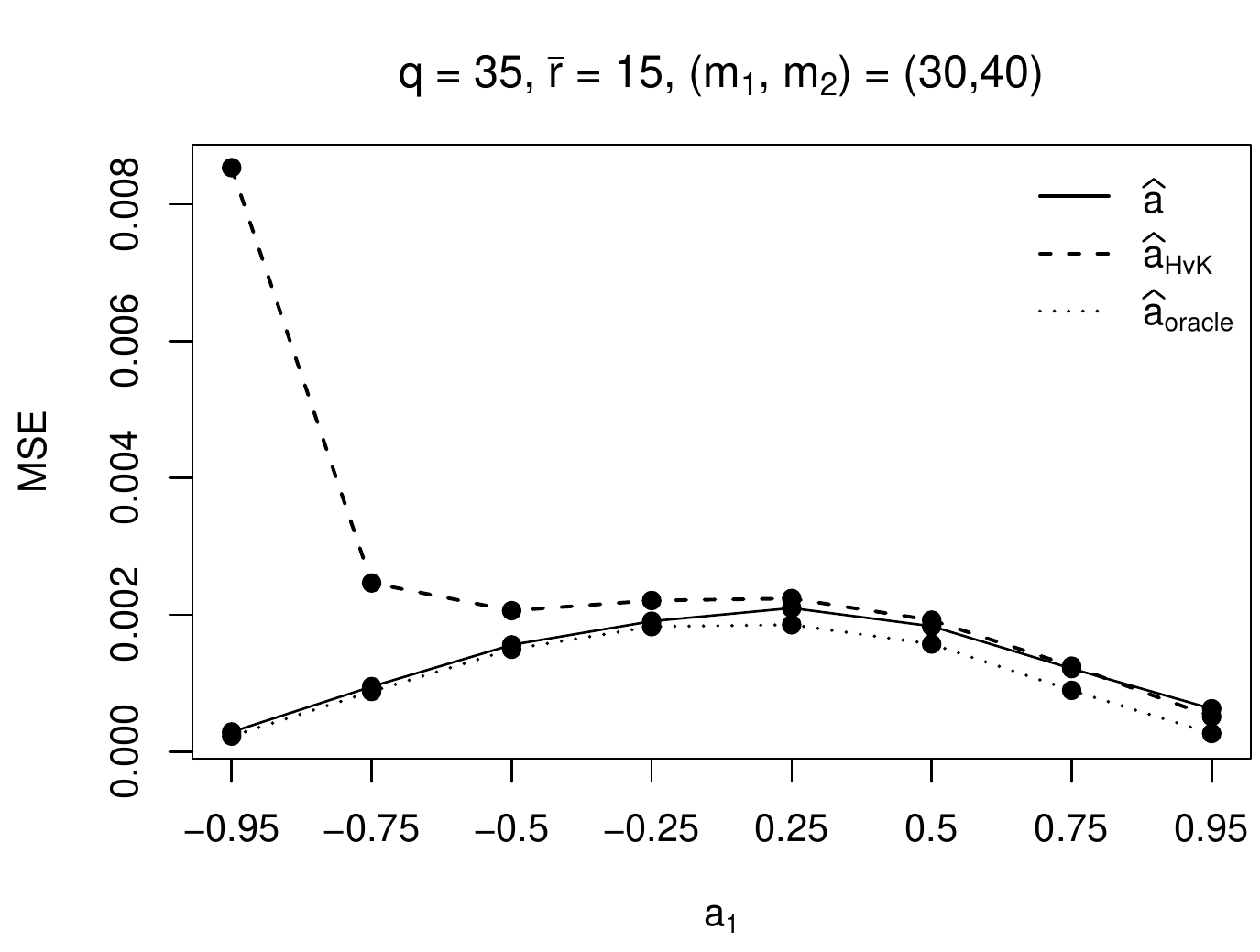}
\end{subfigure}
\caption{MSE values for the estimators $\widehat{a}$, $\widehat{a}_{\text{HvK}}$ and $\widehat{a}_{\text{oracle}}$ in the scenario with a moderate trend ($s_\beta=1$).}\label{fig:MSE_slope1_AR_robust} 
\end{figure}

\begin{figure}[p]
\begin{subfigure}[b]{0.45\textwidth}
\includegraphics[width=\textwidth]{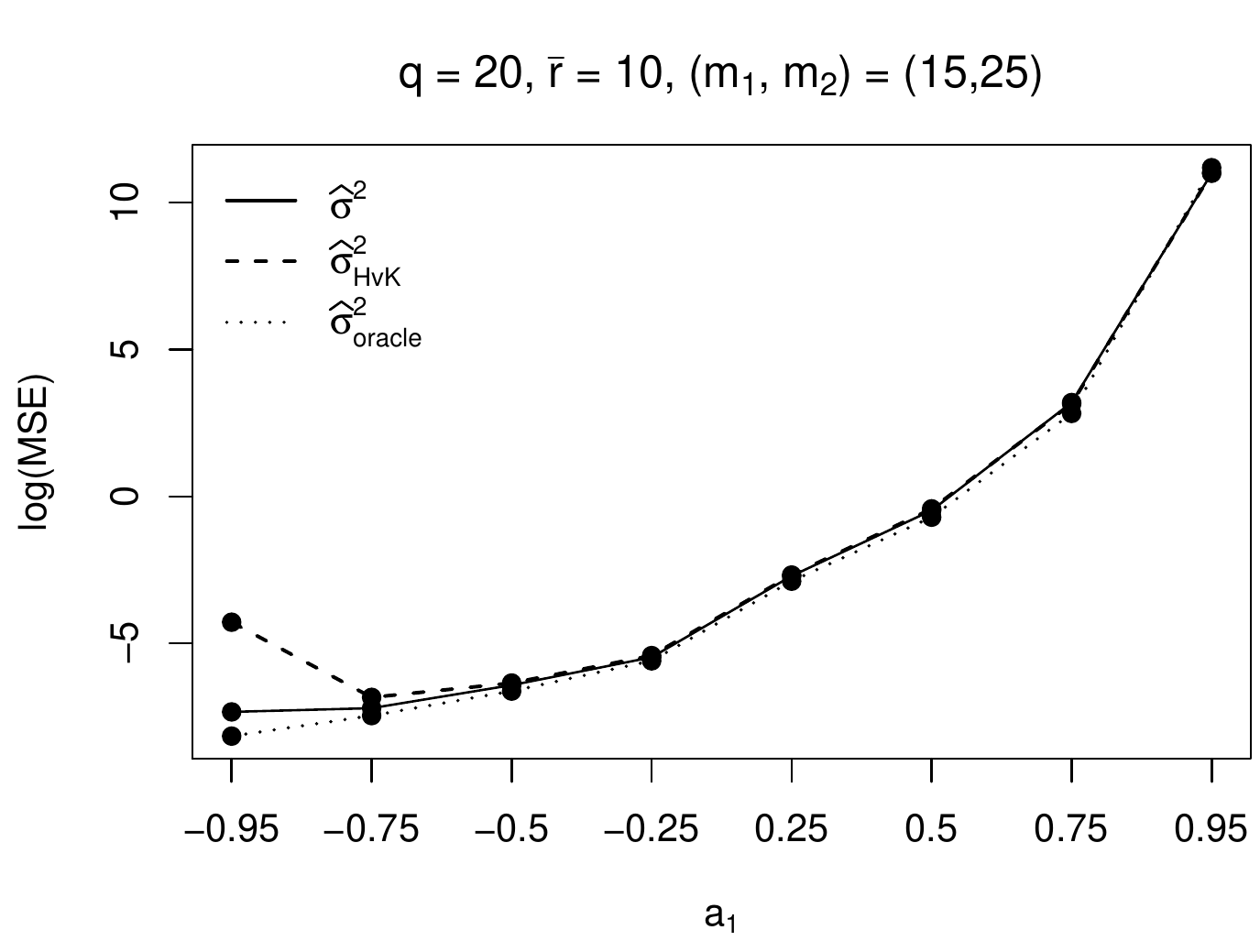}
\end{subfigure}
\hspace{0.25cm}
\begin{subfigure}[b]{0.45\textwidth}
\includegraphics[width=\textwidth]{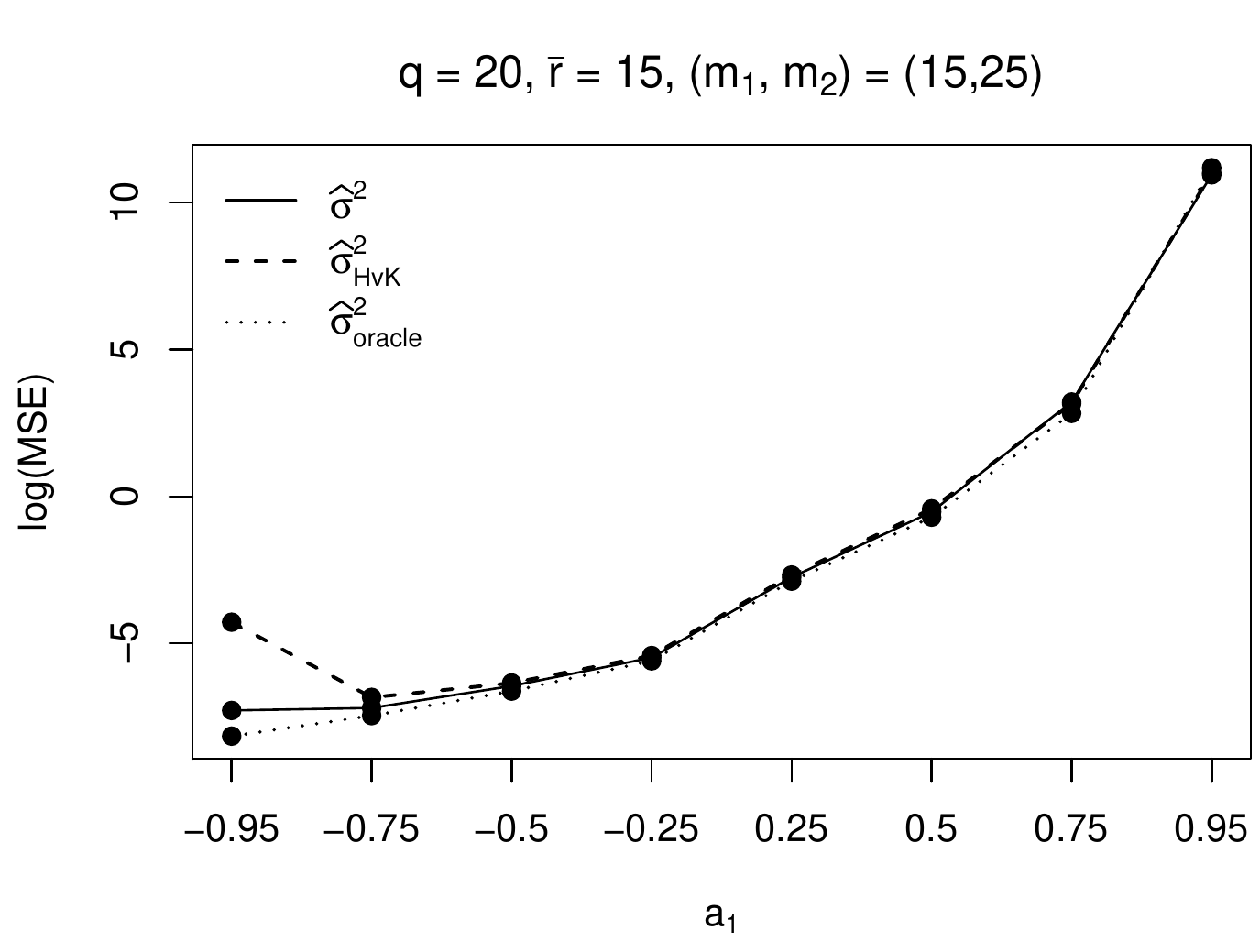}
\end{subfigure}

\begin{subfigure}[b]{0.45\textwidth}
\includegraphics[width=\textwidth]{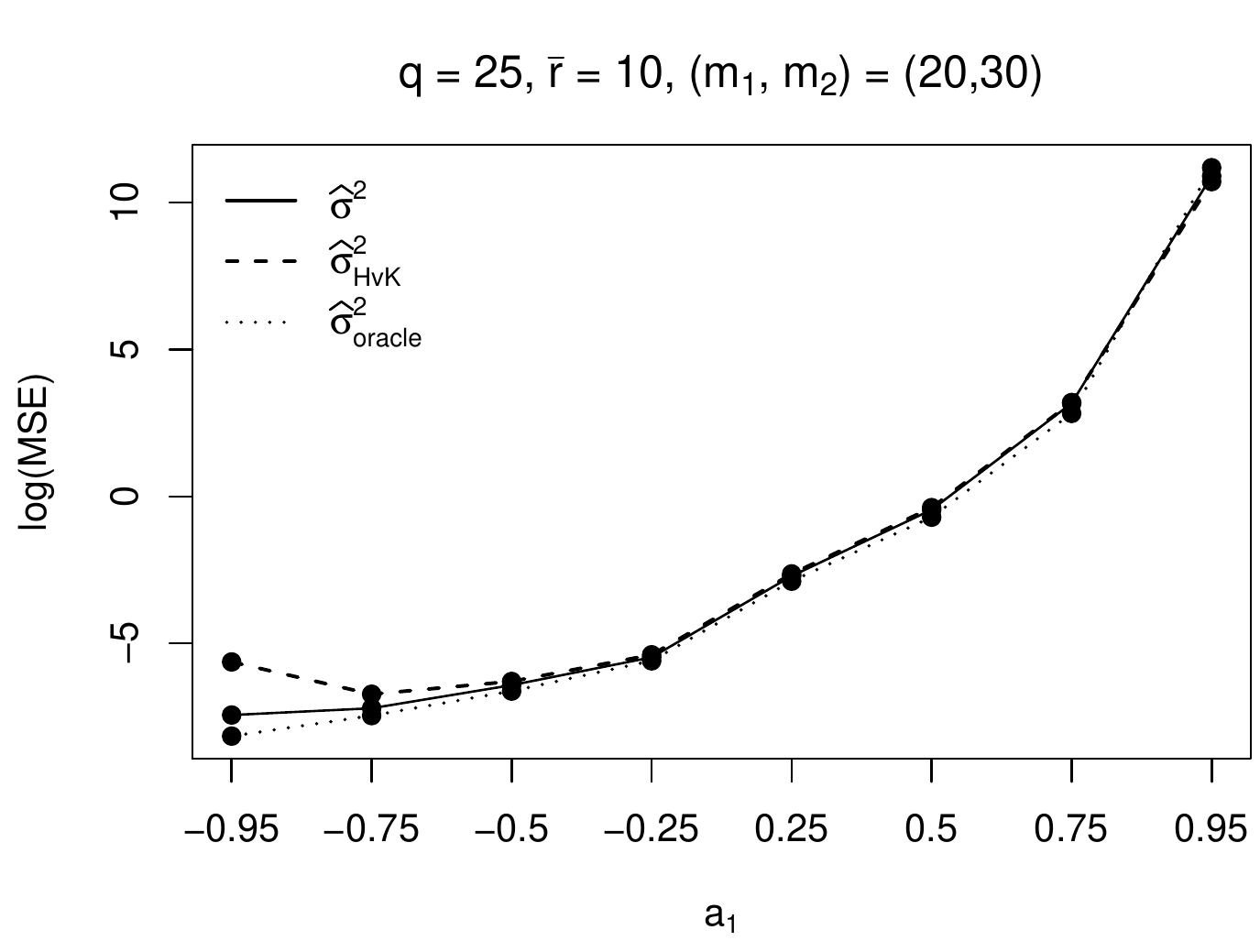}
\end{subfigure}
\hspace{0.25cm}
\begin{subfigure}[b]{0.45\textwidth}
\includegraphics[width=\textwidth]{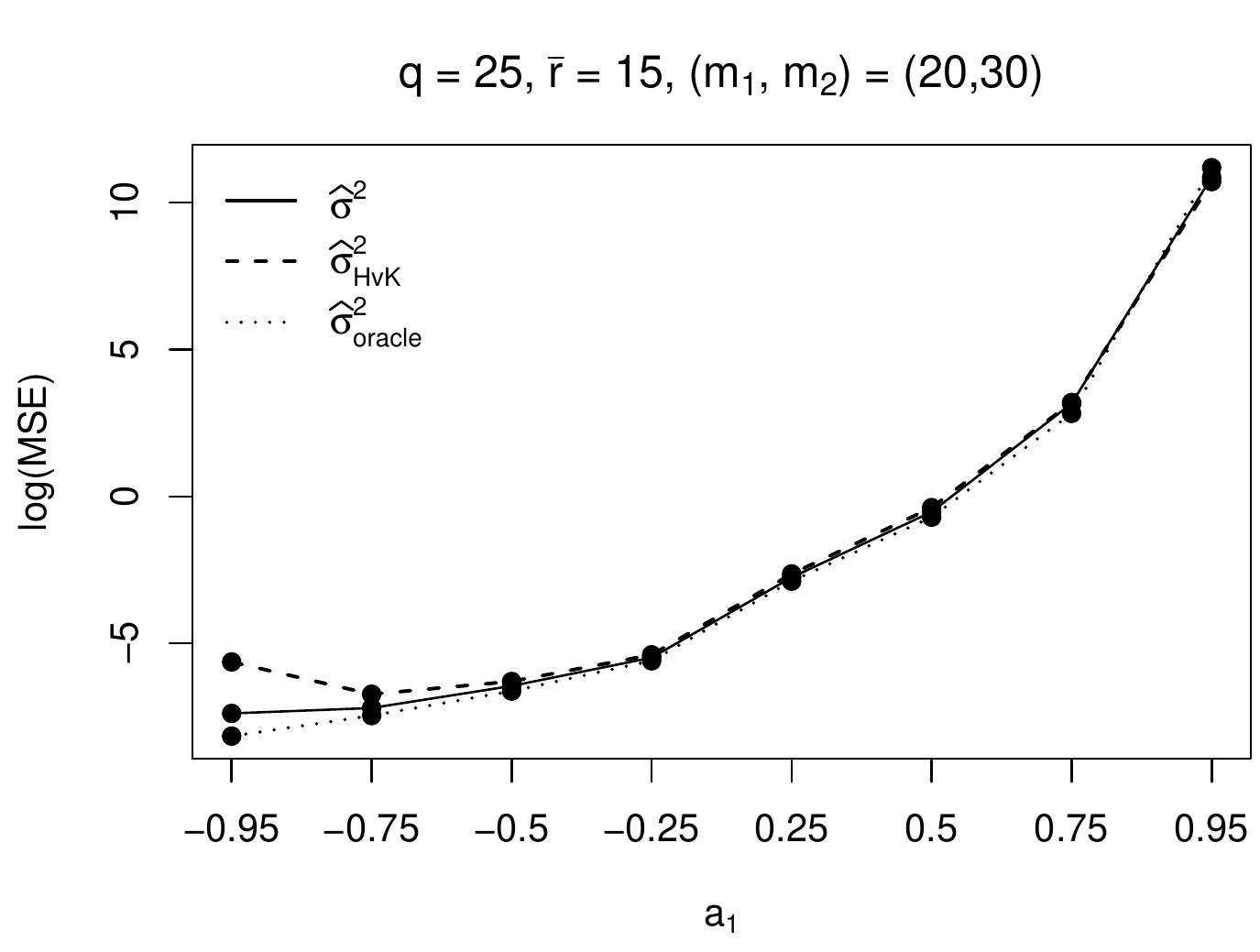}
\end{subfigure}

\begin{subfigure}[b]{0.45\textwidth}
\includegraphics[width=\textwidth]{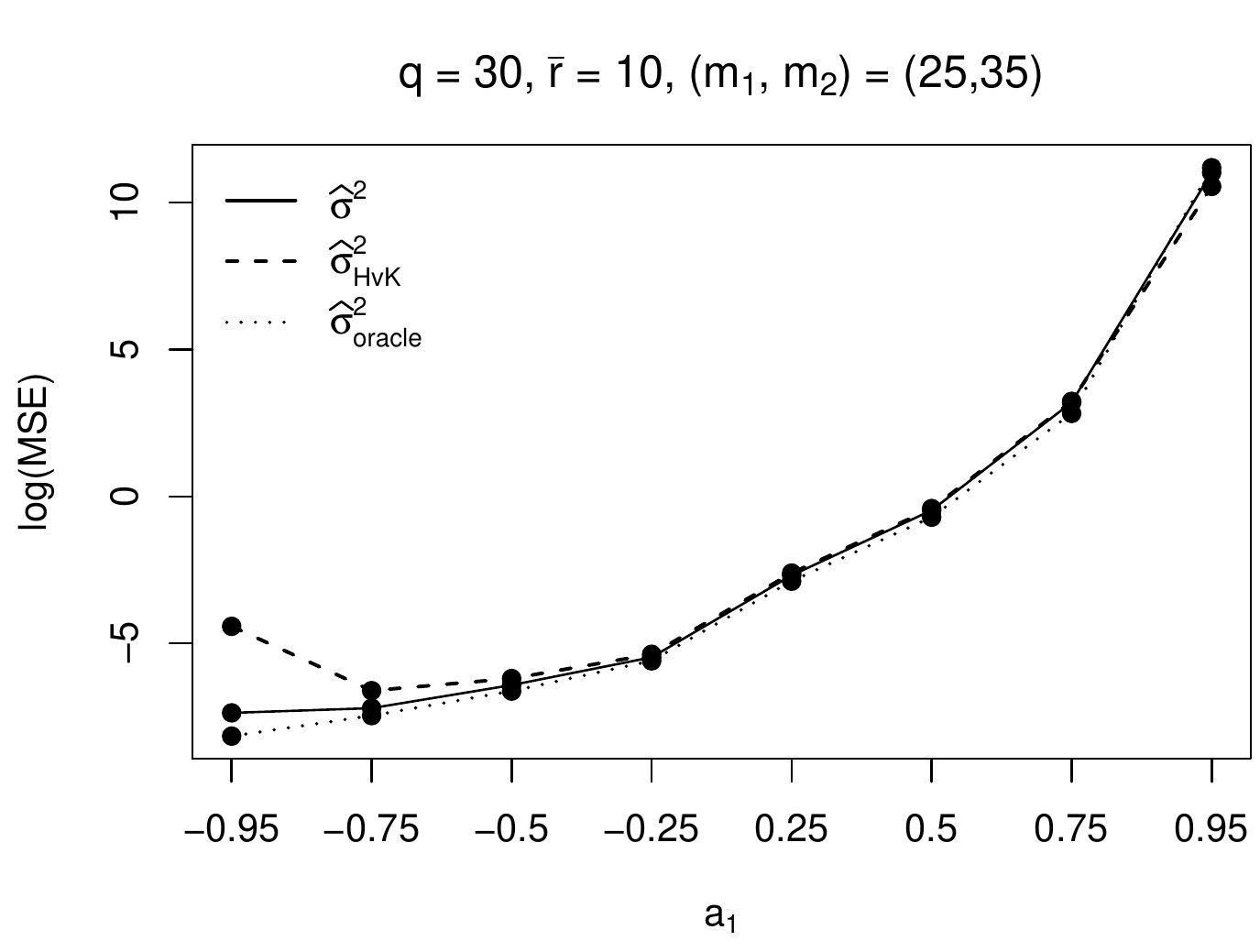}
\end{subfigure}
\hspace{0.25cm}
\begin{subfigure}[b]{0.45\textwidth}
\includegraphics[width=\textwidth]{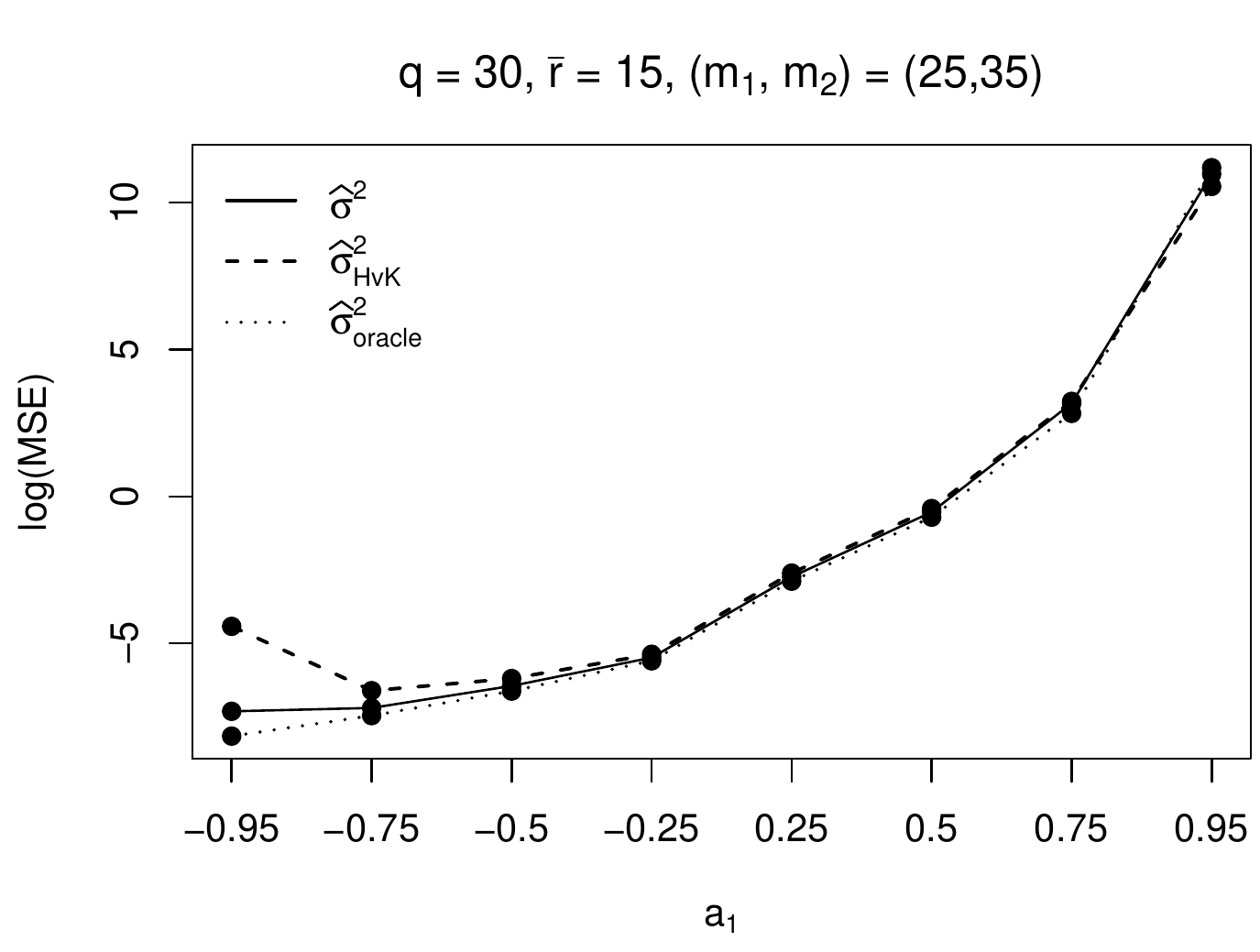}
\end{subfigure}

\begin{subfigure}[b]{0.45\textwidth}
\includegraphics[width=\textwidth]{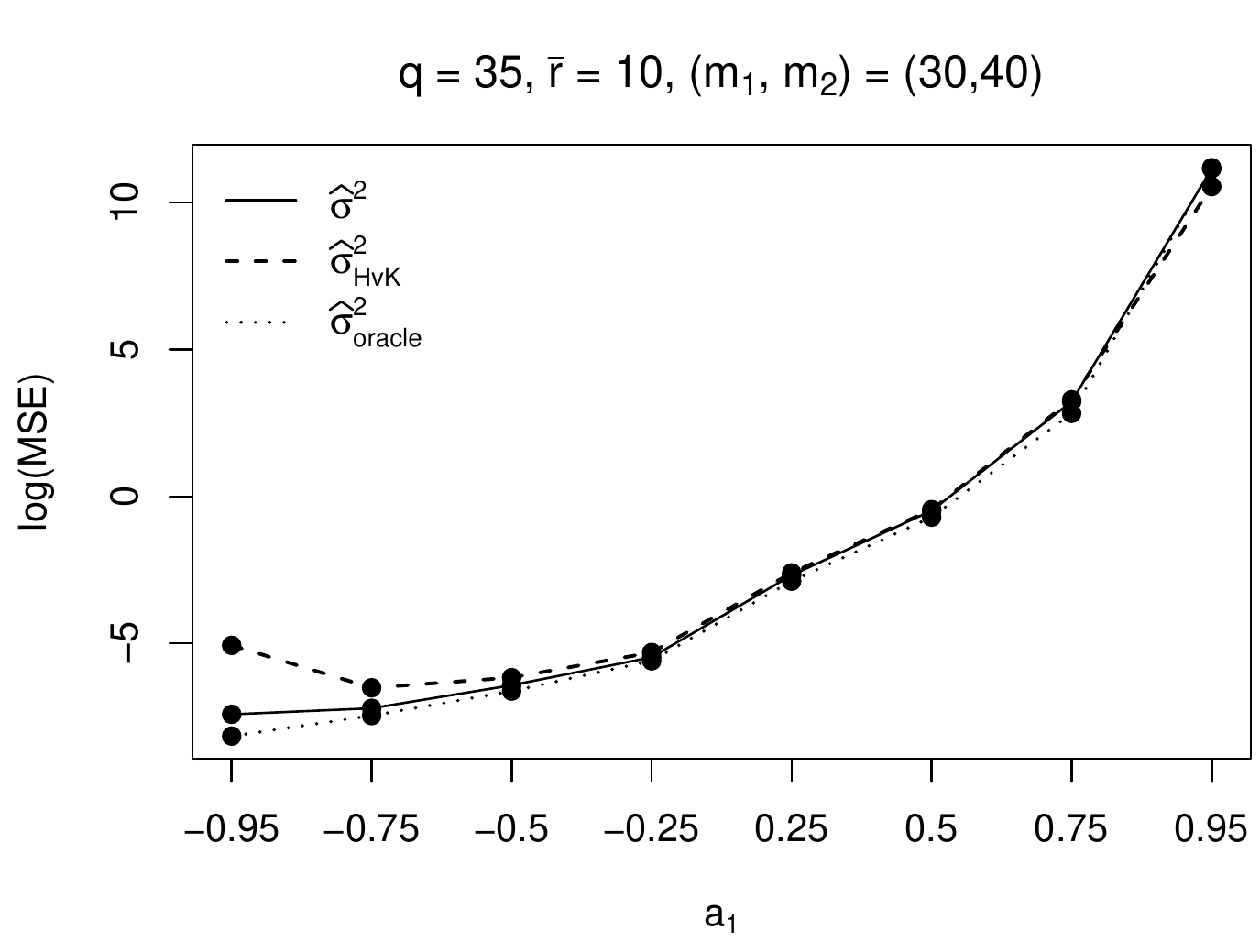}
\end{subfigure}
\hspace{0.25cm}
\begin{subfigure}[b]{0.45\textwidth}
\includegraphics[width=\textwidth]{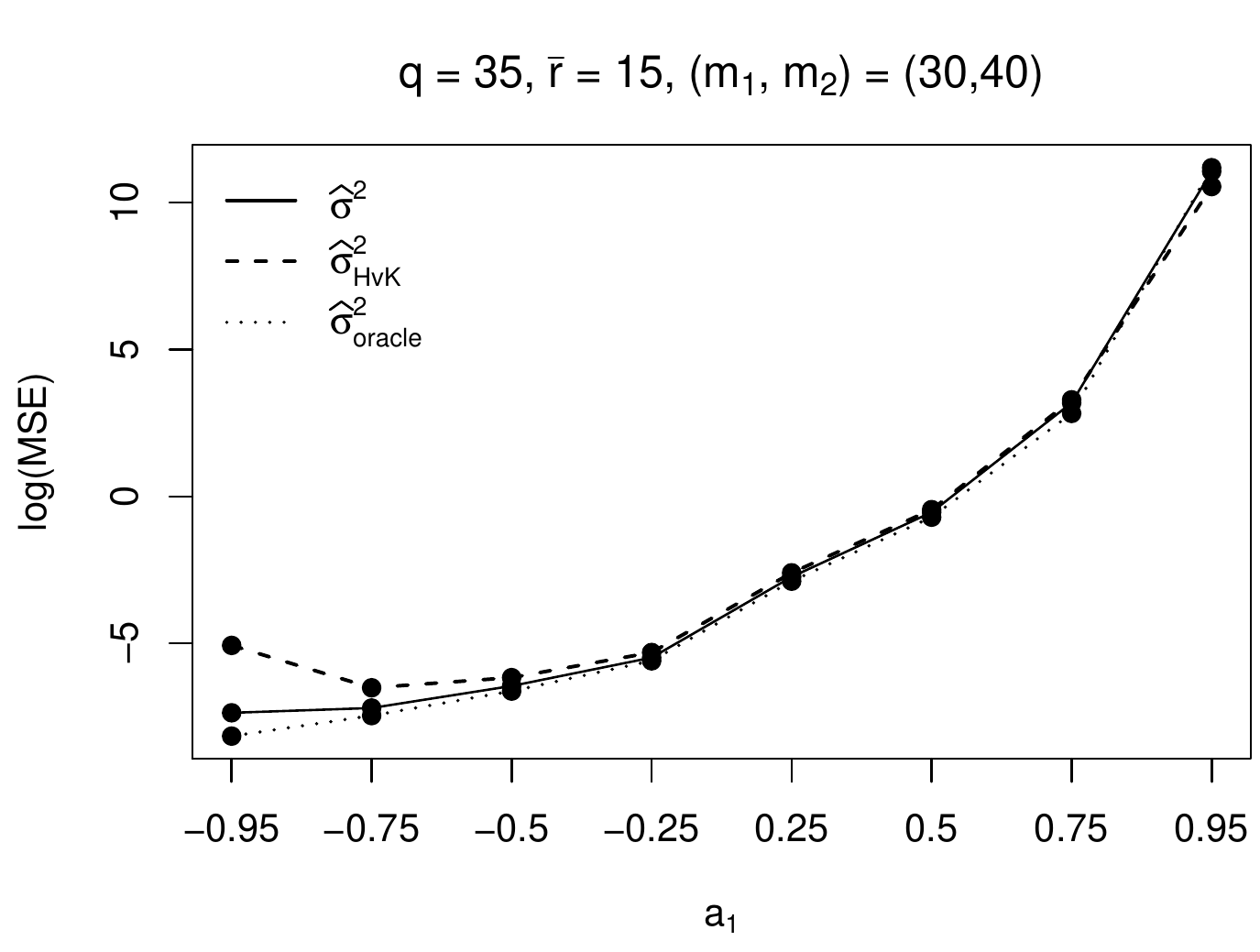}
\end{subfigure}
\caption{Logarithmic MSE values for the estimators $\widehat{\sigma}^2$, $\widehat{\sigma}^2_{\text{HvK}}$ and $\widehat{\sigma}^2_{\text{oracle}}$ in the scenario with a moderate trend ($s_\beta=1$).}\label{fig:MSE_slope1_lrv_robust}
\end{figure}

\begin{figure}[p]
\begin{subfigure}[b]{0.45\textwidth}
\includegraphics[width=\textwidth]{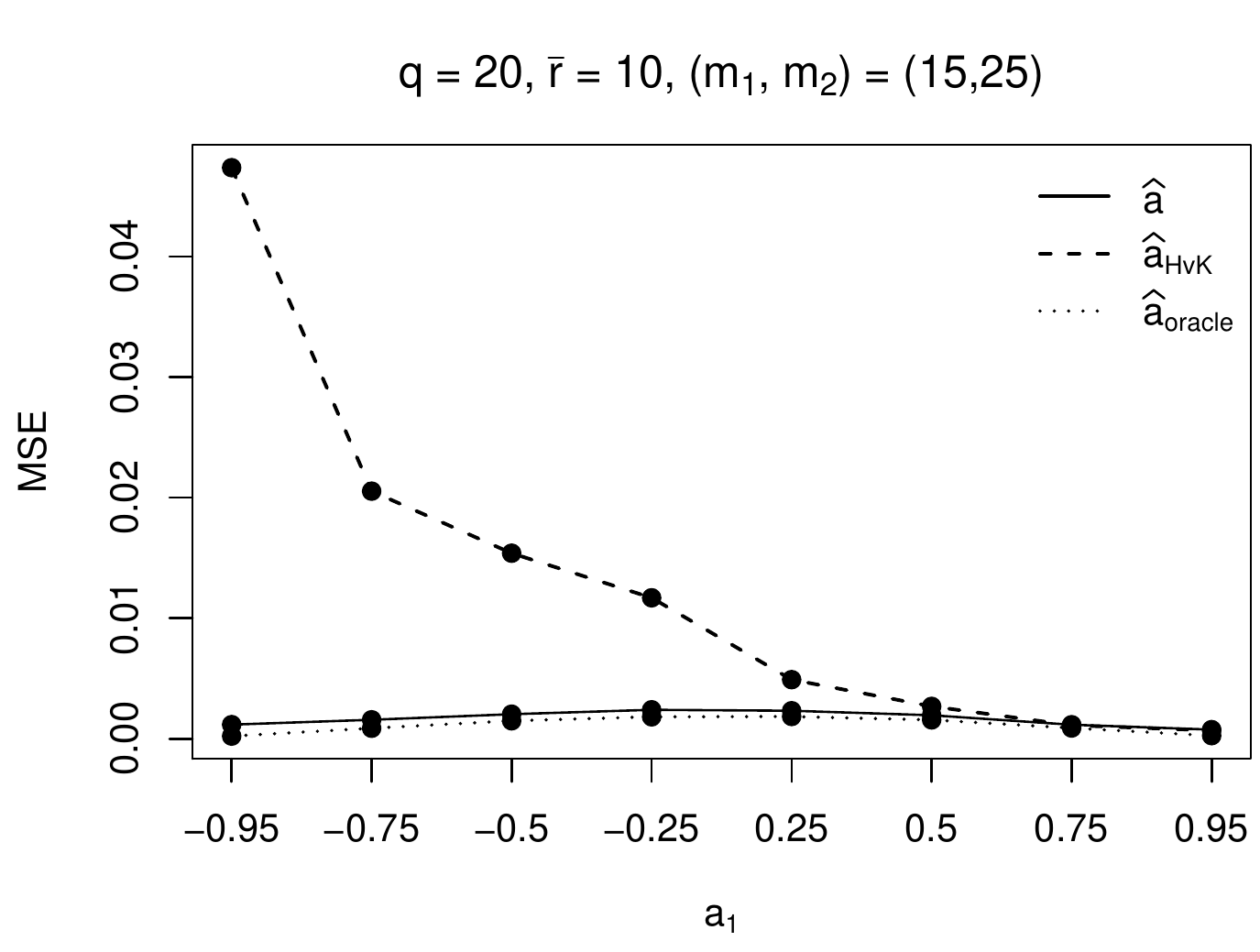}
\end{subfigure}
\hspace{0.25cm}
\begin{subfigure}[b]{0.45\textwidth}
\includegraphics[width=\textwidth]{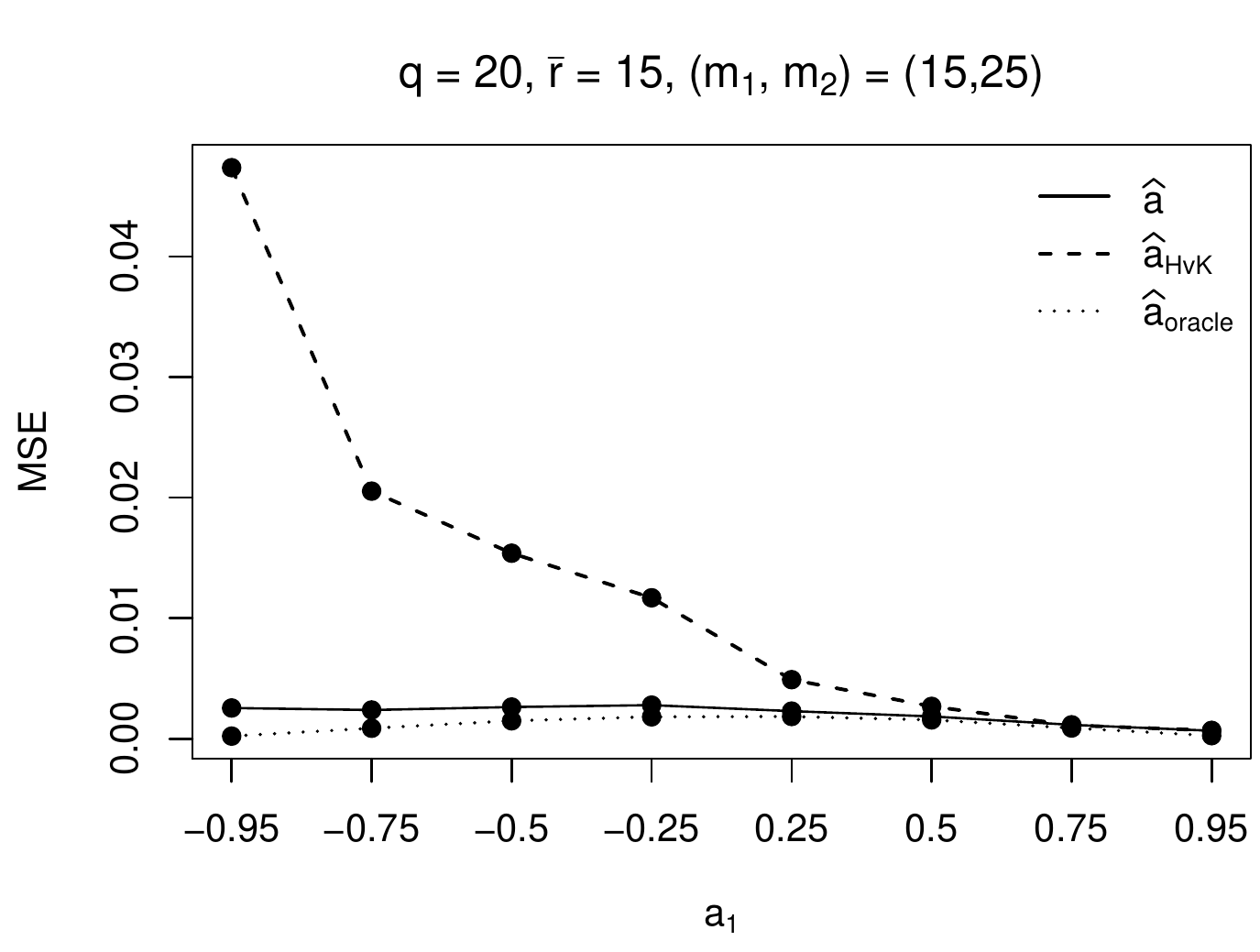}
\end{subfigure}

\begin{subfigure}[b]{0.45\textwidth}
\includegraphics[width=\textwidth]{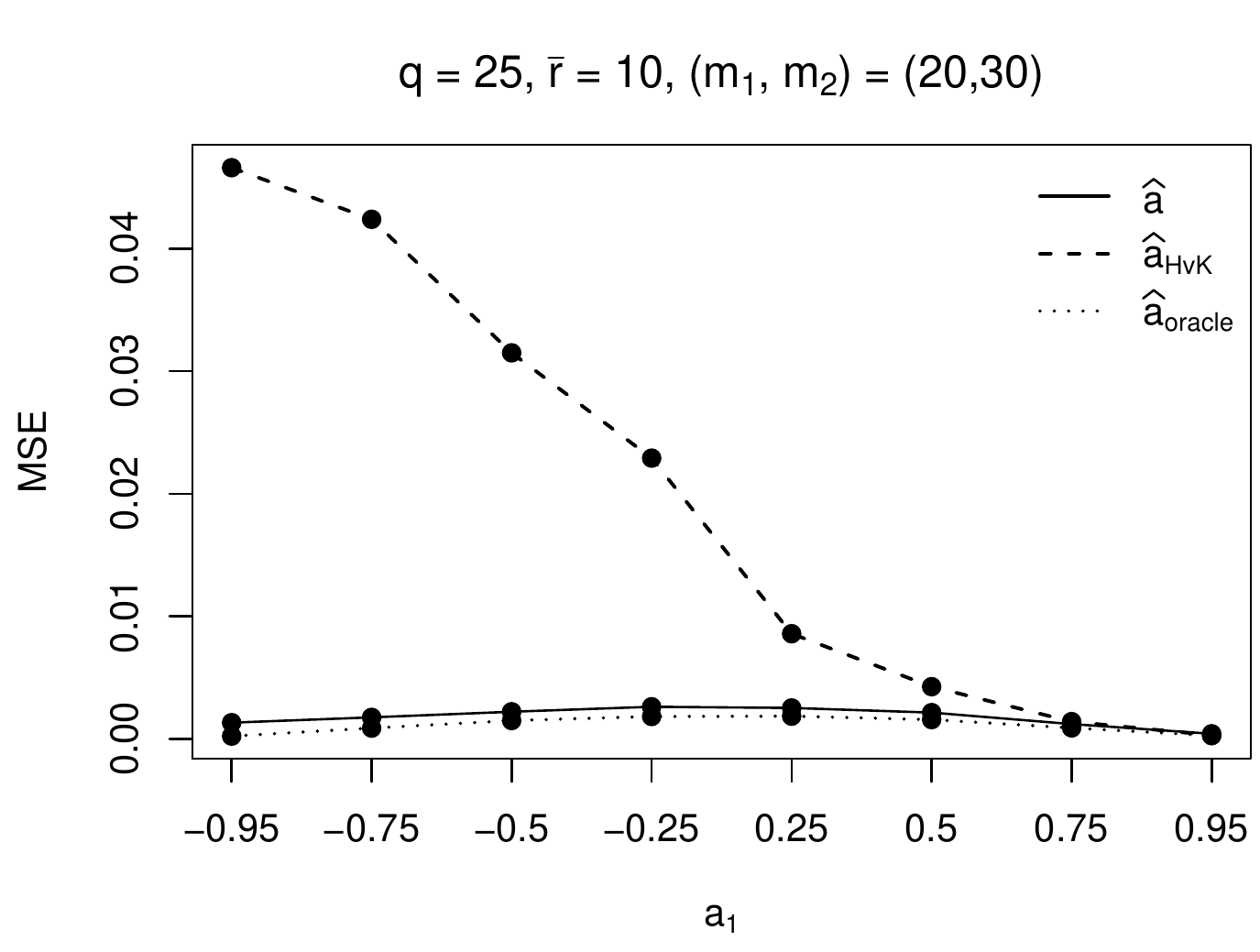}
\end{subfigure}
\hspace{0.25cm}
\begin{subfigure}[b]{0.45\textwidth}
\includegraphics[width=\textwidth]{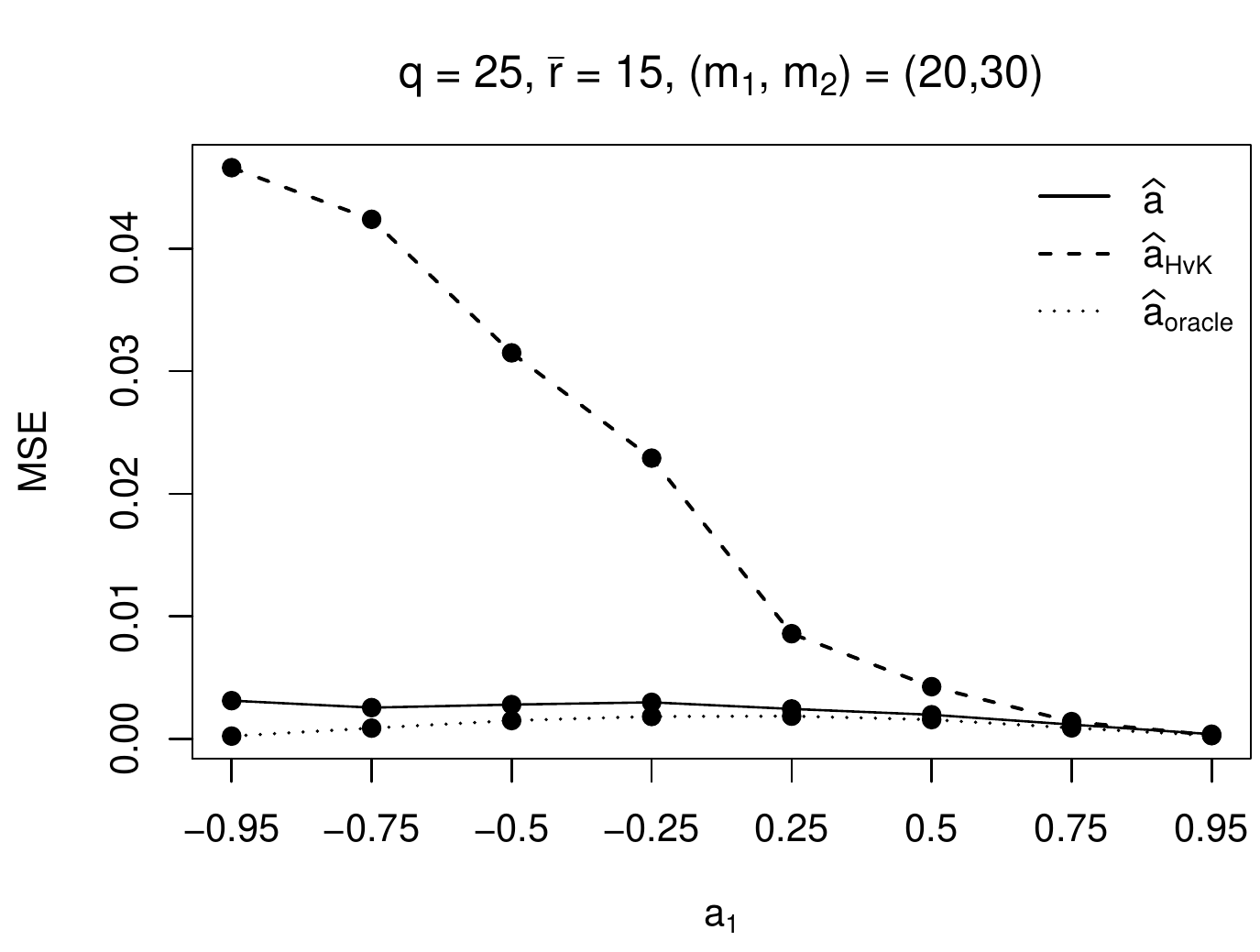}
\end{subfigure}

\begin{subfigure}[b]{0.45\textwidth}
\includegraphics[width=\textwidth]{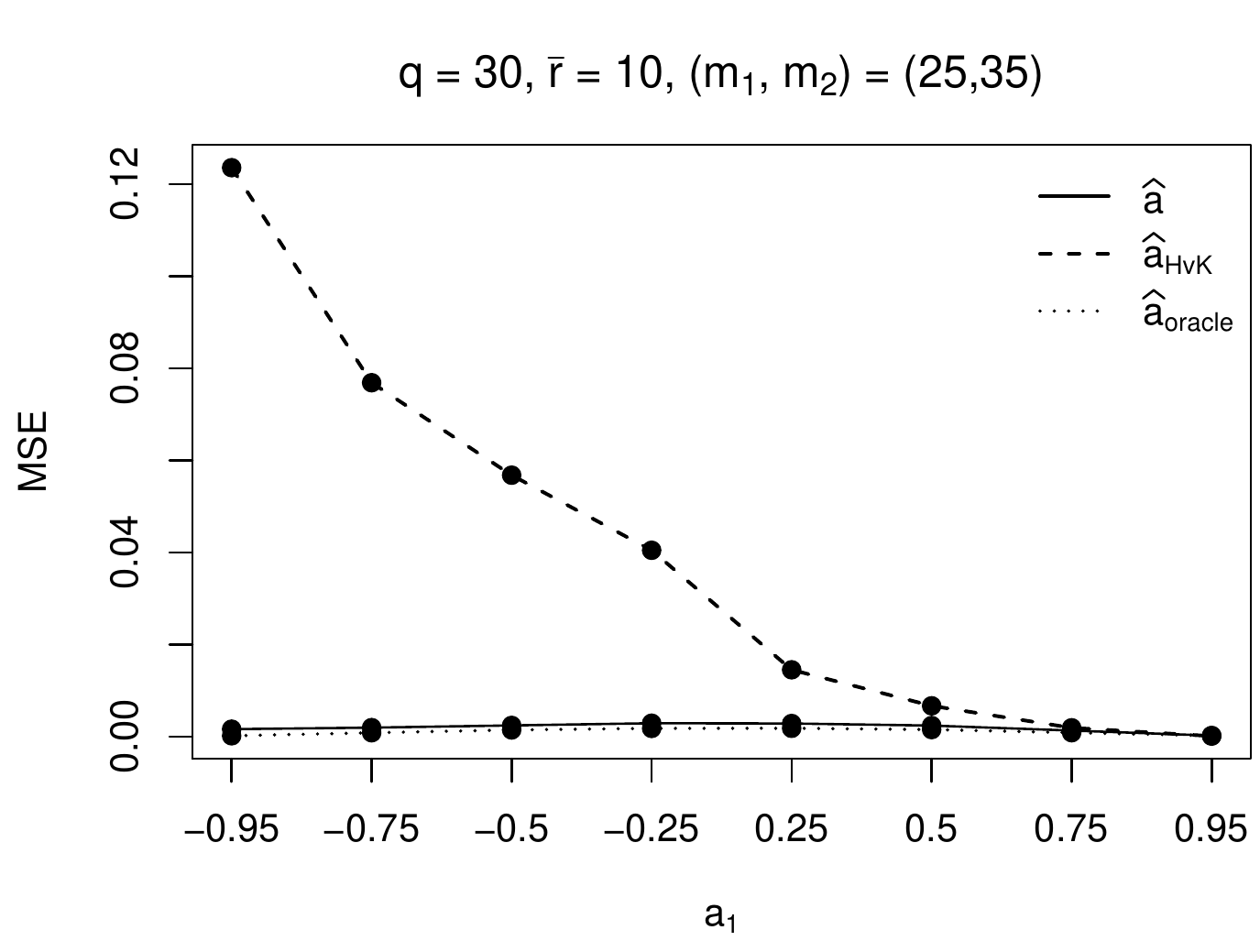}
\end{subfigure}
\hspace{0.25cm}
\begin{subfigure}[b]{0.45\textwidth}
\includegraphics[width=\textwidth]{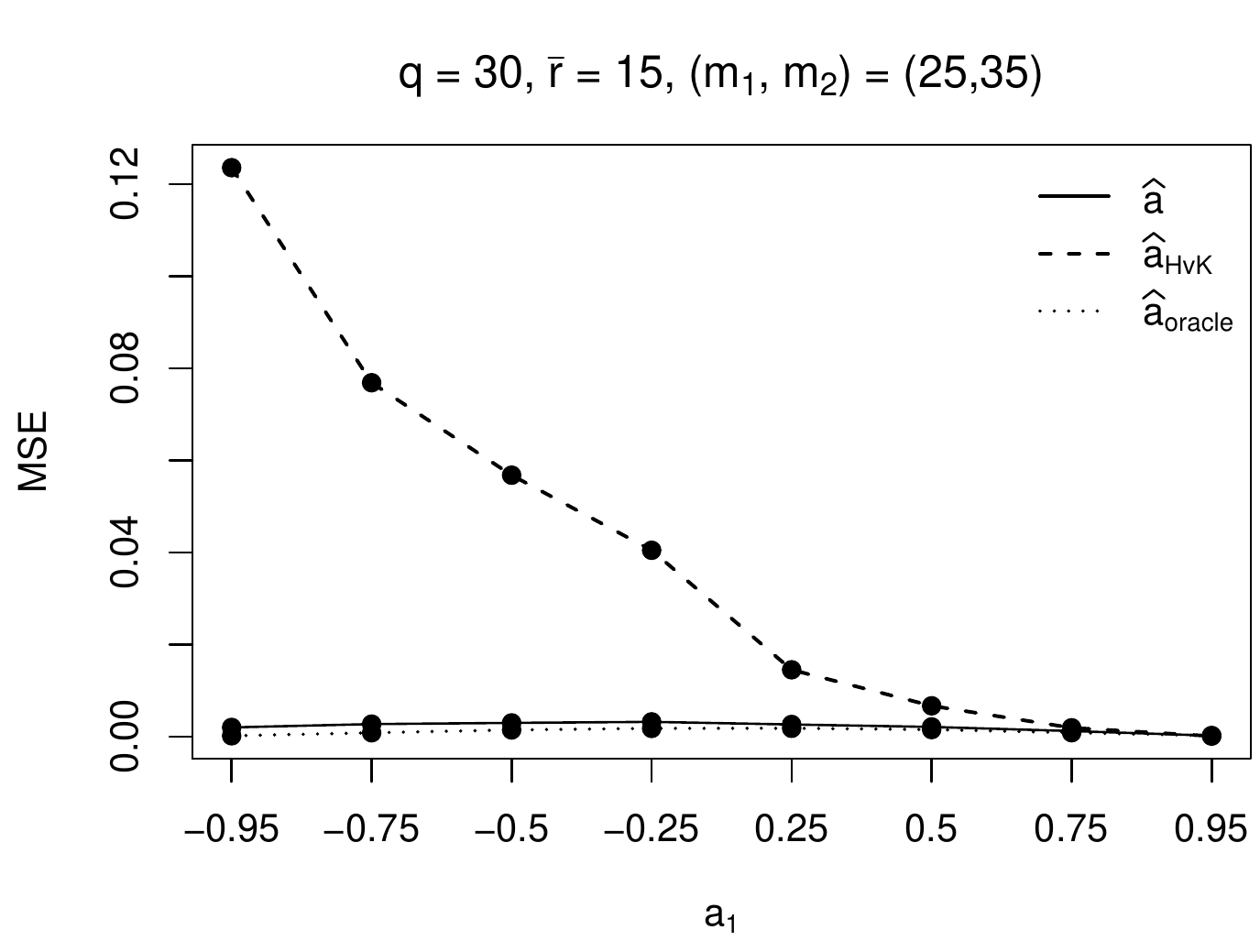}
\end{subfigure}

\begin{subfigure}[b]{0.45\textwidth}
\includegraphics[width=\textwidth]{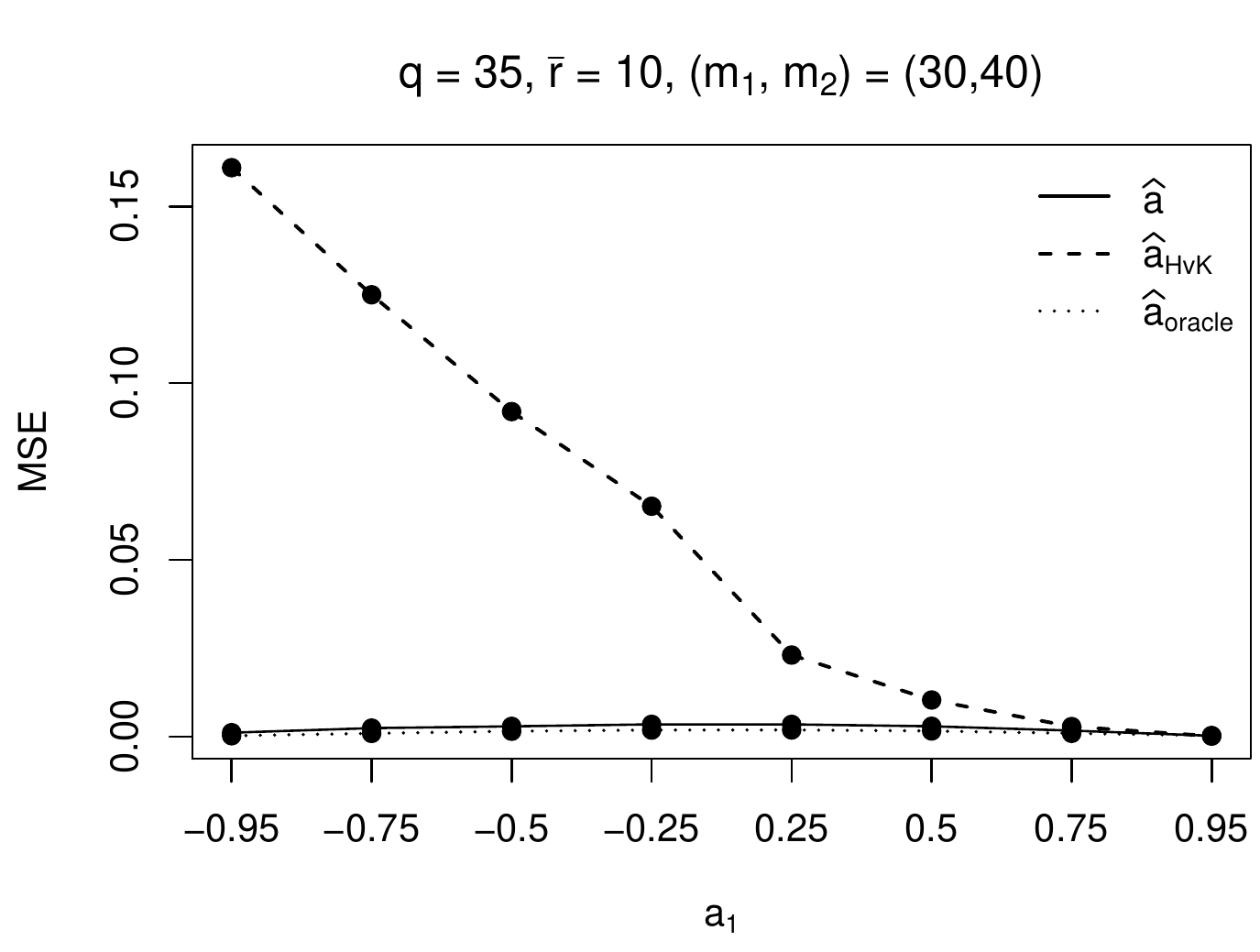}
\end{subfigure}
\hspace{0.25cm}
\begin{subfigure}[b]{0.45\textwidth}
\includegraphics[width=\textwidth]{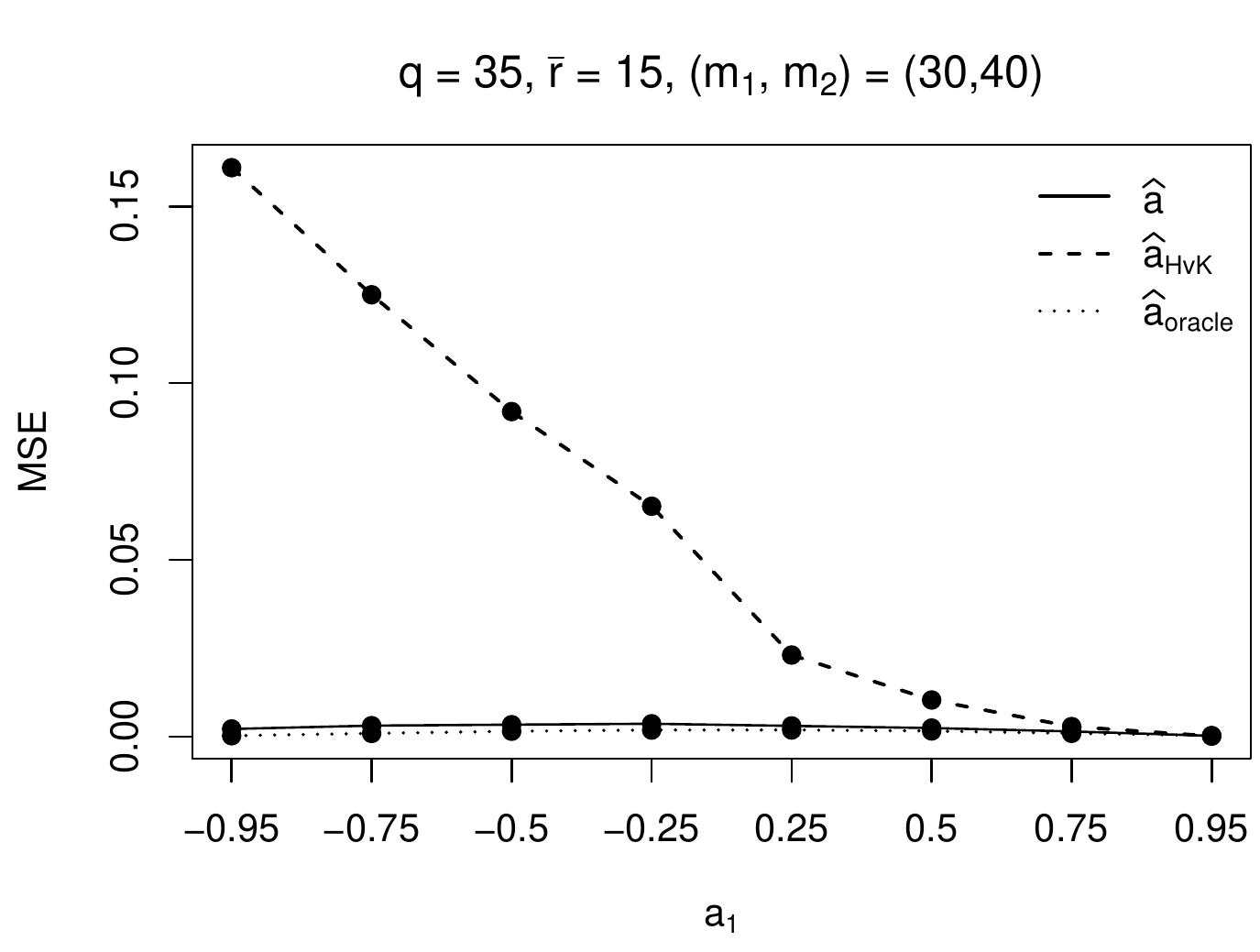}
\end{subfigure}
\caption{MSE values for the estimators $\widehat{a}$, $\widehat{a}_{\text{HvK}}$ and $\widehat{a}_{\text{oracle}}$ in the scenario with a pronounced trend ($s_\beta=10$).}\label{fig:MSE_slope10_AR_robust} 
\end{figure}

\begin{figure}[p]
\begin{subfigure}[b]{0.45\textwidth}
\includegraphics[width=\textwidth]{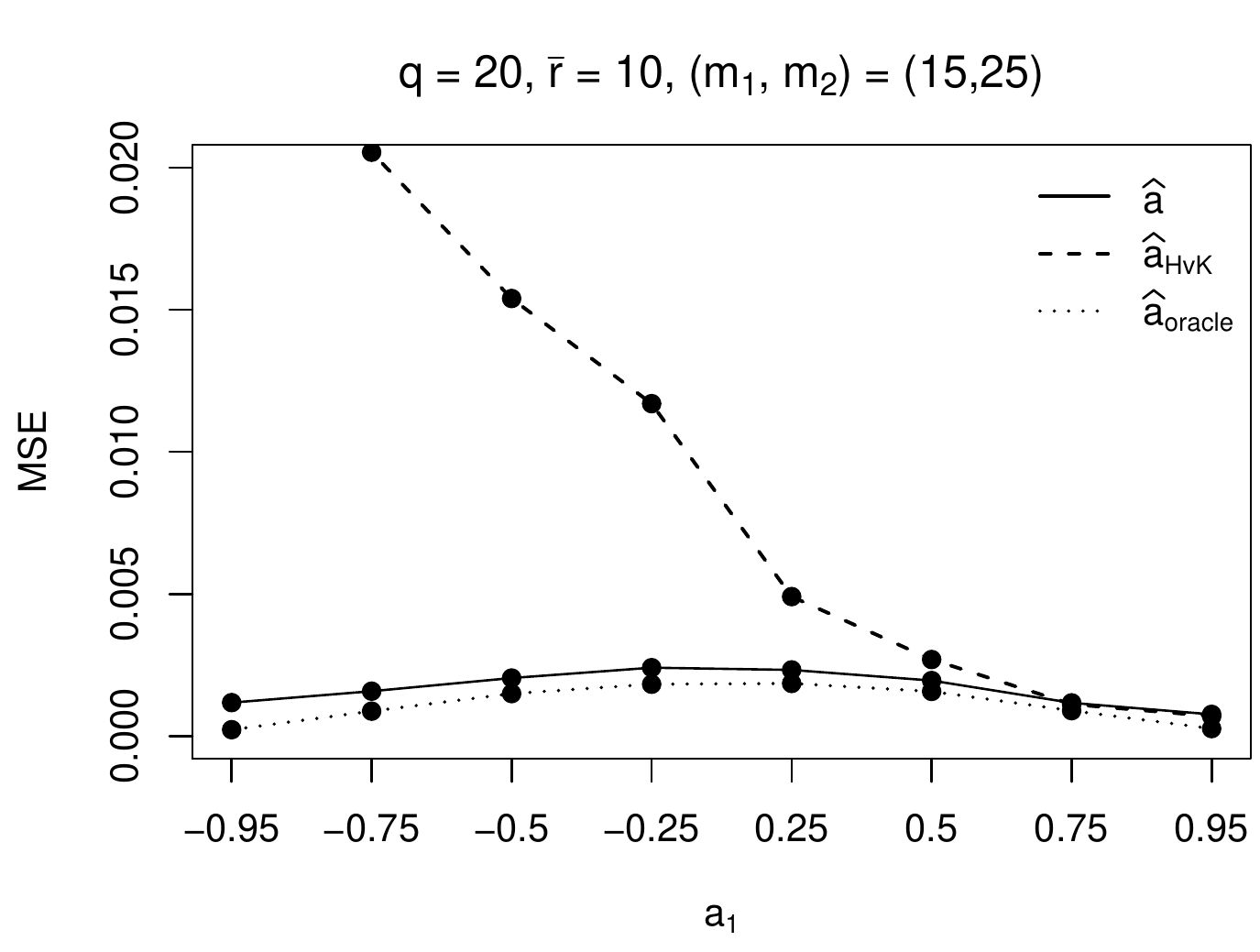}
\end{subfigure}
\hspace{0.25cm}
\begin{subfigure}[b]{0.45\textwidth}
\includegraphics[width=\textwidth]{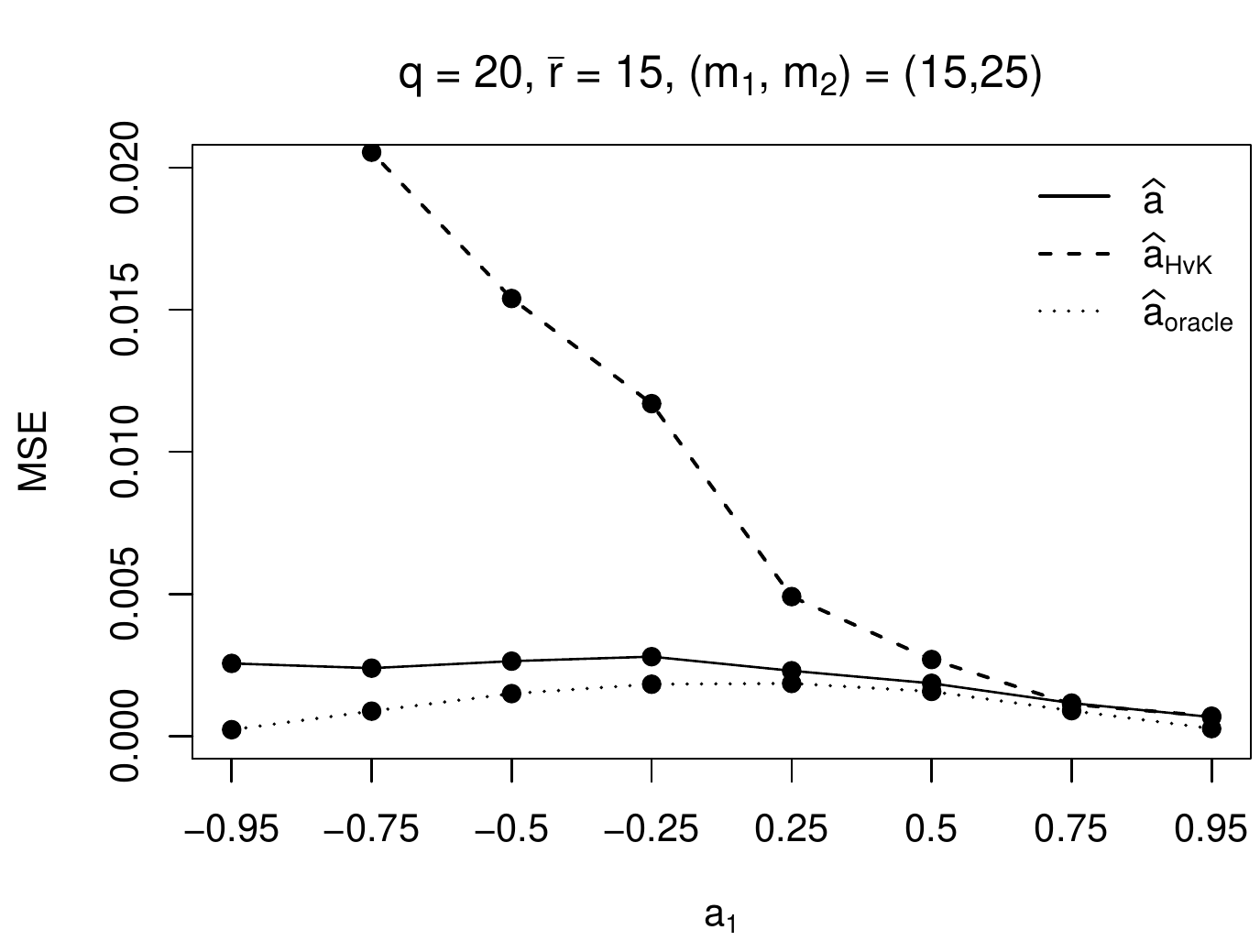}
\end{subfigure}

\begin{subfigure}[b]{0.45\textwidth}
\includegraphics[width=\textwidth]{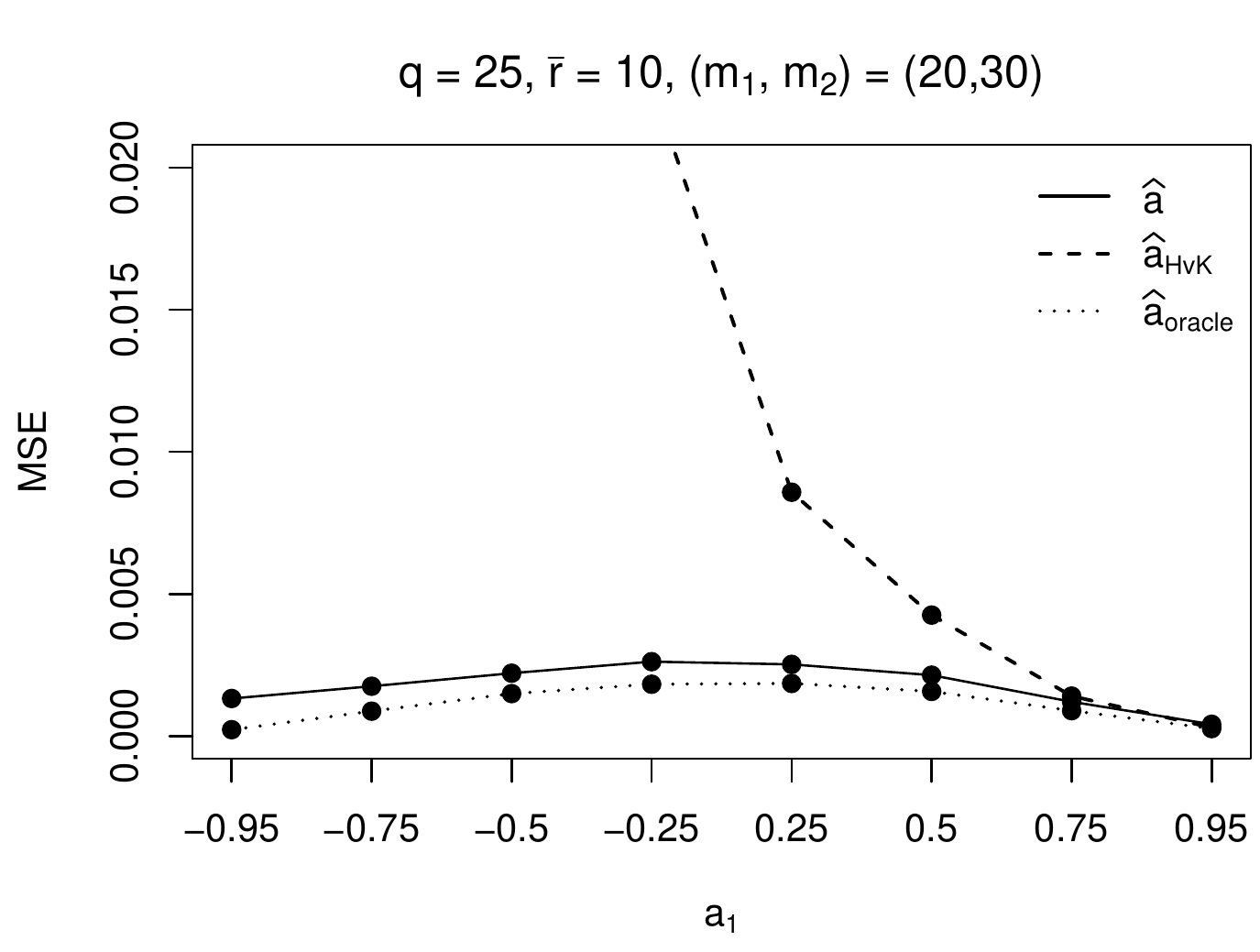}
\end{subfigure}
\hspace{0.25cm}
\begin{subfigure}[b]{0.45\textwidth}
\includegraphics[width=\textwidth]{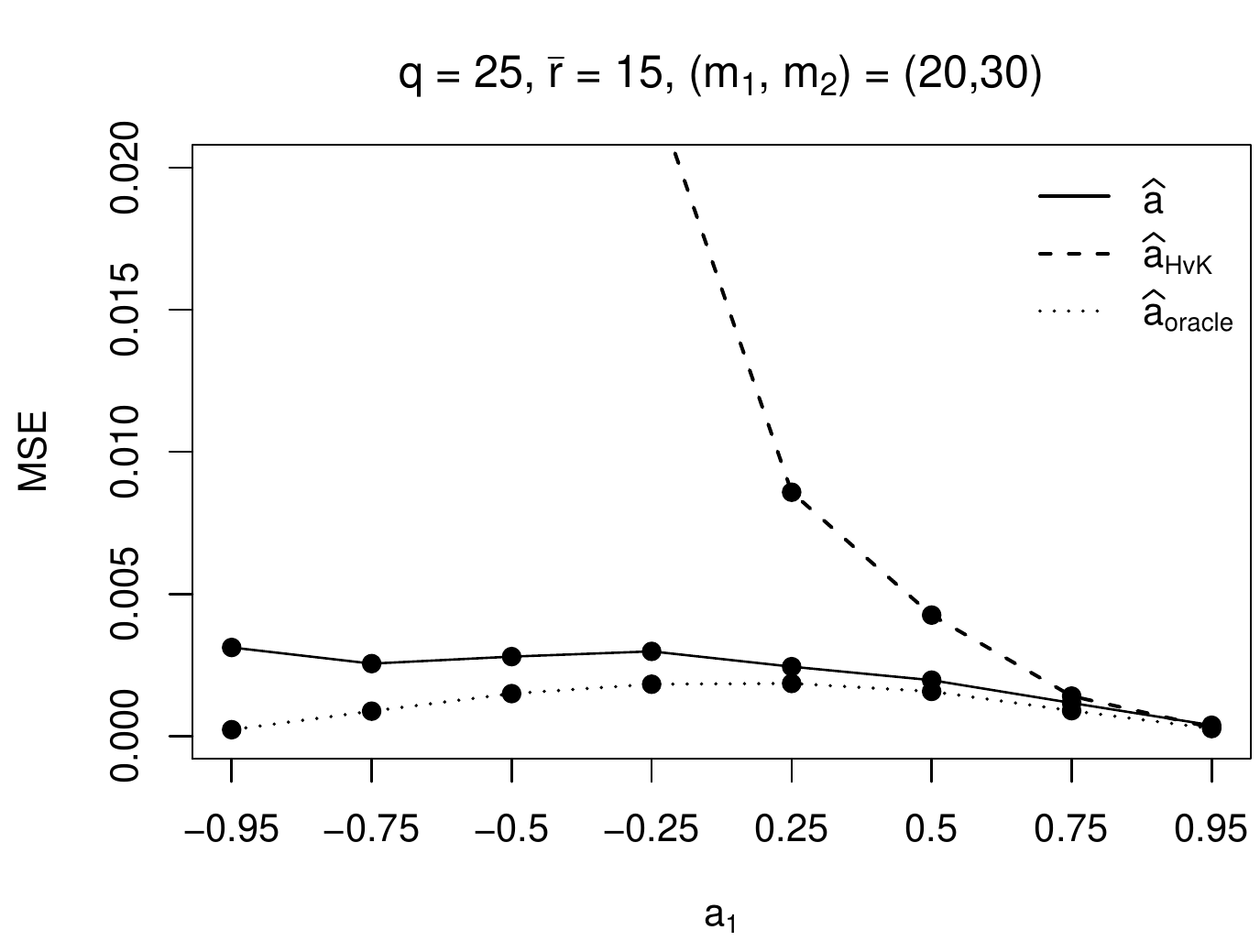}
\end{subfigure}

\begin{subfigure}[b]{0.45\textwidth}
\includegraphics[width=\textwidth]{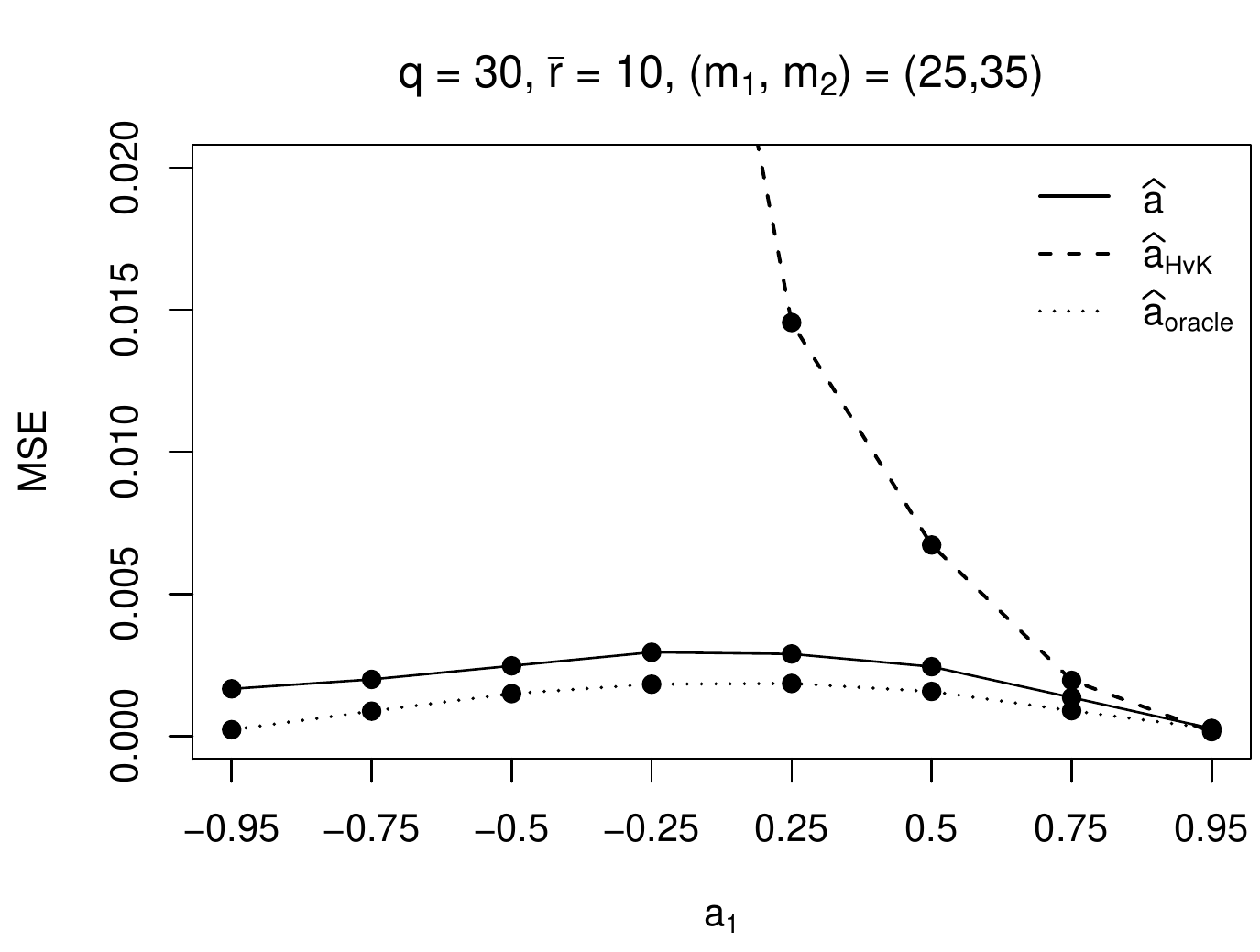}
\end{subfigure}
\hspace{0.25cm}
\begin{subfigure}[b]{0.45\textwidth}
\includegraphics[width=\textwidth]{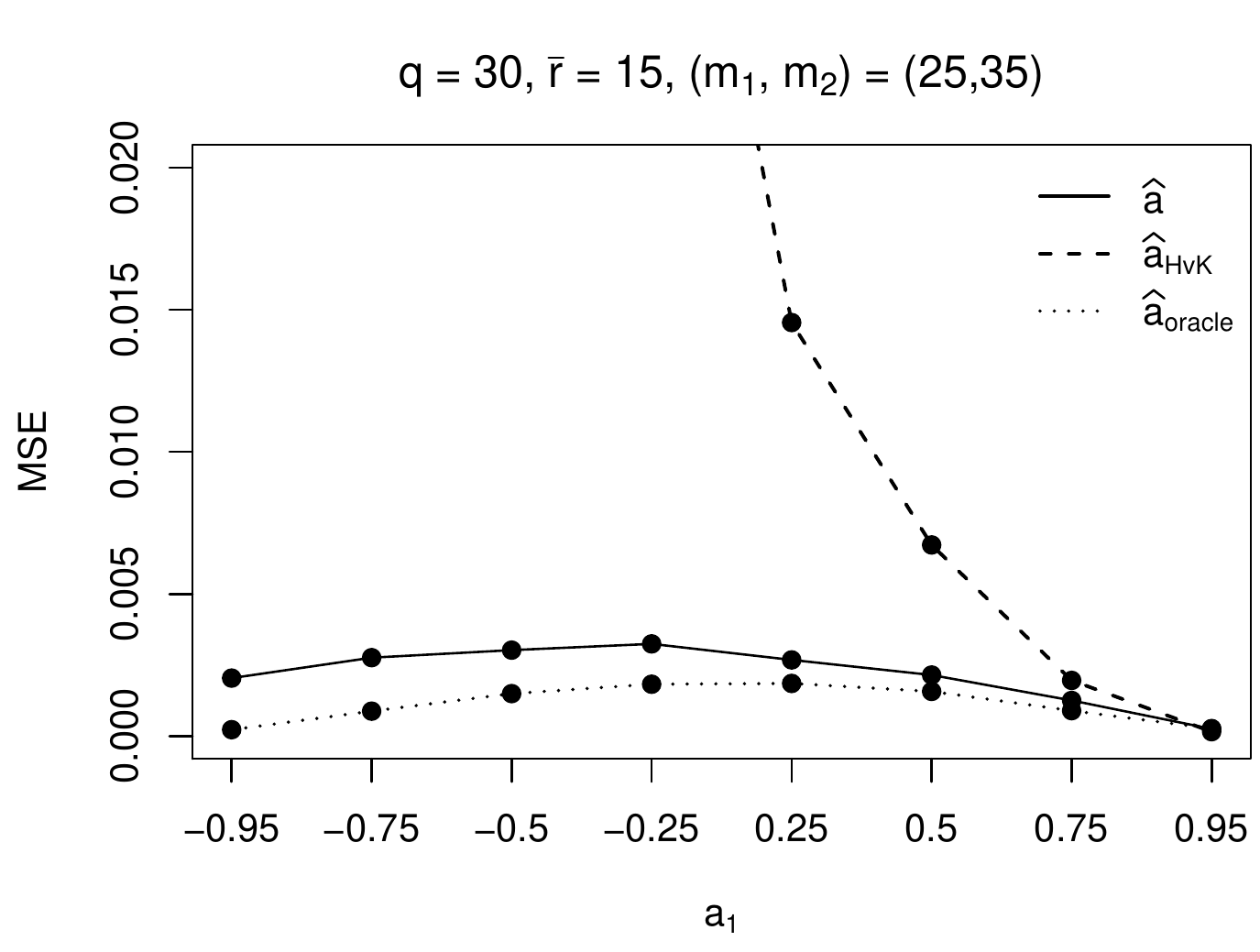}
\end{subfigure}

\begin{subfigure}[b]{0.45\textwidth}
\includegraphics[width=\textwidth]{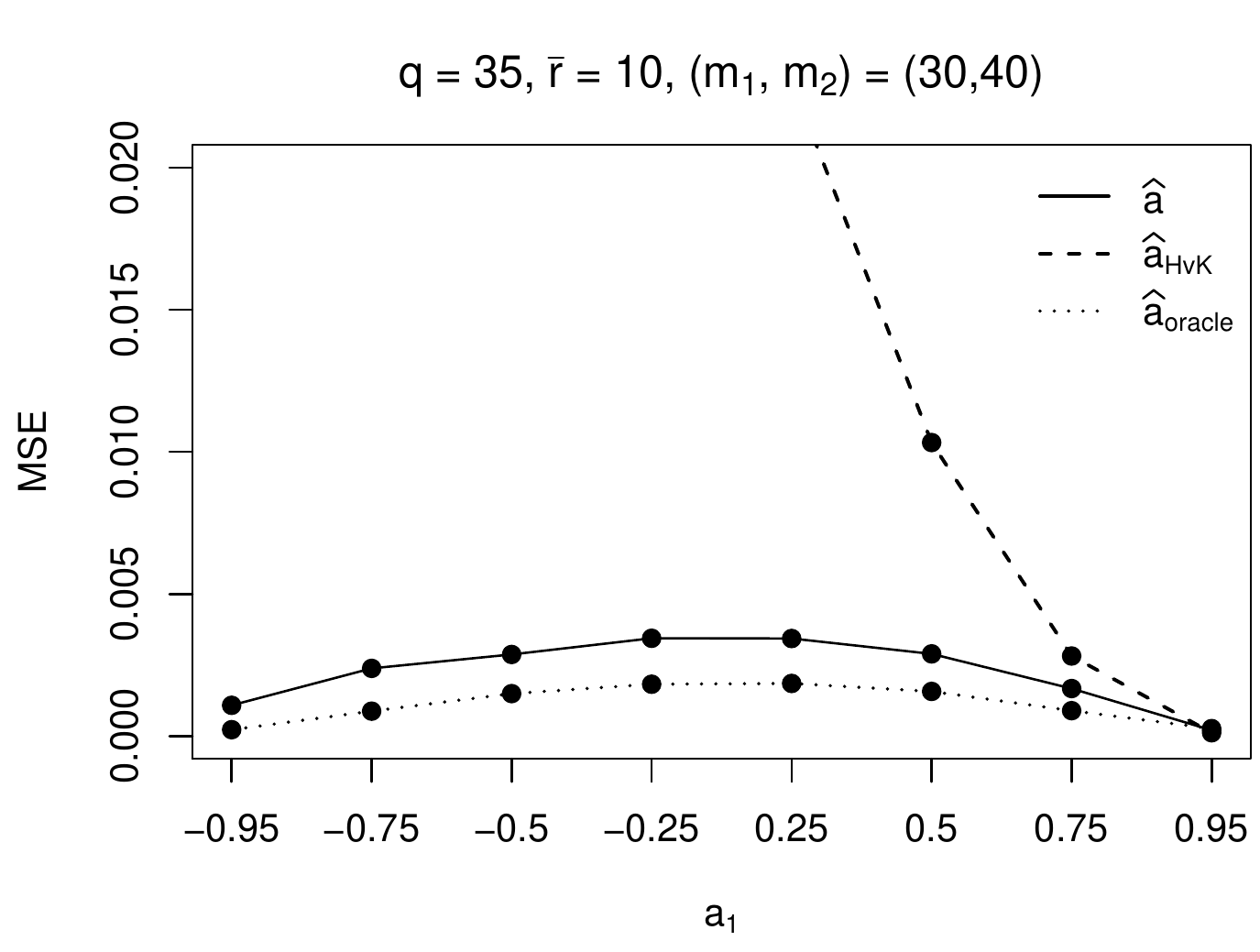}
\end{subfigure}
\hspace{0.25cm}
\begin{subfigure}[b]{0.45\textwidth}
\includegraphics[width=\textwidth]{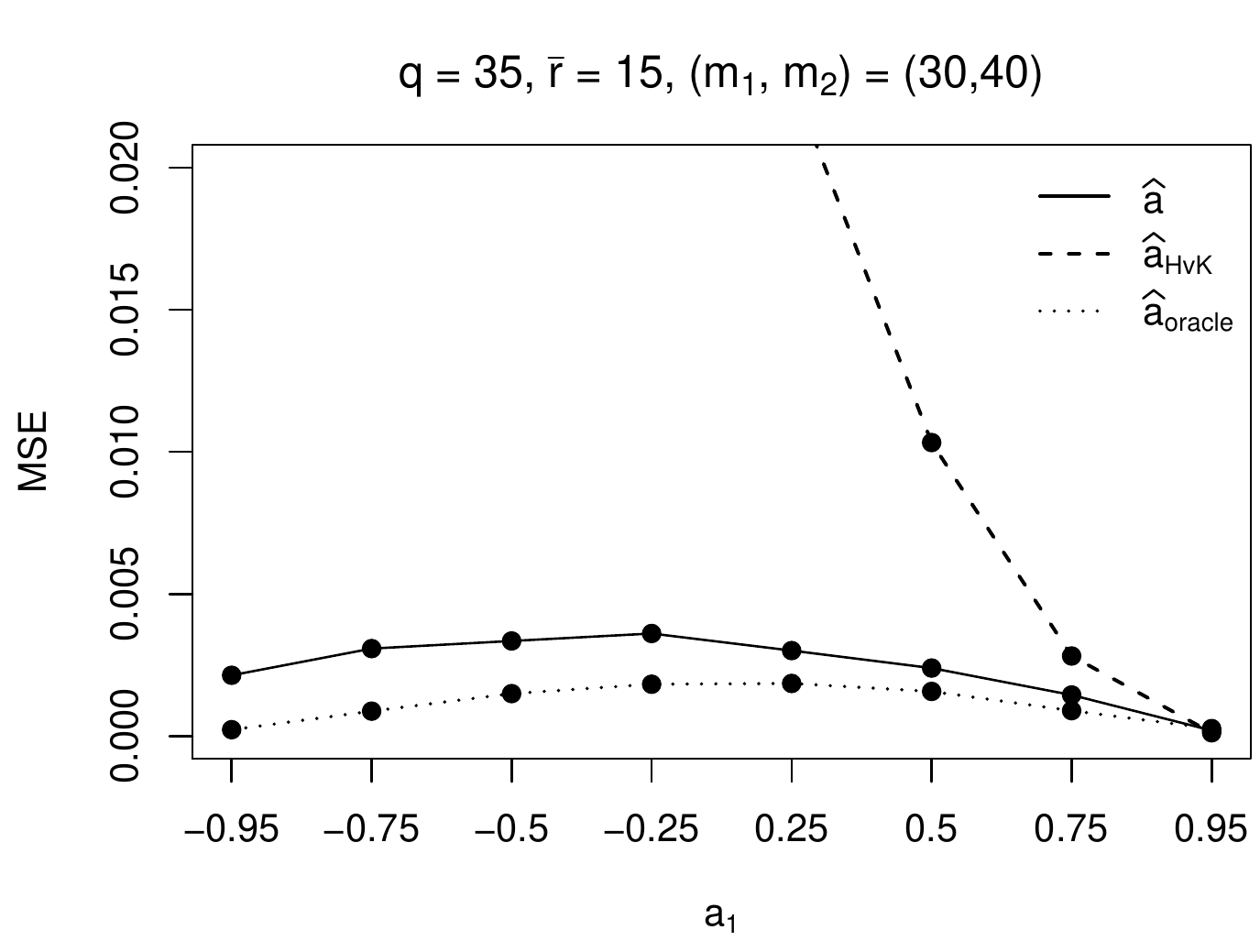}
\end{subfigure}
\caption{MSE values for the estimators $\widehat{a}$, $\widehat{a}_{\text{HvK}}$ and $\widehat{a}_{\text{oracle}}$ in the scenario with a pronounced trend ($s_\beta=10$). The plots are zoomed-in versions of the respective plots in Figure \ref{fig:MSE_slope10_AR_robust}.}\label{fig:MSE_slope10_AR_zoom_robust}
\end{figure}

\begin{figure}[p]
\begin{subfigure}[b]{0.45\textwidth}
\includegraphics[width=\textwidth]{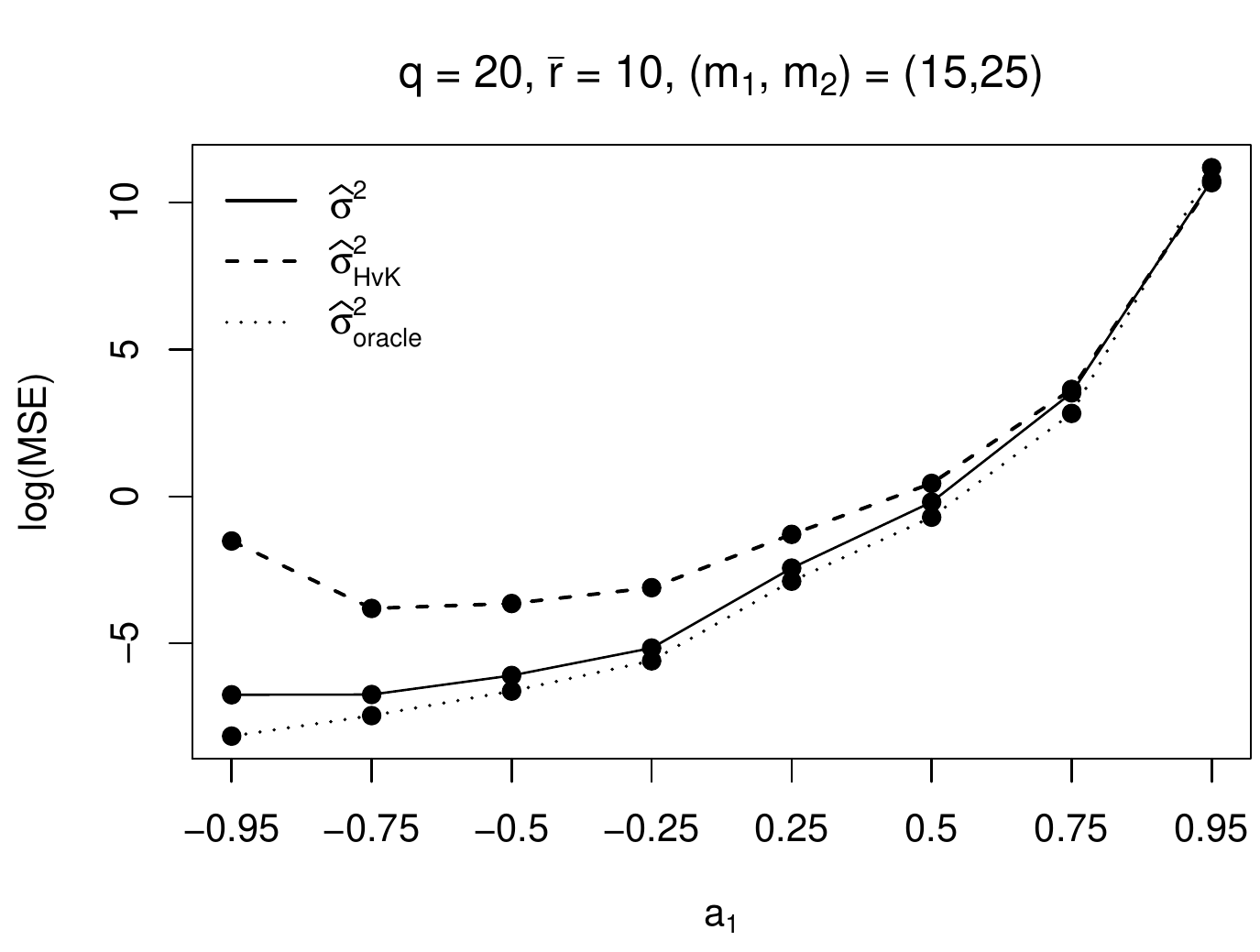}
\end{subfigure}
\hspace{0.25cm}
\begin{subfigure}[b]{0.45\textwidth}
\includegraphics[width=\textwidth]{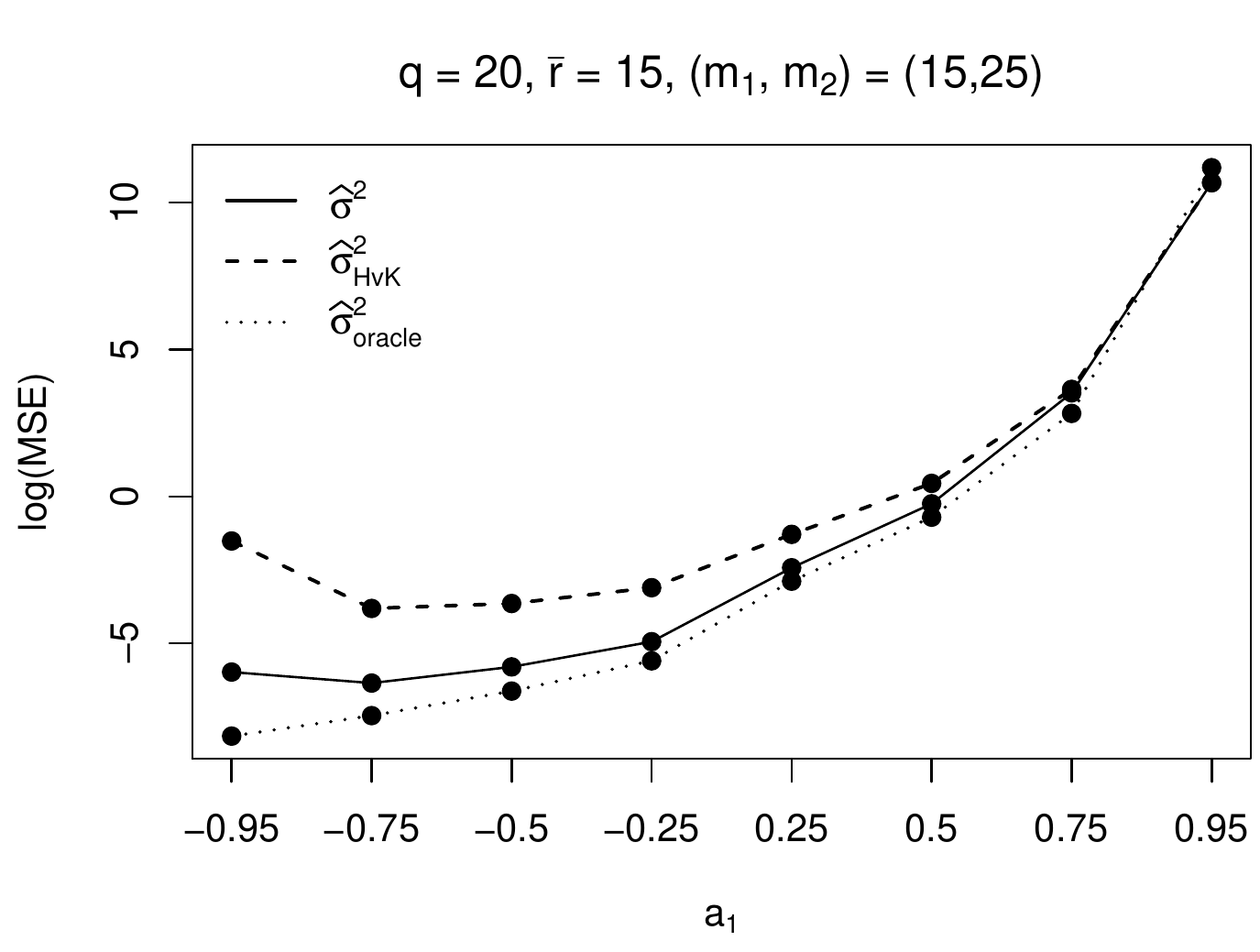}
\end{subfigure}

\begin{subfigure}[b]{0.45\textwidth}
\includegraphics[width=\textwidth]{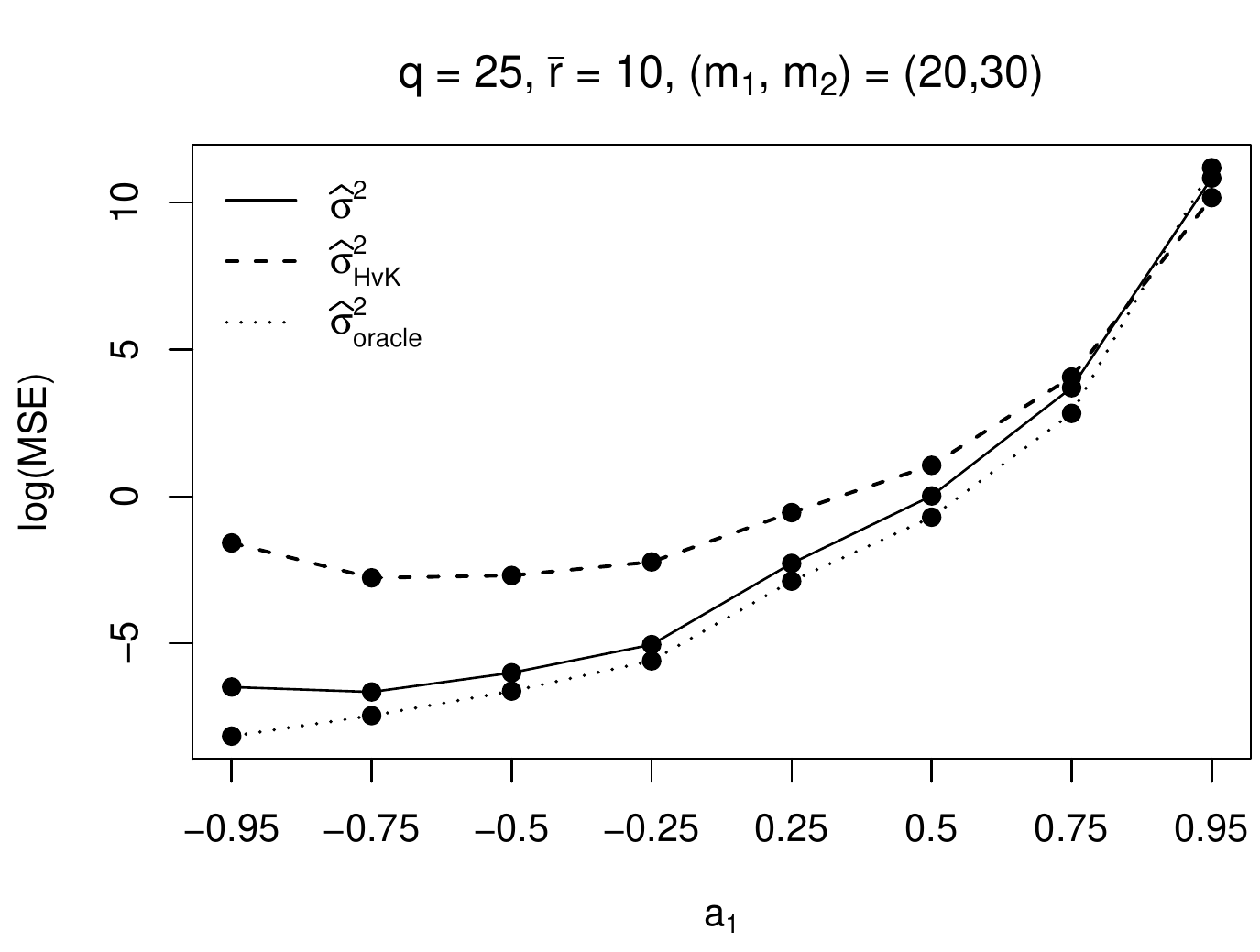}
\end{subfigure}
\hspace{0.25cm}
\begin{subfigure}[b]{0.45\textwidth}
\includegraphics[width=\textwidth]{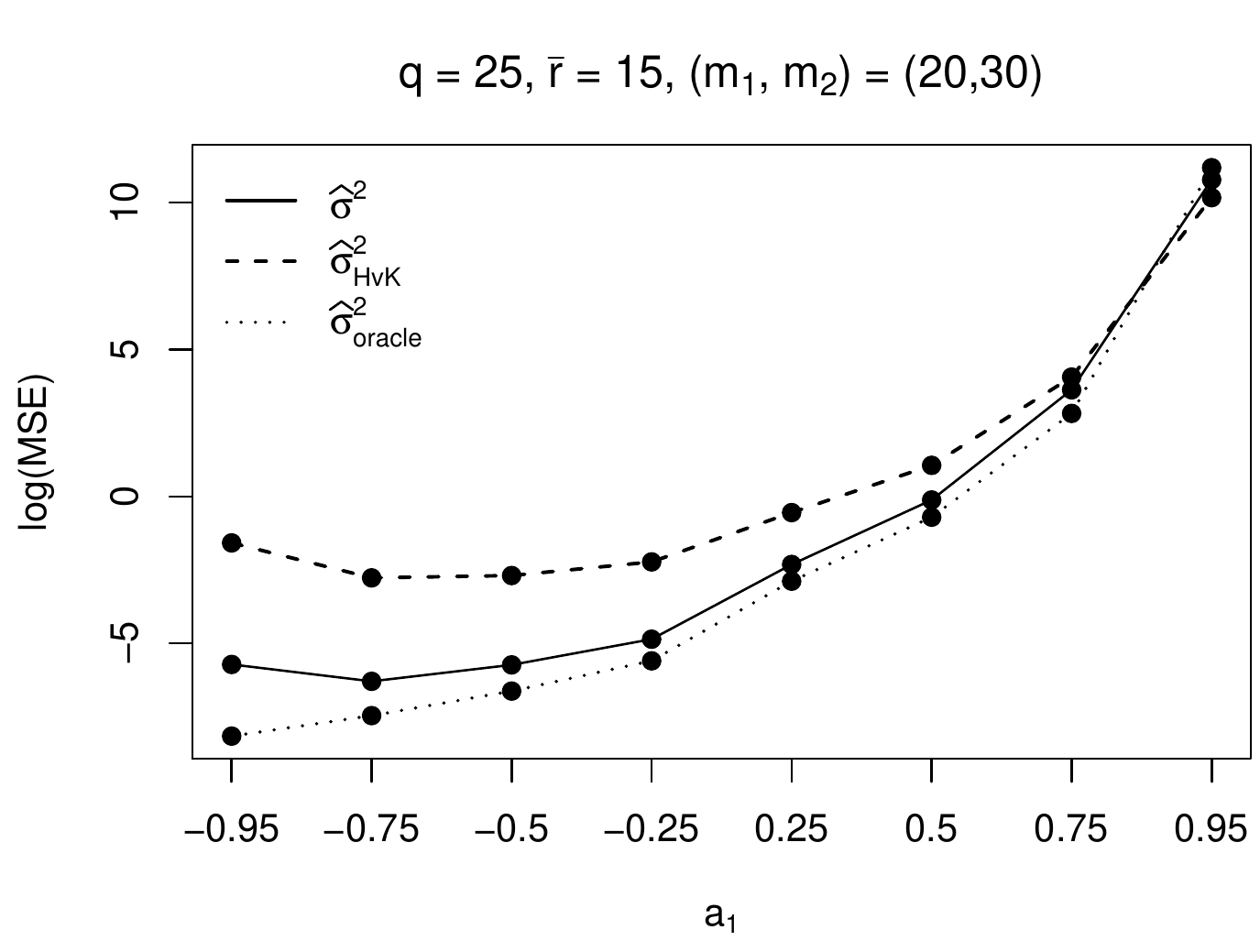}
\end{subfigure}

\begin{subfigure}[b]{0.45\textwidth}
\includegraphics[width=\textwidth]{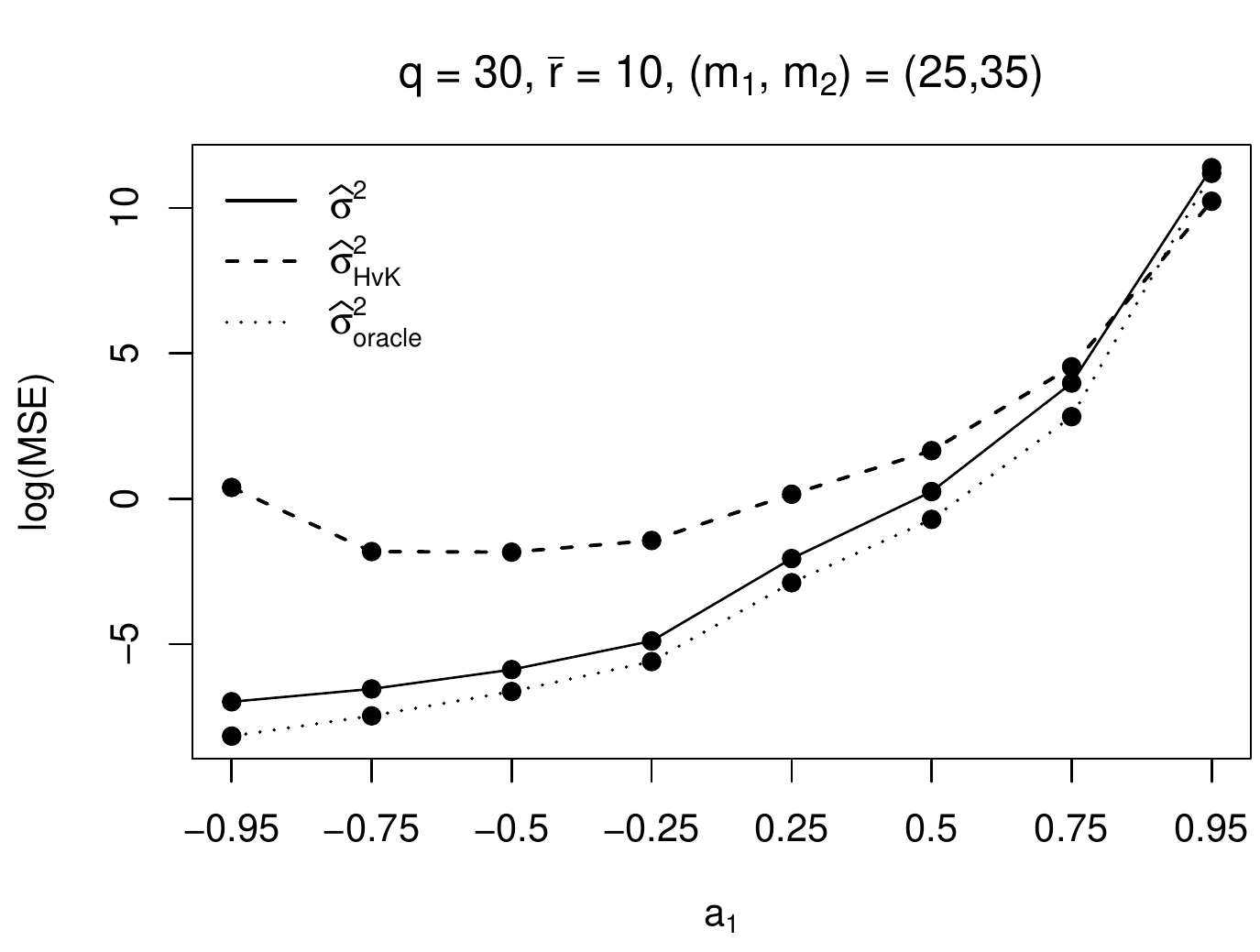}
\end{subfigure}
\hspace{0.25cm}
\begin{subfigure}[b]{0.45\textwidth}
\includegraphics[width=\textwidth]{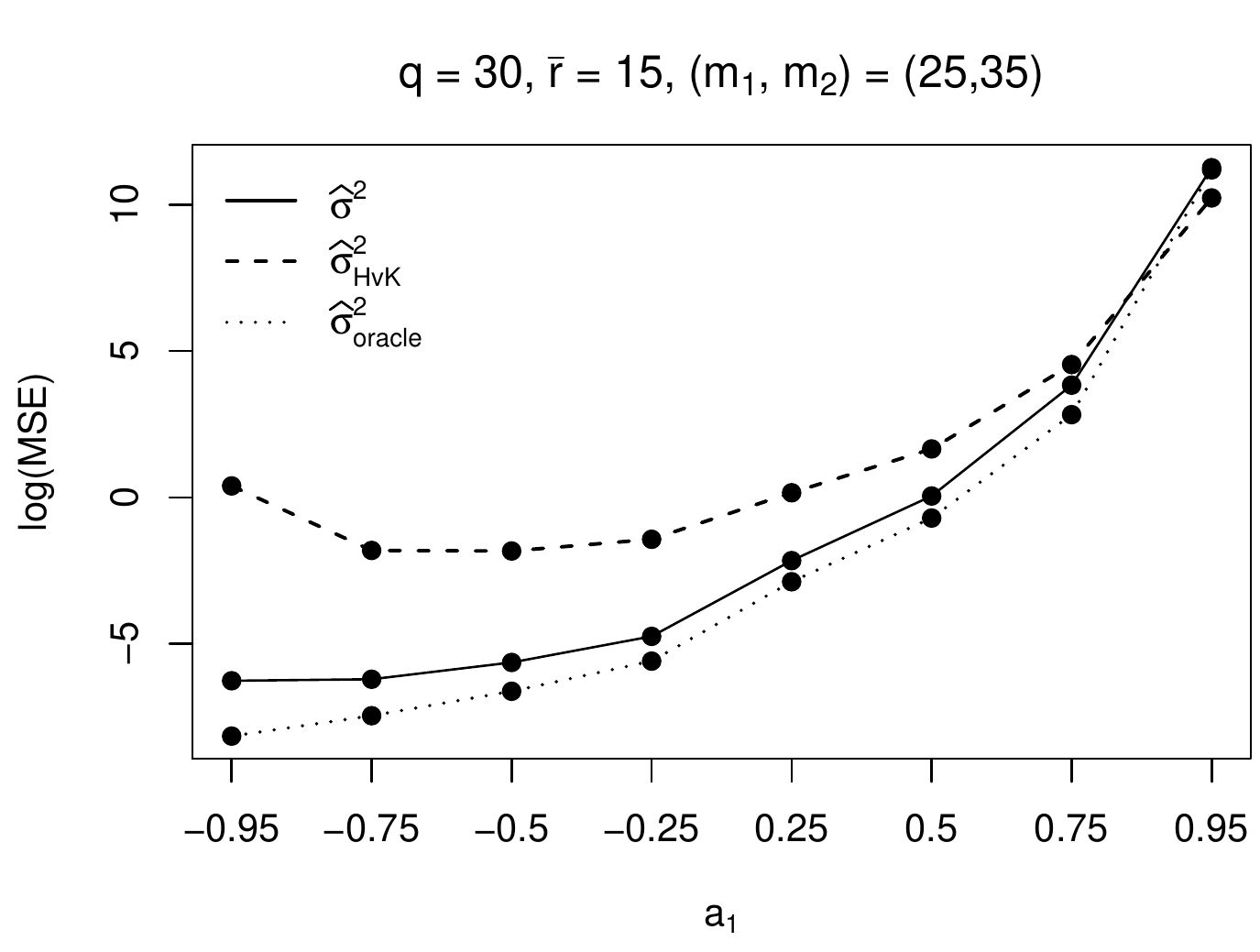}
\end{subfigure}

\begin{subfigure}[b]{0.45\textwidth}
\includegraphics[width=\textwidth]{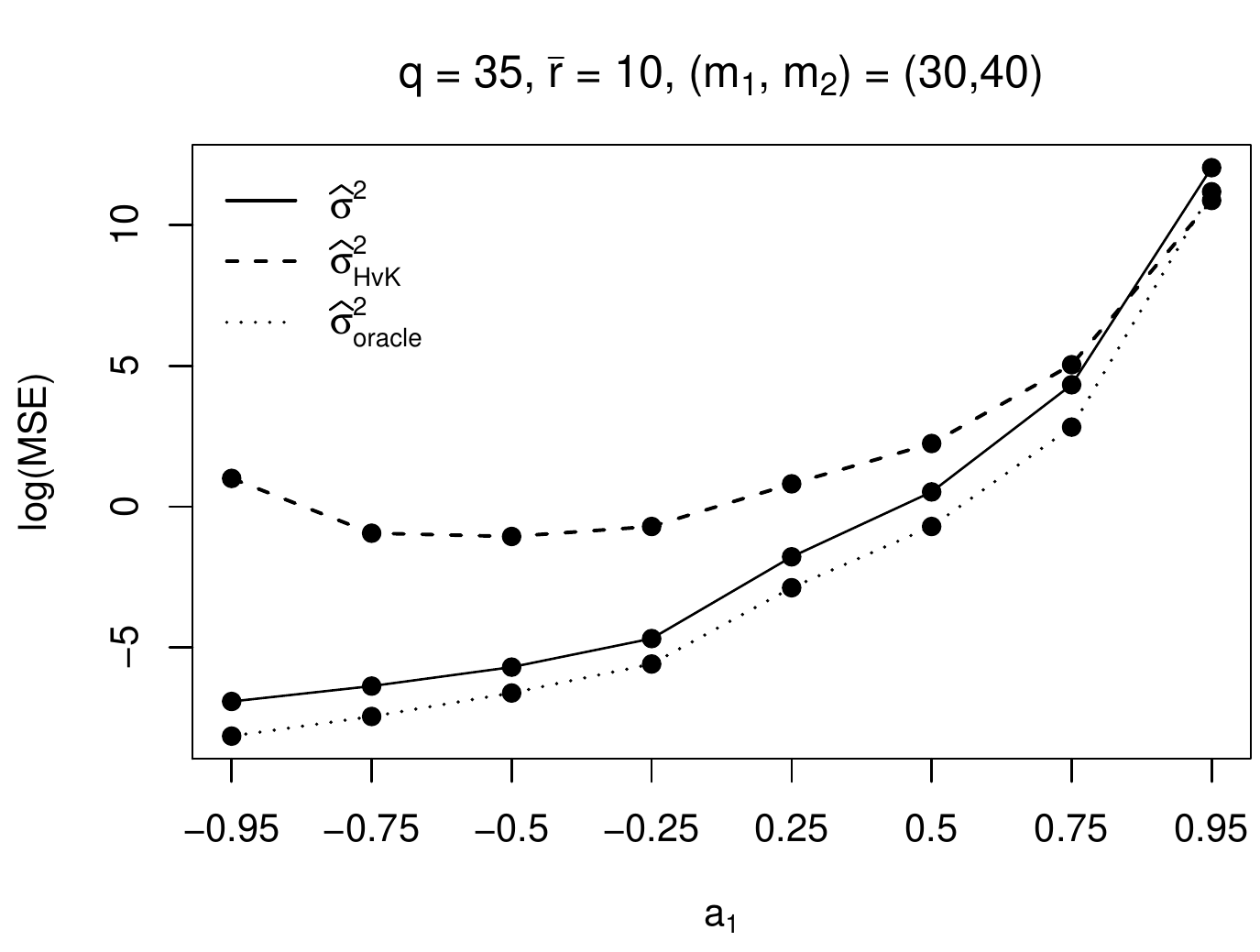}
\end{subfigure}
\hspace{0.25cm}
\begin{subfigure}[b]{0.45\textwidth}
\includegraphics[width=\textwidth]{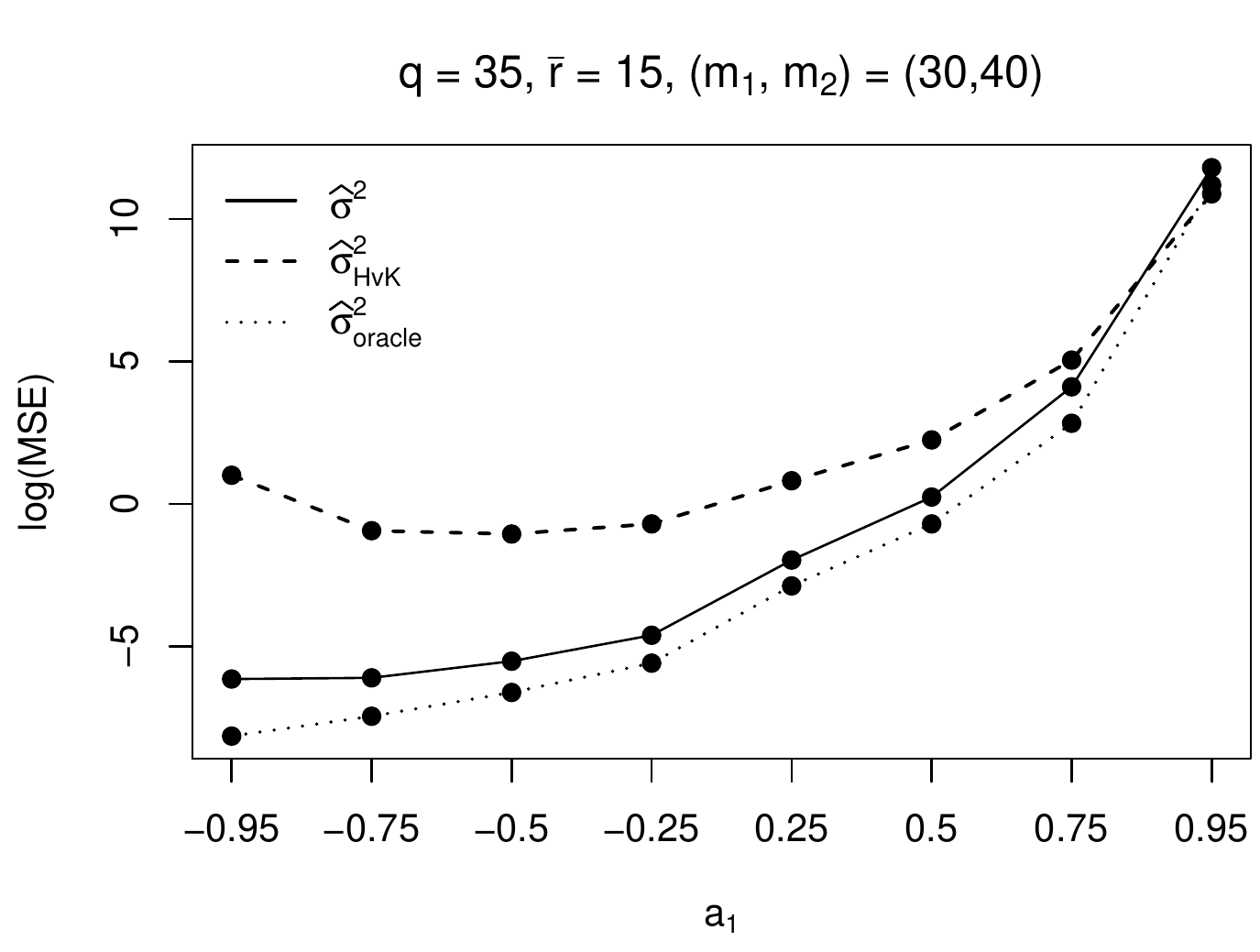}
\end{subfigure}
\caption{Logarithmic MSE values for the estimators $\widehat{\sigma}^2$, $\widehat{\sigma}^2_{\text{HvK}}$ and $\widehat{\sigma}^2_{\text{oracle}}$ in the scenario with a pronounced trend ($s_\beta=10$).}\label{fig:MSE_slope10_lrv_robust} 
\end{figure}

\FloatBarrier
\newpage
\subsection*{Implementation of SiZer in Section \ref{subsec-sim-2}}

The SiZer methods for the comparison study in Section \ref{subsec-sim-2} are implemented as follows:
\begin{enumerate}[leftmargin=0.7cm,label=(\alph*)]

\item Computation of the grid $\mathcal{G}_T^*$:

To start with, we compute the variance of $\bar{Y} = T^{-1} \sum_{t=1}^T Y_{t,T}$, which is given by
\[ \var(\bar{Y}) = \frac{\gamma_\varepsilon(0)}{T} + \frac{2}{T} \sum_{k = 1}^{T-1} \Big(1 - \frac{k}{T}\Big)\gamma_\varepsilon(k). \]
Since the autocovariance function $\gamma_{\varepsilon}(\cdot)$ is known by assumption, we can calculate the value of $\var(\bar{Y})$ by using the formula $\gamma_\varepsilon(k) = \nu^2 a_1^{|k|} / (1 - a_1^2)$ together with the true para\-meters $a_1$ and $\nu^2 = \ex[\eta_t^2]$. We next compute 
\[ T^* = \frac{\gamma_\varepsilon(0)}{\var(\bar{Y})}, \]
which can be interpreted as a measure of information in the data. For each point $(u,h) \in \mathcal{G}_T$ from \eqref{grid-sim-app}, we finally calculate the effective sample size for dependent data 
\[ \text{ESS}^*(u, h) = \frac{T^*}{T} \frac{\sum_{t=1}^T K_h(t/T - u)}{K_h(0)} \]
with $K_h(v) = h^{-1} K(v/h)$ and set $\mathcal{G}_T^* = \{ (u,h) \in \mathcal{G}_T: \text{ESS}^*(u, h) \ge 5 \}$. 

\item Computation of the local linear estimators and their standard deviations: 

For each $(u,h) \in \mathcal{G}_T^*$, we compute a standard local linear estimator $\widehat{m}^\prime_h(u)$ of the derivative $m^\prime(u)$ together with its standard deviation $\text{sd}(\widehat{m}^\prime_h(u))$. The latter is given by $\text{sd}(\widehat{m}^\prime_h(u)) = \{ \var(\widehat{m}^\prime_h(u)) \}^{1/2}$, where $\var(\widehat{m}^\prime_h(u)) = e^\top V e$ with $e = (\begin{matrix} 0 & 1 \end{matrix})^\top$ and 
\[ V = (X^T W X)^{-1} (X^T \Sigma X) (X^T W X)^{-1}. \]
The matrices $X$, $W$ and $\Sigma$ are defined as follows: $\Sigma$ is a $T \times T$ matrix with the elements
\[ \Sigma_{st} = \gamma_\varepsilon(s-t) K_h\Big( \frac{s}{T} - u \Big) K_h\Big( \frac{t}{T} - u \Big), \]
$W$ is a $T \times T$ diagonal matrix with the diagonal entries $K_h(t/T-u)$ and 
\[ X =   
\begin{pmatrix}
1 & (1/T - u)   \\
1 & (2/T - u)   \\
\vdots & \vdots \\
1 & (1 - u)     \\
\end{pmatrix}. \]

\newpage
\item Computation of the confidence intervals:  

For a given confidence level $\alpha$ and for each bandwidth value $h$ with $(u,h) \in \mathcal{G}_T^*$, we compute the quantile 
\[ q(h) = \Phi^{-1} \Big( \Big( 1 - \frac{\alpha}{2} \Big)^{1/(\theta g)} \Big), \]
where $\Phi$ is the distribution function of a standard normal random variable, $g$ is the number of locations $u$ in the grid $\mathcal{G}_T$, and the cluster index $\theta$ is defined on p.1519 in \cite{ParkHannigKang2009}. The confidence interval of $\widehat{m}^\prime_h(u)$ is then computed as $[\widehat{m}^\prime_h(u) - q(h) \, \text{sd}(\widehat{m}^\prime_h(u)),\widehat{m}^\prime_h(u) + q(h) \, \text{sd}(\widehat{m}^\prime_h(u))]$. 

\end{enumerate}

\end{document}